\documentclass[11pt]{article}
\usepackage{etex}
\usepackage[margin={1.2in}]{geometry}
\usepackage[titletoc]{appendix}
\usepackage{enumitem}
\usepackage{tikz}
\usetikzlibrary{matrix}
\usepackage{tikz-cd}
\usetikzlibrary{shapes, arrows}
\usetikzlibrary{positioning}
\usetikzlibrary{decorations.markings,patterns,bending}
\usepackage{eucal}
\usepackage{mathrsfs}
\usepackage{stmaryrd}
\usepackage{mathtools}
 \usepackage{float}

\DeclarePairedDelimiter\abs{\lvert}{\rvert}%

\usepackage{rotating}

\usepackage{longtable}

\usepackage[sc]{mathpazo}
\usepackage{courier}

\usepackage[utf8]{inputenc}
\usepackage[T1]{fontenc}

\usepackage[english]{babel}
\usepackage{blindtext}

\usepackage{amsmath}

\usepackage{microtype}

\usepackage{amsmath}
\usepackage{amssymb}
\usepackage{amsthm}
\usepackage{color}
\usepackage{latexsym,extarrows}

\usepackage[all]{xy}

\usepackage[stable,multiple]{footmisc}

\RequirePackage[font=small,format=plain,labelfont=bf,textfont=it]{caption}
\makeatletter
\newsavebox{\@brx}
\newcommand{\llangle}[1][]{\savebox{\@brx}{\(\m@th{#1\langle}\)}%
  \mathopen{\copy\@brx\kern-0.5\wd\@brx\usebox{\@brx}}}
\newcommand{\rrangle}[1][]{\savebox{\@brx}{\(\m@th{#1\rangle}\)}%
  \mathclose{\copy\@brx\kern-0.5\wd\@brx\usebox{\@brx}}}
\makeatother

\begin{document}
\def\e#1\e{\begin{equation}#1\end{equation}}
\def\ea#1\ea{\begin{align}#1\end{align}}
\def\eq#1{{\rm(\ref{#1})}}
\theoremstyle{plain}
\newtheorem{thm}{Theorem}[section]
\newtheorem{lem}[thm]{Lemma}
\newtheorem{prop}[thm]{Proposition}
\newtheorem{cor}[thm]{Corollary}
\theoremstyle{definition}
\newtheorem{dfn}[thm]{Definition}
\newtheorem{ex}[thm]{Example}
\newtheorem{rem}[thm]{Remark}
\newtheorem{conjecture}[thm]{Conjecture}
\newtheorem{convention}[thm]{Convention}
\newtheorem{assumption}[thm]{Assumption}
\newtheorem{notation}[thm]{Notation}

\newcommand{\D}{\mathrm{d}}
\newcommand{\A}{\mathcal{A}}
\newcommand{\LL}{\llangle[\Big]}
\newcommand{\RR}{\rrangle[\Big]}
\newcommand{\LD}{\Big\langle}
\newcommand{\RD}{\Big\rangle}
\newcommand{\F}{\mathcal{F}}
\newcommand{\HH}{\mathcal{H}}
\newcommand{\X}{\mathcal{X}}

\newcommand{\K}{\mathscr{K}}
\newcommand{\q}{\mathbf{q}}

\newcommand{\op}{\operatorname}
\newcommand{\C}{\mathbb{C}}
\newcommand{\N}{\mathbb{N}}
\newcommand{\R}{\mathbb{R}}
\newcommand{\Q}{\mathbb{Q}}
\newcommand{\Z}{\mathbb{Z}}
\renewcommand{\H}{\mathbf{H}}
\newcommand{\PP}{\mathbb{P}}

\newcommand{\Etau}{\text{E}_\tau}
\newcommand{\E}{{\mathcal E}}
\newcommand{\G}{\mathbf{G}}
\newcommand{\eps}{\epsilon}

\newcommand{\im}{\op{im}}

\newcommand{\h}{\mathbf{h}}

\newcommand{\Gmax}[1]{G_{#1}}
\newcommand{\AW}{E}
\newcommand{\abracket}[1]{\left\langle#1\right\rangle}
\newcommand{\bbracket}[1]{\left[#1\right]}
\newcommand{\fbracket}[1]{\left\{#1\right\}}
\newcommand{\bracket}[1]{\left(#1\right)}
\newcommand{\ket}[1]{|#1\rangle}
\newcommand{\bra}[1]{\langle#1|}

\newcommand{\ora}[1]{\overrightarrow#1}

\providecommand{\from}{\leftarrow}
\providecommand{\to}{\rightarrow}
\newcommand{\bl}{\textbf}
\newcommand{\mbf}{\mathbf}
\newcommand{\mbb}{\mathbb}
\newcommand{\mf}{\mathfrak}
\newcommand{\mc}{\mathcal}
\newcommand{\cinfty}{C^{\infty}}
\newcommand{\pa}{\partial}
\newcommand{\prm}{\prime}
\newcommand{\dbar}{\bar\pa}
\newcommand{\OO}{{\mathcal O}}
\newcommand{\hotimes}{\hat\otimes}
\newcommand{\BV}{Batalin-Vilkovisky }
\newcommand{\CE}{Chevalley-Eilenberg }
\newcommand{\suml}{\sum\limits}
\newcommand{\prodl}{\prod\limits}
\newcommand{\into}{\hookrightarrow}
\newcommand{\Ol}{\mathcal O_{loc}}
\newcommand{\mD}{{\mathcal D}}
\newcommand{\iso}{\cong}
\newcommand{\dpa}[1]{{\pa\over \pa #1}}
\newcommand{\Kahler}{K\"{a}hler }
\newcommand{\0}{\mathbf{0}}

\newcommand{\B}{\mathcal{B}}
\newcommand{\V}{\mathcal{V}}

\newcommand{\M}{\mathfrak{M}}

\newcommand{\DD}{\Omega^{\text{\Romannum{2}}}}


\numberwithin{equation}{section}

\newcommand{\comment}[1]{\textcolor{red}{[#1]}} 

\makeatletter
\newcommand{\subjclass}[2][2010]{%
  \let\@oldtitle\@title%
  \gdef\@title{\@oldtitle\footnotetext{#1 \emph{Mathematics Subject Classification.} #2}}%
}
\newcommand{\keywords}[1]{%
  \let\@@oldtitle\@title%
  \gdef\@title{\@@oldtitle\footnotetext{\emph{Key words and phrases.} #1.}}%
}
\makeatother

\makeatletter
\let\orig@afterheading\@afterheading
\def\@afterheading{%
   \@afterindenttrue
  \orig@afterheading}
\makeatother

\makeatletter
\newcommand*\bigcdot{\mathpalette\bigcdot@{.5}}
\newcommand*\bigcdot@[2]{\mathbin{\vcenter{\hbox{\scalebox{#2}{$\m@th#1\bullet$}}}}}
\makeatother

\title{\bf Gromov-Witten Theory of $A_n$ type quiver varieties and Seiberg Duality}
\author{Yingchun Zhang}
\date{}
\maketitle
\begin{abstract}
Seiberg duality conjecture asserts that the Gromov-Witten theories (Gauged Linear Sigma Models) of two quiver varieties related by quiver mutations are equal via variable change. 
In this work, we prove this conjecture  for $A_n$ type quiver varieties.
 \end{abstract}

\setcounter{tocdepth}{2} \tableofcontents
\newpage
\section{Introduction}
Various dualities in physics have driven many mathematical developments in recent years. Mirror symmetry and LG/CY correspondence are two such examples.
A much less studied example in mathematics is the famous Seiberg duality which asserts the equivalence of gauge theories on quiver varieties. One can construct a new quiver variety by mutating an existing quiver.
Seiberg duality claims that the corresponding gauge theories in all dimensions are equivalent! Since gauge theories in different dimensions can be quite different in mathematics, this is a striking statement. To the author, its mathematical implication has not been studied nearly as much as those from other dualities. In this article, we will restrict ourselves to the 2d case where an excellent mathematical conjecture is available 
(see physical origin \cite{Hori,HoriTong,benini2015cluster,gomis2016m2} and  mathematical conjecture in \cite{nonabelianGLSM:YR}).

\subsection{Seiberg duality conjecture}
Suppose we have a quiver diagram $(Q_f\subset Q_0, Q_1, W)$, where $Q_0$ is the set of all nodes among which $Q_f$ is the set of frame nodes and $Q_0\backslash Q_f$ is the set of gauge nodes, $Q_1$ is the set of arrows, and $W$ is the potential. 
Denote an arrow between nodes $i$ and $j$ by $i\rightarrow j\in Q_1$, and denote the number of all arrows  $i\rightarrow j\in Q_1$ by $b_{ij}$. 
Assume there are no 1-cycle and 2-cycles, and such a quiver is called a cluster quiver. 
Let $\vec v=(v_i)_{i\in Q_0}$ be a collection of non-negative integers for each node $i\in Q_0$. Consider the affine variety algebraic variety $V=\bigoplus_{i\rightarrow j\in Q_1}Hom(\C^{N_i},\C^{N_j})$  and the connected linear reductive group $G=\prod_{i\in Q_0\backslash Q_f} GL(N_i)$.
For a choice of characters $\theta\in\chi(G)$, the associated quiver variety is defined by the GIT quotient $\mc X:=V\sslash_\theta G$.
See  Definition \ref{dfn:quiver} and Definition \ref{dfn:quivervariety} for details. 

\begin{dfn}[Definition \ref{dfn:mutation}] 
Fix a gauge node $k\in Q_0\backslash Q_f$. 
A quiver mutation at the gauge node $k$ is defined by the following steps. 
\begin{itemize}
    \item \textbf{Step (1)} Add another arrow $i\rightarrow j$ for each path $i\rightarrow k\rightarrow j$ passing through $k$, and invert directions of all arrows that start or end at the $k$. Suppose that we have a cycle containing this path in the potential $W$, and then we use the arrow $i\rightarrow j$ to replace the path $i\rightarrow k\rightarrow j$.
    \item \textbf{Step (2)} Convert $N_k$ to $N'_k:=\max(N_f(k), N_a(k))-N_k$, where $N_f(k)=\sum_{k\rightarrow j}b_{kj}N_j$ is called the outgoing and $N_a=\sum_{i\rightarrow k}b_{ik}N_i$ is called the incoming of the node $k$.
    \item \textbf{Step (3)} Remove  all pairs of opposite arrows between two nodes introduced by the mutation until all arrows between two nodes are in a unique direction. 
    \item \textbf{Step (4)} Add the new cubic terms arising from the quiver mutation. We get a new potential $\tilde W$. 
\end{itemize}
After performing a quiver mutation once, we denote the new quiver diagram by $(\tilde Q_f\subset \tilde Q_0, \tilde Q_1, \tilde W)$ and denote the associated input data for the quiver variety by $(\tilde V, \tilde G, \tilde \theta)$.
\end{dfn}
Consider the critical locus $Z:=Z(dW)$ of the potential $W$ and the GIT quotient $\mc Z=[Z\sslash_\theta G]$ before the quiver mutation. 
Consider the critical locus $\tilde Z:=Z(d\tilde W)$ of the potential $\tilde W$ and the GIT quotient $\tilde {\mc Z}:=[\tilde Z\sslash_{\tilde \theta} {\tilde G}]$, where $\tilde \theta$ is a carefully chosen character of the gauge group $\tilde G$. 
We aim to study the relations of Gromov-Witten (GW) theories of the two varieties $\mc Z$ and $\tilde{\mc Z}$. Genus $g$ Gromov-Witten invariants of a variety count stable maps from genus $g$ Riemann surfaces to the variety, see \cite{GW:kontsevich1995enumeration,GW:BM,GW:LT,GW:B}. 
Denote generating functions of genus $g$ GW invariants of $\mc Z$ and $\tilde{\mc Z}$ by $\mathscr{F}_g^{\mc Z}(\vec q)$ and $\mathscr{F}_{g}^{\tilde{\mc Z}}(\vec q')$ respectively, where $\vec q$ and $\vec q'$ are their k\"ahler variables. 
\begin{conjecture}[Seiberg duality conjecture \cite{benini2015cluster,nonabelianGLSM:YR}]\label{conj}
\footnote{One can find that the original mathematical Seiberg duality conjecture in \cite{nonabelianGLSM:YR} is about the transformations of GLSM under a quiver mutation, and in our work we are investigating the transformations of GW theory. 
In fact, the phases \eqref{eqn:phaseflag} and \eqref{eqn:quiverphaseofdual} before and after a quiver mutation we are choosing are geometric phases, where the GLSM theories are equal to GW theories of critical loci. Hence our results actually coincide with the original conjecture.}
\begin{equation}
    \mathscr{F}^{\mc Z}_g(\vec q)=\mathscr{F}^{\tilde{\mc Z}}_{g}(\vec q')\,,
\end{equation}
and the cluster transformations on the K\"ahler coordinates\footnote{Great thanks to Peng Zhao for suggesting this illuminating terminology.} are : $q_k'=q_k^{-1}$ and for $j\neq k$,
\begin{itemize}
        \item  if $N_f(k)>N_a(k)$, $\frac{{e}^{\pi i(N_f(j)'-1)}q_j'}{ {e}^{\pi i(N_f(j)-1)}q_j}=({e}^{\pi i N_k'}q_k)^{[b_{kj}]_+}
        ({e}^{\pi i N_k'})^{[-b_{kj}]_+} \prod_{i\neq k}e^{\pi iN_i a_{ij}} $;
        \item if $N_f(k)=N_a(k)$, $\frac{{e}^{\pi i(N_f(j)'-1)}q_j'}{ {e}^{\pi i(N_f(j)-1)}q_j}=\left(\frac{{e}^{\pi i N_k'}q_k}{1+(-1)^{N_1'}q_k}\right)^{[b_{kj}]_+}
        \left({e}^{\pi i N_k'}(1+(-1)^{N_1'}q_k)\right)^{[-b_{kj}]_+} \prod_{i\neq k}e^{\pi iN_i a_{ij}} $.
        \item If $N_f(k)<N_a(k)$, $\frac{{e}^{(N_f(j)'-1)}q_j'}{ {e}^{(N_f(j)-1)}q_j}=({e}^{\pi i N_k'})^{[b_{kj}]_+}
        \left({e}^{\pi i(N_f(k)-N_k)}q_k\right)^{-[-b_{kj}]_+} \prod_{i\neq k}e^{\pi iN_i a_{ij}} $
    \end{itemize}
    where $a_{ij}$ denotes the number of “annihilated” 2-cycles between the nodes i and j in step $(3)$ of the quiver mutation mechanism.
\end{conjecture}
This work will focus on the genus-zero Seiberg duality conjecture.
We utilize the quasimap theory and wall-crossing theorems to investigate the Gromov-Witten theory of GIT quotients $\mc Z$ and $\tilde{\mc Z}$, see \cite{MR2729000,MR3126932,MR3412343} for the quasimap theory, see \cite{MR3272909,MR3586512,MR3412343,ciocan2017higher,clader2017higher,wang2019mirror,zhou2019quasimap,ciocan2020quasimap} for wall-crossing theorems for various targets and genera.

Assume both $\mc Z$ and $\tilde{\mc Z}$ are quasiprojective, semi-positive, and admit a good torus action $S$. Further, we assume that both  $Z$ and $\tilde Z$ have at most lci singularities. All of our examples satisfy those conditions. 
The genus-zero wall-crossing theorem  \cite{MR3272909} states that the equivariant quasimap small $I$-function and the equivariant small $\mc J$-function of a GIT quotient are equal under the mirror map when the GIT quotient satisfies the above conditions, see a review in Section \ref{Sec:GW}.
Therefore, in order to prove the genus-zero Seiberg duality conjecture, we only have to prove that the equivariant quasimap small $I$-function of $\mc Z$ denoted by $I^{\mc Z, S}$ and that of $\tilde{\mc Z}$ denoted by $I^{\tilde{\mc Z}, S}$ satisfy all relations in the conjecture. 

All of our examples are nonabelian GIT quotients, so we apply the results of Rachel Webb about the abelian-nonabelian correspondence to study the equivariant quasimap small $I$-functions \cite{abeliannonabelian:Webb,abelianizationlef:Webb}. 
Let $T\subset G$ be the maximal torus. Then, consider the GIT quotient $Z\sslash_\theta T$ and its equivariant quasimap small $I$-function.
The abelian-nonabelian correspondence proves that the equivariant quasimap $I$-function of $\mc Z$ is the equivariant quasimap $I$-function of $Z\sslash_\theta T$ twisted by a factor.
There are some earlier nice works about the abelian-nonabelian correspondence \cite{givental1995quantum,abelian/nonabelian:CKB,abeliannonabelian:BCK2,abalian/nonabelian:BCK3,flag:quantumcohomology}.

\subsection{Main results}
In this work, we mainly consider $A_n$ type quivers in Figure \ref{diag:flagintrod} with $N_1\leq \ldots\leq N_D$.
\begin{figure}[H]
\centering
 \includegraphics[width=3.2in]{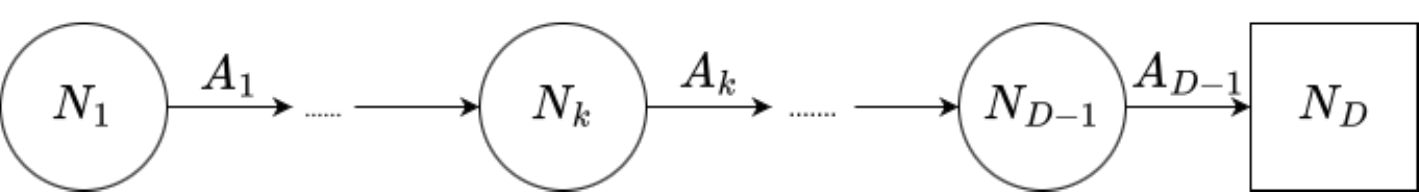}
 \caption{$A_n$-type quiver of length $D$}
 \label{diag:flagintrod}
\end{figure}
For the character $\theta$ in \eqref{eqn:phaseflag} of the gauge group $G=\prod_{i=1}^{D-1}GL(N_i)$, the quiver variety is a flag variety  $F(N_1,\ldots,N_D)$.

 Applying a quiver mutation, we obtain another quiver diagram with a potential $W=tr(BA_kA_{k-1})$ as Figure \ref{diag:dualofflagintrod}.
 \begin{figure}[H]
     \centering
     \includegraphics[width=3.5in]{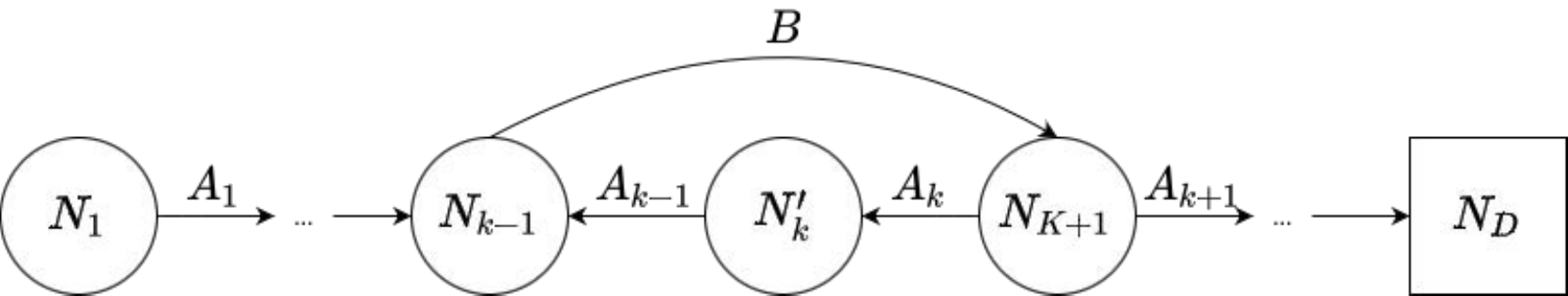}
     \caption{Quiver mutation of an $A_n$-type quiver}
     \label{diag:dualofflagintrod}
 \end{figure}
Consider the critical locus of the potential $Z=Z(dW)$. In the character $\tilde \theta$ in \eqref{eqn:quiverphaseofdual}, 
\begin{equation}
    Z^s(G)=\{ BA_k=0, \,s.t. \,A_i,  \text { for } i\neq k-1, \text{ and }\,B \text{ nondegenerate}  \}.
\end{equation}
Hence the GIT quotient
\begin{equation}\label{eqn:intr_complete}
    \mc Z^1:=Z\sslash_{\tilde \theta} G=\{BA_k=0\}\sslash_{\tilde\theta} G\,,
\end{equation}
is a subvariety in $\mc X^1$ which is defined by the quiver diagram below in the character $\tilde \theta$ in \eqref{eqn:quiverphaseofdual}.
\begin{figure}[H]
    \centering
    \includegraphics[width=3.5in]{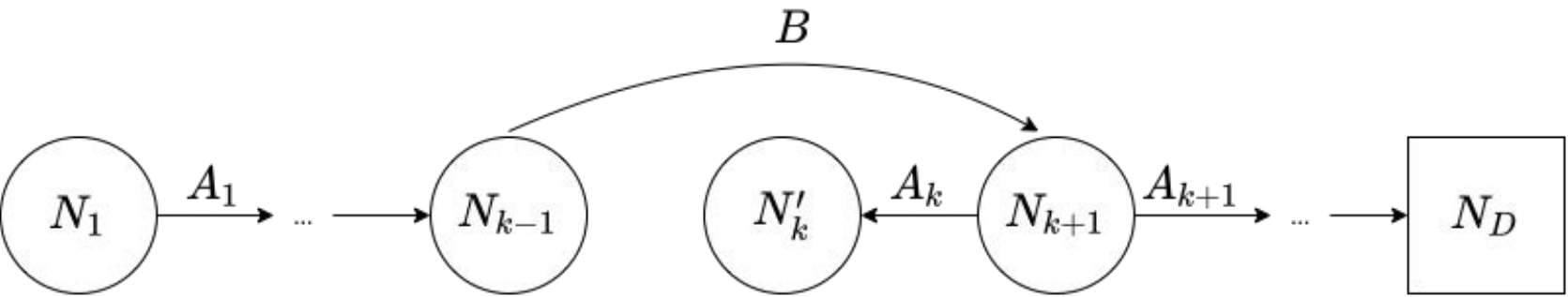}
\end{figure}
Both $F(N_1,\ldots, N_D)$ and $\mc Z^1$ admit a good torus action $S=(\C^*)^{N_D}$ coming from the frame node, see Equations \eqref{eqn:Saction}\eqref{eqn:torusS2}.
Denote the equivariant quasimap small $I$-functions of $F(N_1,\ldots, N_D)$ and $\mc Z^1$ by $I^{F, S}(\vec q)$ and $I^{\mc Z^1,S}(\vec q')$.
One of our main results proves that the two functions $I^{F, S}(\vec q)$ and $I^{\mc Z^1,S}(\vec q')$ satisfy the relations in Seiberg Duality conjecture.
\begin{thm}[Theorem \ref{thm:main}]\label{thm:intro1}
\begin{itemize}
    \item When $N_{k+1}\geq N_{k-1}+2$,
    \begin{equation}
    I^{F, S}(\vec q)=I^{\mc Z^1,S}(\vec q')\,,
\end{equation}
and the cluster transformations on the K\"ahler coordinates are,
\begin{equation}
    q_{k}'=q_k^{-1},\, q_{k+1}'=q_{k+1}q_k, \,q_i'=q_i\, \text{ for } i\neq k, k+1\,.
\end{equation}
\item When $N_{k+1}=N_{k-1}+1$, 
\begin{equation}
    I^{F,S}(\vec q)={e}^{(-1)^{N_{k}'}q_k}I^{\mc Z^1,S}(\vec q')\,, 
\end{equation}
with the cluster transformations on the K\"ahler coordinates
\begin{equation}
    q_{k}'=q_k^{-1}, \,q_{k+1}'=q_{k+1}q_k,\,q_i'=q_i\, \text{ for } i\neq k, k+1\,.
\end{equation}
\item When $N_{k+1}=N_{k-1}$,
 \begin{equation}
    I^{F,S}(\vec q)=I^{\mc Z^1,S}(\vec q')\,,
    \end{equation}
    with the cluster transformations on the K\"ahler coordinates,
    \begin{equation}
    q_k'=q_k^{-1}\,,
        q_{k+1}'=\frac{q_kq_{k+1}}{1+q_k}\,, q_{k-1}={q_{k-1}}({1+q_k})\,,q_i'=q_i \,\text{ for } i\neq k, k+1,k-1\,.
    \end{equation}
\end{itemize}
\end{thm}

Furthermore, we consider the generalized $A_n$-type quiver in Figure \ref{diag:tautflag} with $N_0\leq N_2$ and $N_1\leq \ldots\leq N_D$
\begin{figure}[H]
    \centering
    \includegraphics[width=3.4in]{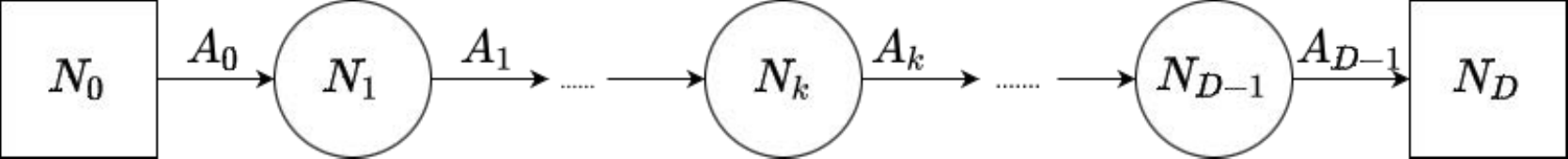}
    \caption{$A_n$-type quiver with two frame nodes}
    \label{diag:tautflag}
\end{figure}
\noindent which gives the total space of $N_0$-copies of the tautological bundle $S_1$ over a flag variety, denoted by $S_1^{\oplus N_0}\rightarrow F(N_1,\ldots,N_D)$, see Example \ref{ex:tautflag}. 
There are two different situations when performing a quiver mutation.
In the first case, the quiver mutation is at a gauge node $k\neq 1$. Repeating the construction we do for the flag variety, we obtain $\mc Z^3:=\{ BA_k=0 \}\sslash_{\tilde \theta} \tilde G\subseteq \mc X_3$, where $\mc X_3$ is the quiver variety defined by the quiver diagram below, 
\begin{figure}[H]
    \centering
    \includegraphics[width=3.4in]{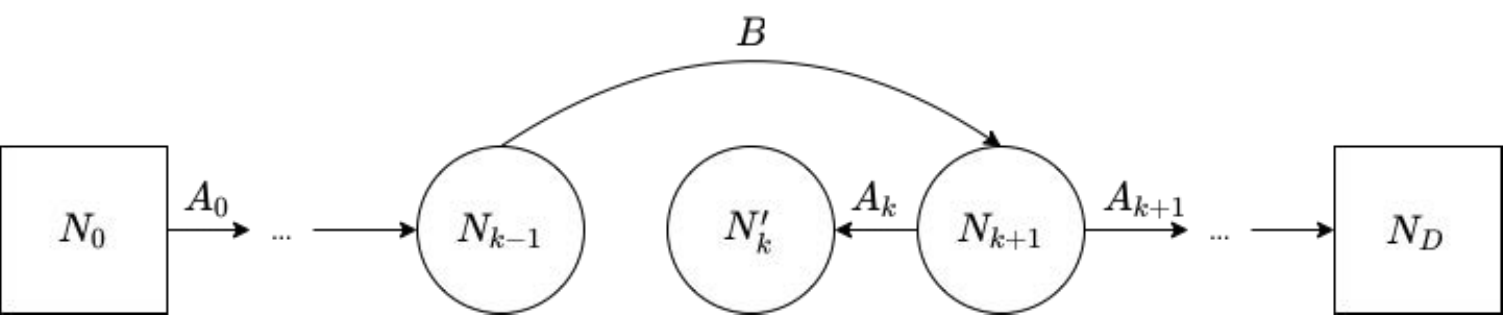}
\end{figure}
\noindent Notice that $\mc Z^3$ is the total space of $N_0$-copies of the tautological bundle $S_1$ over $\mc Z^1$, which is known as the local target over $\mc Z^1$ in GW theory.
In the second case, we apply a quiver mutation at gauge node $k=1$, and get $\mc Z^4 =\{BA_1=0\}\sslash_{\tilde \theta} \tilde G \subseteq \mc X^4$, where $\mc X^4$ is the quiver variety defined by the quiver diagram below,
\begin{figure}[H]
    \centering
    \includegraphics[width=3.4in]{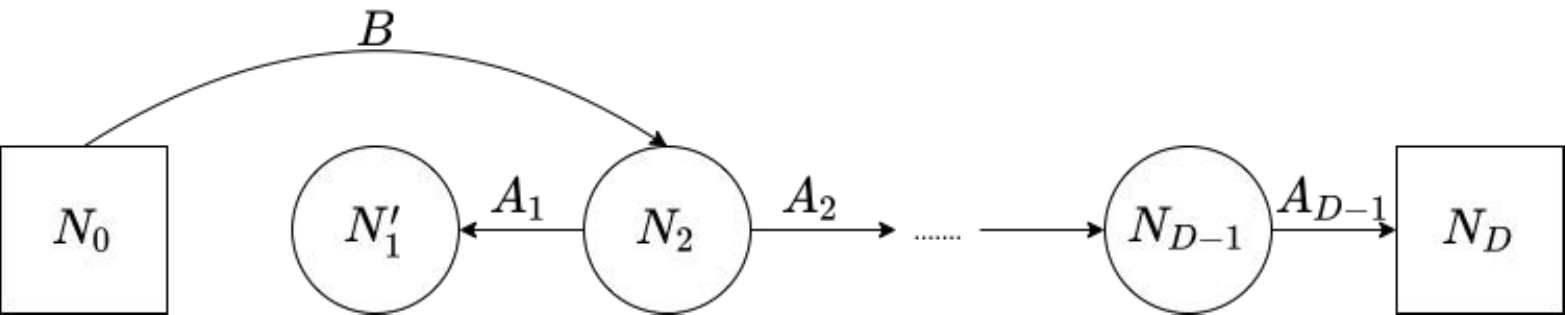}
\end{figure}
\noindent The two cases behave differently, since in the first one the leftmost frame node in the quiver diagram is far away from the gauge node $k$, but in the second case it's connected to the mutated gauge node. This slight difference causes different results  when $N_0=N_2$.

All three varieties $S_1^{\oplus N_0}\rightarrow F(N_1,\ldots, N_D)$ and $\mc Z^3$ $\mc Z^4$ admit a good torus action $S^2=(\C^*)^{N_0}\times (\C^*)^{N_D}$.
Denote their equivariant quasimap small $I$-functions by $I^{tF,S^2}(\vec q)$, $I^{\mc Z^3,S^2}(\vec q')$ and $I^{\mc Z^4,S^2}(\vec q')$ respectively. 
\begin{thm}[Theorem \ref{thm:main2}]
\begin{enumerate}
\item
Applying a quiver mutation at a gauge node $k\neq 1$,  we prove that $I^{tF,S^2}(\vec q)$ and $I^{\mc Z^3,S^2}(\vec q')$ exactly satisfy all relations in Theorem \ref{thm:intro1}.
\item
Applying a quiver mutation at the gauge node $k=1$, we obtain the following results. 
\begin{enumerate}
    \item When $N_2\geq N_0+1$, $I^{tF,S^2}(\vec q)$ and $I^{\mc Z^4,S^2}(\vec q')$ satisfy the same relations with those in Theorem \ref{thm:intro1} with the same K\"ahler coordinates transformations.
    \item  When $N_0=N_2$, 
    \begin{equation}
    I^{tF,S^2}(\vec q)=(1+(-1)^{N_1'}q_1)^{\sum_{A=1}^{N_0}\eta_A-\sum_{F=1}^{N_2}x_{F}^{(2)}+N_1'}I^{\mc Z^4,S^2}(\vec q')\,,
    \end{equation}
   with cluster transformations on the K\"ahler coordinates
         \begin{equation}
         q_1'=q_1^{-1},\, q_{2}'=\frac{q_1q_{2}}{1+(-1)^{N_1'}q_1},\, q_i'=q_i,\text{ for } i\neq 1,2\,,
    \end{equation}
 where $(1+(-1)^{N_1'}q_1)^{\sum_{A=1}^{N_0}\eta_A-\sum_{F=1}^{N_2}x_{F}^{(2)}+N_1'}$ is formally expanded as 
  \begin{equation}
\sum_{m\geq 0}\frac{\prod_{l=0}^{m-1}(\sum_{A=1}^{N_0}\eta_A-\sum_{F=1}^{N_2}x_{F}^{(2)}+N_1'-l)}{m!} ((-1)^{N_1'}q_1)^m\,.
  \end{equation}
See Section \ref{sec:proof} for notations in the above expressions.
\end{enumerate}
\end{enumerate}
\end{thm}
Let $\mc J^{\mc Z, S}(\vec q)$ and $\mc J^{\tilde{\mc Z}, S}$ denote the small $J$-functions of $\mc Z$ and $\tilde{\mc Z}$ which comprises their genus-zero GW invariants.
By wall-crossing theorem \cite{MR3272909,zhou2019quasimap}, when $\mc Z$ and $\tilde{\mc{Z}}$ are Fano of index at least 2, 
\begin{equation}
   \mc J^{\mc Z,S}(\vec q)=I^{\mc Z,S}(\vec q),\,\, \mc J^{\tilde{\mc Z},S}(\vec q)=I^{\tilde{\mc Z},S}(\vec q)\,.
\end{equation}
The above two theorems conclude the Seiberg duality conjecture for small $\mc J$-functions. 
\begin{cor}[Theorem \ref{cor:seiberg}]
When $N_1<N_2<\ldots<N_D$ and $N_0+2\leq N_2$, 
the genus-zero Seiberg duality conjecture holds for $A_n$-type quivers: let $(\mc Z, \tilde {\mc Z})$ represent pairs of varieties $(F(N_1,\ldots, N_D), \mc Z^1),\, (S_1^{\oplus N_0}\rightarrow F(N_1,\ldots, N_D), \mc Z^3),\,S_1^{\oplus N_0}\rightarrow (F(N_1,\ldots, N_D), \mc Z^4)$, and then  
\begin{equation}
    \mc J^{\mc Z, S}(\vec q)=\mc J^{\tilde{\mc Z}, S}(\vec q')
\end{equation}
under the k\"ahler coordinates transformations $q_{k}'=q_{k}^{-1},\, q_i'=q_i$ for $i\neq k$. 
\end{cor}
\begin{rem}
The easiest case of Seiberg duality is that of the Grassmannian $Gr(r,n)$ and the dual Grassmannian $Gr(n-r, n)$. However, $Gr(r,n)$  and $Gr(n-r, n)$ are isomorphic, and the 2d Seiberg duality holds for a trivial reason. Our cases probably provide some examples with different underlying geometries.
\end{rem}
\begin{rem} 
3d gauge theories are thought to correspond to quantum K-theory \cite{benini2011comments,xie2013three,closset2012seiberg}. Mathematically, there is a nice result relating K-theoretic $I$-functions of the Grassmannian and the dual Grassmannian with a nontrivial change of level structures in \cite{dong2020level}.
\end{rem}
Let's describe ideas for our proofs. The fundamental building block is Seiberg duality between the pair of quivers in Figure \ref{figure5}. 
\begin{figure}[H]
    \centering
    \includegraphics[width=4.2in]{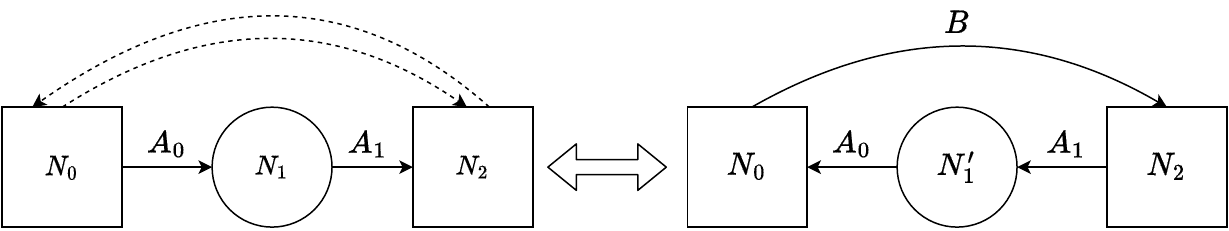}
    \caption{The two can be transferred to each other by mutating the middle gauge node. From the left to the right, there is a $3$-cycle arising and hence we have a potential $W=tr(BA_1A_0)$. From the right to the left, there are two opposite arrows which are dashed in the figure, and they are are annihilated.}
    \label{figure5}
\end{figure}
The two geometries before and after the quiver mutation are the total space of $N_0$-copies of the tautological bundle  $S_1$ over a Grassmannian $S_1^{\oplus N_0}\rightarrow Gr(N_1,N_2)$ and the total space of $N_0$-copies of the dual of the tautological bundle $S_1^\vee$ over the dual Grassmannian $(S_1^\vee)^{\oplus N_0}\rightarrow Gr(N_1',N_2)$. The two quasiprojective varieties admit a good torus action $S^{2}=(\C^*)^{N_0}\times (\C^*)^{N_2}$. 
The two varieties' equivariant quasimap small $I$-functions are equal via variable change under the condition $N_0\leq N_2-1$, see \cite{donghai}.

Seiberg duality is a local property, which means 
 a quiver mutation at a gauge node $k$ only affects the nodes $k-1$, $k+1$ that admit arrows with the node $k$, as shown in Figure \ref{diag:ideaofproof} below.  
Let $\mc Z$ and $\tilde{\mc Z}$ denote the two varieties before and after a quiver mutation.  
By isolating the terms of $I^{\mc Z, S}$ and $I^{\tilde {\mc Z}, S}$ that involve information of nodes $k,k-1$ and $k+1$, we find these terms behave in the same manner as $I$-functions of the fundamental building block. We are done if we could identify those terms. 
\begin{figure}[H]
    \centering
    \includegraphics[width=5.2in]{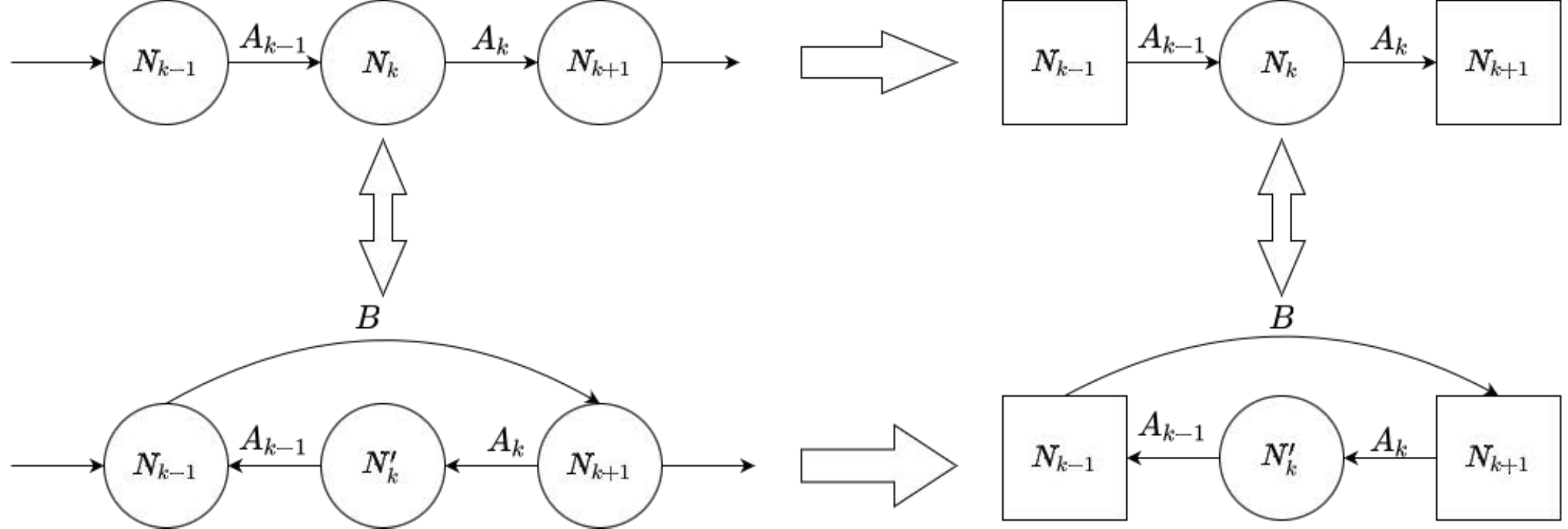}
    \caption{Local Seiberg duality}
    \label{diag:ideaofproof}
\end{figure}

In this work, we only consider $A_n$-type quivers with gauge group $GL(m)$ assigned to each gauge node, and we will consider star-shaped quivers in another work \cite{starquiver:hezhang}. 
Hori has studied another fascinating duality with Lie groups other than $GL(m)$, and it is not clear if there are any relations with Seiberg duality, see \cite{HoriTong,HoriKnapp,Hori}.

\subsection{Organization of the paper}
Section \ref{Sec:quiver} is mainly about some facts about quiver varieties, definitions of quiver mutations, and the construction of varieties after a quiver mutation. Section 3 is devoted to Gromov-Witten theory and wall-crossing theorems. 
Section 4 is about the equivariant quasimap small $I$-functions of our examples. Finally, section 5 consists of all proofs of our theorems.
\subsection*{Acknowledgments}
The author is grateful to Yongbin Ruan for proposing this topic, to Weiqiang He for his helpful discussion in many aspects,  Rachel Webb for help with abelian/nonabelian correspondence and Nawaz Sultani for clarification about $I$-functions for GIT quotients and suggestions in polishing the work, to Aaron Pixton for suggestions on polishing this work and more interesting questions related to this topic, to Peng Zhao for helpful discussion on related works in this topic. 
\section{Quiver varieties and its mutations}\label{Sec:quiver}
There are many good references for quiver varieties, and we mainly consult the nice book \cite{quiver}. 
\subsection{Quiver varieties}
Let $V=\op{Spec}(A)$ be an affine algebraic variety over $\C$ with at most lci singularities and let $G$ be a connected reductive algebraic group acting on $V$.
Let $\chi(G): G\rightarrow \C^*$ be the character group of $G$ and let $\theta\in \chi(G)$ be a character. 
Each character $\theta \in \chi(G)$ determines an one-dimensional representation $\C_\theta$ of $G$ and a line bundle 
\begin{equation}
    L_\theta:=V\times \C_\theta\in \text{Pic}^G(V)\,.
\end{equation}
\begin{dfn} Given an input data $(V,G,\theta)$,
$x\in V$ is called $\theta$-semistable if 
     $\exists\, k>0$ and $s\in {H}^0(V, {L}_\theta^k)^G$, such that $s(x)\neq 0$ and every $G$-orbit in $D_s=\{s\neq 0\}$ is closed. 
     Further, a $\theta$-semistable point $x\in V$ is called $\theta$-stable if its
    stabilizer $\op{Stab}_{G}(x)=\{ g\in G, g\cdot x=x\}$ is finite.
Let  $V^{ss}_\theta (G)$ denote the set of semistable points, $V^{s}_\theta(G)$ the set of stable points, and $V^{us}_\theta(G)$ the set of unstable points. 
The GIT quotient of $(V, G, \theta)$ is defined as $V\sslash_\theta G:=V^{ss}_\theta (G)\slash G$. 
\end{dfn}
The following will be assumed throughout. 
\begin{enumerate}[label=\roman*]
    \item $V^s=V^{ss}\neq \emptyset$.
    \item The subscheme $V^s$ is nonsingular. 
    \item The group $G$ acts freely on $V^s$. 
\end{enumerate}
We will instead denote the set of semistable points, the set of stable points, and the set of unstable points by $V^{ss}(G)$, $V^s(G)$, and $V^{us}(G)$, and denote the GIT quotient of $(V, G, \theta)$ by $V\sslash G$ when there is no confusion arising for the character.

\begin{dfn}[\cite{quiver}]\label{dfn:quiver}
A  {quiver diagram} is a finite oriented graph consisting of $(Q_f\subset Q_0, Q_1, W)$, where 
\begin{itemize}
    \item $Q_0$ is the set of vertices among which $Q_f$ is the set of frame nodes, usually denoted by $\framebox(3,3){}$ in the graph, and $Q_0\backslash Q_f$ is the set of gauge nodes, usually denoted by $\bigcirc $ in the graph.
    \item $Q_1$ is the set of arrows. An arrow from nodes $i$ to $j$ is denoted by $i\rightarrow j\in Q_1$, and the number of all such arrows is denoted by $b_{ij}$.
    \item $W$ is the potential, defined as a function on cycles in the diagram. 
\end{itemize}
\end{dfn}
We always assume the quiver diagram has no $1$-cycle or $2$-cycles, and this type of quiver is known as a cluster quiver.
For a cluster quiver, $b_{ij}\in \Z$ can be positive or negative. $b_{ij}>0$ indicates that all arrows are from $i$ to $j$, and $b_{ij}<0$ indicates that all arrows are from $j$ to $i$.

\begin{dfn}\label{dfn:quivervariety}
For a quiver diagram $(Q_f\subseteq Q_0, Q_1, W)$, let $\vec v=(N_i)_{i\in Q_0}$ be a collection of nonnegative integers for each node $i\in Q_0$.
Let $V=\bigoplus_{i\rightarrow j\in Q_1}\C^{N_i\times N_j}$ and $G=\prod_{i\in Q_0\backslash Q_f}GL(N_i)$ be an affine variety and a connected reductive group.
Fix the action of $G$ on $V$ firmly in the following way. For each $g=(g_i)_{i\in Q_0\backslash Q_f}\in G$ and each $A=(A_{i\rightarrow j})_{i\rightarrow j\in Q_1}\in V$, where $A_{i\rightarrow j}$ is an element in the vector space of matrices $\C^{N_i\times N_j}$,
\begin{equation}\label{eqn:sec2GonV}
    g\cdot (A_{i\rightarrow j})=(g_iA_{i\rightarrow j}g_{j}^{-1})\,.
\end{equation}
For a character $\theta\in \chi(G)$
\begin{equation}\label{eqn:polarizationchar}
\theta(g)=\prod_{i\in Q_{0}\backslash Q_f} \det(g_i)^{\sigma_i}\,,
\end{equation}
where  $\sigma_i\in \R$,
the quiver variety is defined to be the GIT quotient $V\sslash _{\theta}G$.
For a cycle $k_1\rightarrow k_2\rightarrow \cdots\rightarrow k_i\rightarrow k_1$ in the quiver diagram, we can define a $G$-invariant function 
\begin{equation}
    tr(A_{k_1\rightarrow k_2}\cdots A_{k_i\rightarrow k_1})\,.
\end{equation}
The potential $W$ is the sum of such $G$-invariant functions on cycles.
\end{dfn}
\begin{dfn}\label{dfn:outgoingincoming}
Given a quiver diagram $(Q_f\subset Q_0, Q_1,W)$, the outgoing of a node $k$ is defined as $N_f(k):=\sum_{i}[b_{ki}]_+N_i$, and the incoming is defined as $N_a(k):=\sum_{i}[b_{ik}]_+N_i$, with $[b]_{+}:=\max\{ b,0\}$ for any integer $b$.
\end{dfn}
Notice that the potential $W$ is $G$-invariant, so $W$ descends to a function on $V\sslash_\theta G$. 
\begin{ex}\label{ex:flag}
An $A_n$-type quiver diagram is as follows with $N_1\leq \ldots \leq N_D$.
\begin{figure}[H]
    \centering
    \includegraphics[width=3.2in]{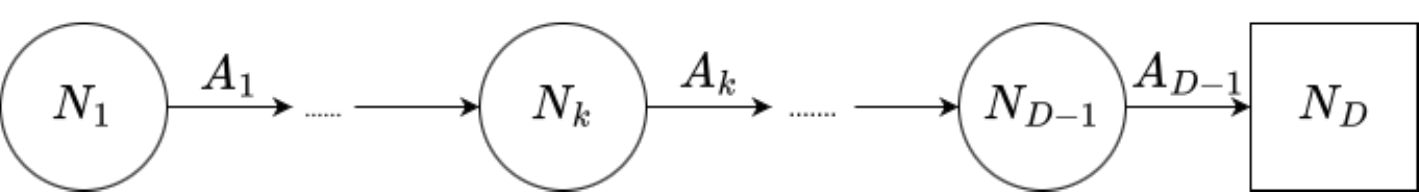}
   \caption{$A_n$ type quivers of length $D$}
    \label{diag:flag}
\end{figure}
\noindent
The $A_i$ above arrows represent matrices in $\C^{N_i\times N_{i+1}}$.
In this example, $V=\oplus_{i=1}^{D-1}\C^{N_i\times N_{i+1}}$, $G=\prod_{i=1}^{D-1}GL(N_i)$, and $G$ acts on $V$ as \eqref{eqn:sec2GonV}. 
Choose a character
\begin{equation}
    \theta(g)=\prod_{i=1}^{D-1}\det(g_i)^{\sigma_i}\,,
\end{equation}
with positive phase
\begin{equation}\label{eqn:phaseflag}
    \sigma_i>0,\,\, \forall i\in Q_0 \backslash Q_f\,.
\end{equation}
The semistable locus is
\begin{equation}
    V^{ss}(G)=\{A_{i}\text{ nondegenerate }, \forall i \}\,.
\end{equation} 
The GIT quotient $V\sslash G$ is a flag variety and we denote it by $F(N_1,\ldots,N_D)$.
In particular, when there is only one gauge node and one frame node, the quiver variety is a Grassmannian $Gr(N_1, N_2)$.

The flag variety $F(N_1,\ldots, N_D)$ admits a set of tautological bundles $S_1\subseteq S_2\subseteq \ldots\subseteq S_{D}=\C^{N_D}$.
\end{ex}

\begin{ex}\label{ex:tautflag}
We can also consider the generalized 
$A_n$-type quiver diagram  with $N_1\leq \ldots \leq N_D$
\begin{figure}[H]
    \centering
    \includegraphics[width=3.5in]{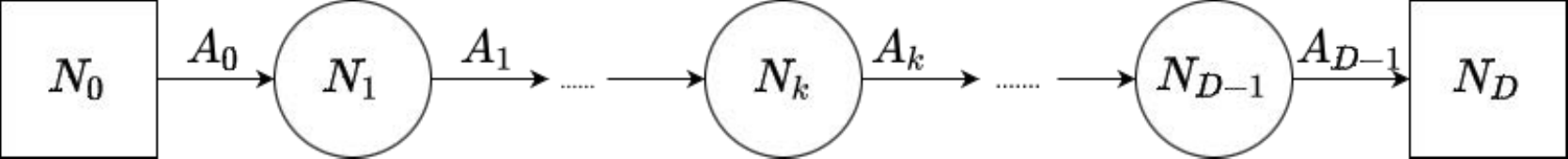} 
    \caption{Generalized $A_n$ type quivers with two frame nodes}
    \label{fig:tautflag}
\end{figure}
\noindent
With abuse of notation, we denote by $(V, G, \theta)$ the input data of the quiver variety, where $G$ acts on $V$ in the standard way \eqref{eqn:sec2GonV} and the phase is chosen as \eqref{eqn:phaseflag}.
Then
\begin{equation}
   V^{ss}(G)=\{A_{k}\text{ nondegenerate }, \text{ for } k=1,\ldots, D-1 \}\,,
\end{equation}
 and the GIT quotient is the total space of $N_0$ copies of the tautological bundle $S_1$ over a flag variety $F(N_1,N_2,\ldots,N_D)$, which we denoted by
    $V\sslash G:=S_1^{\oplus N_0}\rightarrow F(N_1,\ldots,N_D)$.
In particular, when $D=2$, the quiver variety is the total space of $N_0$ copies of the tautological bundle $S_1$ over a Grassmannian $S_1^{\oplus N_0}\rightarrow Gr(N_1, N_2)$. 
\end{ex}
\subsection{Quiver Mutation}
We introduce the quiver mutation applet in this section.
Assume we are given a quiver diagram $(Q_f\subseteq Q_0,Q_1,W)$ and a collection of integers $\vec v=(N_i)_{i\in Q_0}$.
\begin{dfn}\label{dfn:mutation}
Fix a gauge node $k$ of the quiver diagram $(Q_f\subset Q_0, Q_1, W)$.
A quiver mutation at the gauge node $k$ is defined by the following steps. 
\begin{itemize}
    \item \textbf{Step (1)} Add another arrow $i\rightarrow j$ for each path $i\rightarrow k\rightarrow j$ passing through $k$, and invert directions of all arrows that start or end at the $k$. Suppose that we have a cycle containing this path in the potential $W$, and then we use the arrow $i\rightarrow j$ to replace the path $i\rightarrow k\rightarrow j$.
    \item \textbf{Step (2)} Convert $N_k$ to $N'_k=\max(N_f(k), N_a(k))-N_k$. 
    \item \textbf{Step (3)} Remove  all pairs of opposite arrows between two nodes introduced by the mutation until all arrows between two nodes are in a unique direction. 
     \item \textbf{Step (4)} Add new cubic terms arising from the quiver mutation to $W$. 
\end{itemize}
\end{dfn}
A quiver mutation does not generate any 1-cycle or 2-cycles by step $(3)$, so a cluster quiver is transformed to another cluster quiver via a quiver mutation, denoted by $(\tilde Q_f\subset\tilde Q_0, \tilde Q_1,\tilde W)$. 
Throughout the paper, we will reserve the letter $k$ for the gauge node at which we perform quiver mutations.

\begin{ex}\label{ex:dualofflag}
Performing a quiver mutation to an $A_n$-type quiver at a gauge node $k$, we obtain the following quiver diagram with a potential $\tilde W=tr(BA_kA_{k-1})$ .
\begin{figure}[H]
    \centering
    \includegraphics[width=3.6in]{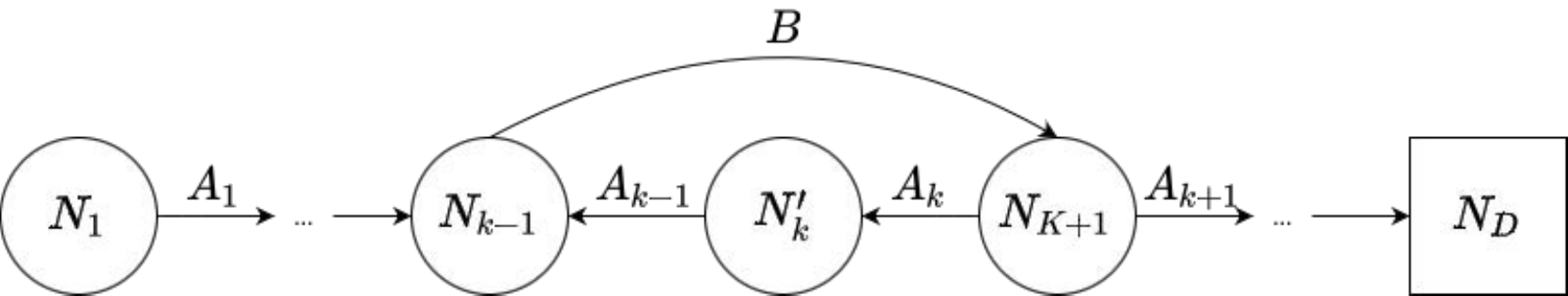}
   \caption{Dual quiver diagram by performing mutation at a gauge node $k$ of $A_n$-type quiver.}
    \label{pic:dualofflagD}
\end{figure}
\noindent
Then  $\tilde V=\bigoplus_{i\neq k-1,k}\C^{N_i\times N_{i+1}}\oplus \C^{N_{k-1}\times N_{k+1}}\oplus \C^{N_{k+1}\times N_{k}'}\oplus \C^{N_k'\times N_{k-1}}$  and $\tilde G=\prod_{i\neq k}GL(N_i)\times GL(N_k')$ are the associated affine variety and the gauge group. $\tilde G$ acts on $\tilde V$ as \eqref{eqn:sec2GonV}. 
Choose a character  ${\theta}(g)=\prod_{i=1}^{D-1}\det{g_i}^{\sigma_i}$ of $\tilde G$ as follows,
\begin{align}\label{eqn:quiverphaseofdual}
    &\sigma_k<0\,,\nonumber\\
    &\sigma_{i}>0\,\,\forall i\neq k\,,\nonumber\\
    & \sigma_{k}+\sigma_{k+1}>0\,.
\end{align}
One can check that 
\begin{equation}
     \tilde V^{ss}_{\tilde \theta}(\tilde G)=\{B, A_i \text{ nondegenerate for } i\neq k-1,k+1 ,\, \begin{bmatrix} A_{k}&A_{k+1} \end{bmatrix} \,\text{nondegenerate}\}\,.
\end{equation}
Comparing the phase in \eqref{eqn:quiverphaseofdual} and the phase of $G$ in \eqref{eqn:phaseflag}, we notice that only the gauge group of the gauge node $k$ changes its phase to the negative.

Consider the critical locus of the potential $Z^1:=Z(d \tilde W)$ which is equivalent to the following three equations,
\begin{equation}
    BA_k=0,\,A_kA_{k-1}=0,\,A_{k-1}B=0\,.
\end{equation}
 Then one can check
 \begin{equation}
     Z^1\cap \tilde V^{ss}(\tilde G)=\{BA_k=0, A_{k-1}=0,\text{ s.t. } A_i, \text{ for } i\neq k-1, \text{ and } B \text{ nondegenerate} \}\,.
 \end{equation}
Consider the quiver diagram in Figure \ref{pic:dualofflag2D} by deleting the arrow $k\rightarrow k-1$ in the Figure \ref{pic:dualofflagD} to respond to the $A_{k-1}=0$ equation in the critical locus of $W$.
\begin{figure}[H]
     \centering
     \includegraphics[width=3.5in]{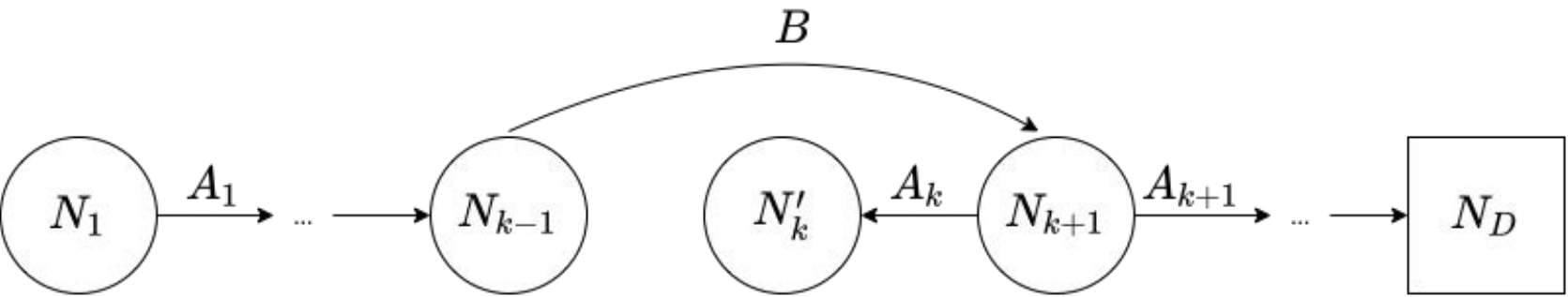}
     \caption{}
     \label{pic:dualofflag2D}
 \end{figure}
\noindent In the character $\tilde\theta$ in Equation \eqref{eqn:quiverphaseofdual}, the semistable locus consists of
\begin{equation}
    \{A_i \,\forall i \text{ and } B  \text{ nondegenerate} \}\,.
\end{equation} 
Denote the corresponding quiver variety by  ${\mc X^1}$, and then one can find the GIT quotient $\mathcal Z^1:=Z^1\sslash G= Z^1\cap \tilde V^{ss}(\tilde G) \slash G$ is a subvariety in $\mathcal X^1$ defined by the equation $\{BA_k=0\}$, where we use $A_i$ and $B$ to denote the matrix element in the quiver variety $\mathcal X^1$ by abuse of notation. 
Alternatively, $\mc Z^1$ is a zero locus of a regular section $s$ of the bundle $S_{k-1}^\vee\otimes S_k^\vee$ over the quiver variety $\mc X^1$.

A special situation is when we apply a quiver mutation at the node $k=1$ as Figure \ref{pic:dualofflagnode1} without potential. 
The resulting quiver diagram is 
\begin{figure}[H]
    \centering
   \includegraphics[width=3in]{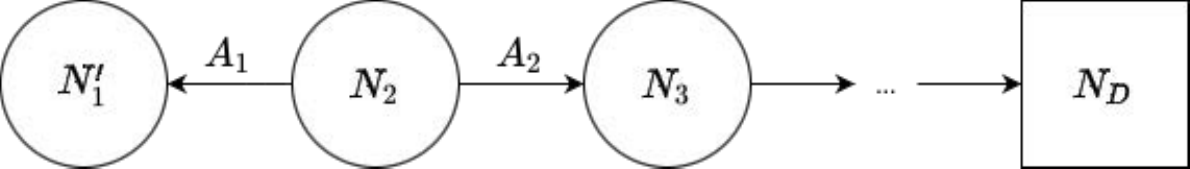}
    \caption{The quiver diagram by performing a quiver mutation at the gauge node $k=1$}
    \label{pic:dualofflagnode1}
\end{figure}
\noindent
In the same phase \eqref{eqn:quiverphaseofdual}, the semistable locus consists of 
\begin{equation}
    \{A_i \text{ nondegenerate }\}\,.
\end{equation}
Since there is no potential, 
the quiver variety denoted by $\mc X^{2}:=\tilde V\sslash \tilde G$ is what we are looking for in the dual side in this case. 
\end{ex}

\begin{ex}\label{ex:dualoftautflag}
In this example, we study quiver mutations to a generalized $A_n$-type quiver diagram in Figure \ref{fig:tautflag} and the corresponding varieties in the dual side.  
\begin{itemize}
    \item Case 1: Applying a quiver mutation at a gauge node $k\neq 1$, we get the quiver diagram below with a potential $\tilde W=tr(BA_kA_{k-1})$
\begin{figure}[H]
    \centering
 \includegraphics[width=3.5in]{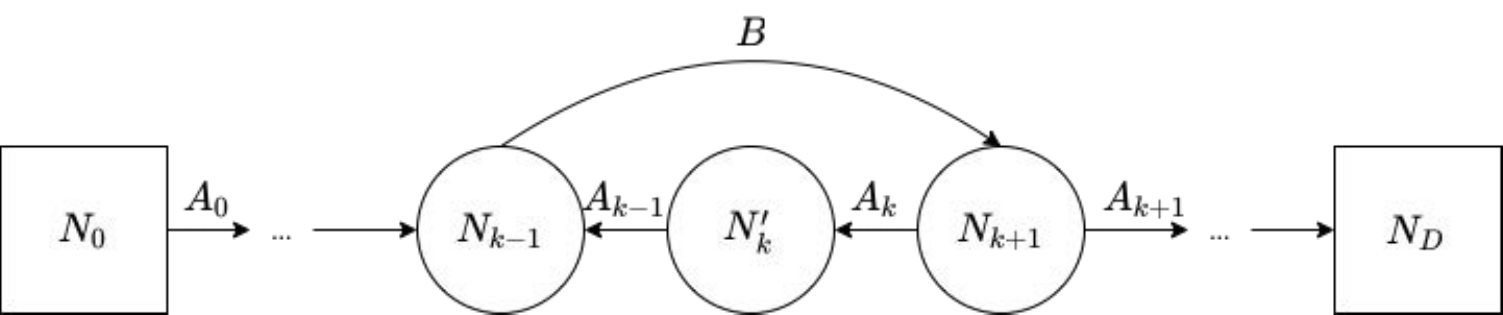}
    \caption{Quiver diagram by performing a quiver mutation at a gauge node $k\neq 1$ to the $A_n$-type quiver with two frame nodes.}
    \label{pic_dualoftautflag4}
\end{figure}
\noindent
This situation is almost the same with that of flag variety. 
By abuse of notation, we still denote the input data of the GIT quotient by $(\tilde V, \tilde G, \tilde \theta)$.
Under the standard group action \eqref{eqn:sec2GonV} and the phase of $\tilde G$  in \eqref{eqn:quiverphaseofdual}, we have
\begin{equation}
     \tilde V^{ss}(\tilde G)=\{B, A_i \text{ nondegenerate for } i\neq 0, k-1,k+1,\,\begin{bmatrix} A_{k}&A_{k+1} \end{bmatrix} \,\text{nondegenerate}  \}\,.
\end{equation}
We consider another quiver diagram in Figure \ref{figure4b} by deleting the arrow $k\rightarrow k-1$ in Figure \ref{pic_dualoftautflag4}.

\begin{figure}[H]
    \centering
   \includegraphics[width=3.4in]{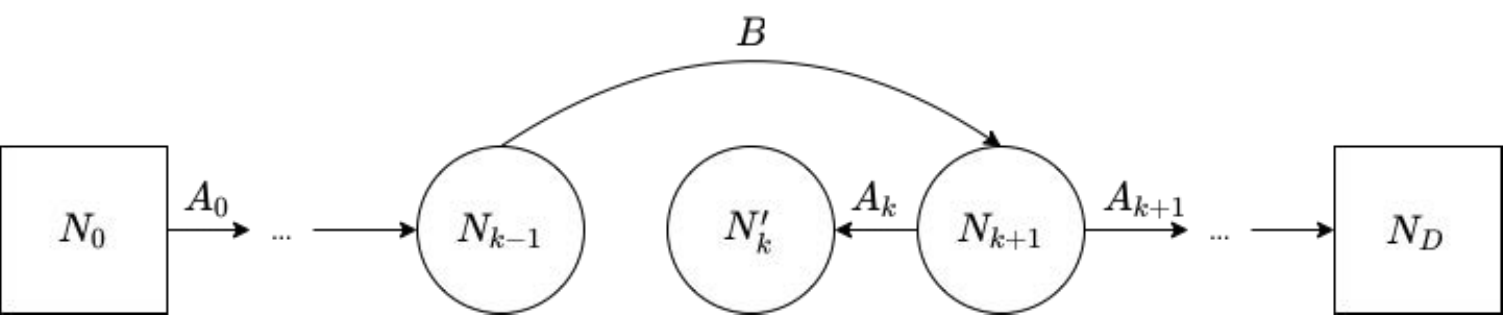}
    \caption{}
    \label{figure4b}
\end{figure}
\noindent Under the standard group action and in the phase \eqref{eqn:quiverphaseofdual}, the semistable locus consists of 
\begin{equation}
    \{A_i, B \text{ nondegenerate }\}\,.
\end{equation}
Denote the corresponding quiver vareity by $\mathcal X^3$, and notice that $\mathcal X^3$ is the total space of $ S_1^{\oplus N_0}\rightarrow \mathcal X^1$.
Consider the critical locus of the potential $Z^3:=Z(d\tilde W)$. It is equivalent to the following three equations
\begin{equation}
    BA_k=0,\,A_kA_{k-1}=0,\,A_{k-1}B=0\,.
\end{equation}
$Z^3\cap \tilde V^{ss}(\tilde G)=\{BA_k=0, A_{k-1}=0\}$, and $\mc Z^3:=Z^3\cap \tilde V^{ss}(\tilde G)\slash \tilde G=\{BA_k=0\}\sslash \tilde G$ is a subvariety of the quasiprojective variety $\mc X^3$. We note that $\mc Z^3=S_1^{\oplus N_0}\rightarrow \mc Z^1$ is the total space of $N_0$ copies of the tautological bundle over $\mc Z^1$.
\item Case 2: Applying a quiver mutation at the gauge node $k=1$, we get the quiver diagram in Figure \ref{figure4c} together with a potential $\tilde W=tr(BA_1A_0)$.
\begin{figure}[H]
    \centering
   \includegraphics[width=3.5in]{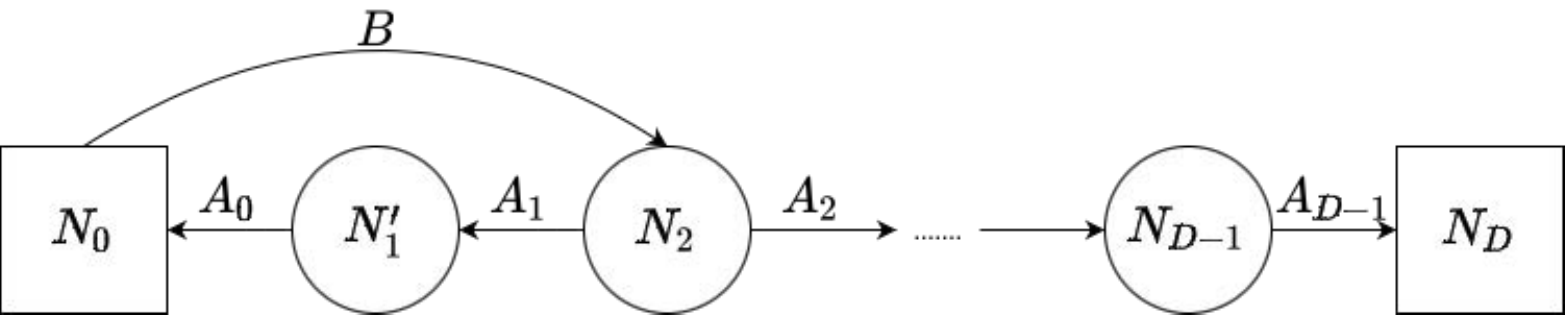}
    \caption{The quiver diagram by performing a quiver mutation at the gauge node $k=1$.}
    \label{figure4c}
\end{figure}
     \noindent
     Again, we denote the input data of this quiver variety by $(\tilde V, \tilde G,\tilde \theta)$. 
 In the phase \eqref{eqn:quiverphaseofdual},
one can check that
 \begin{equation}
     \tilde V^{ss}(\tilde G) =\{A_i,\, i\neq 0, 2 \text{ nondegenerate},\,\begin{bmatrix} A_{1}&A_{2} \end{bmatrix} \,\text{nondegenerate}\}\,.
 \end{equation}
 Consider the quiver diagram in Figure \ref{pic-dualoftautflag6} and its quiver variety. 
 \begin{figure}[H]
    \centering
   \includegraphics[width=3.5in]{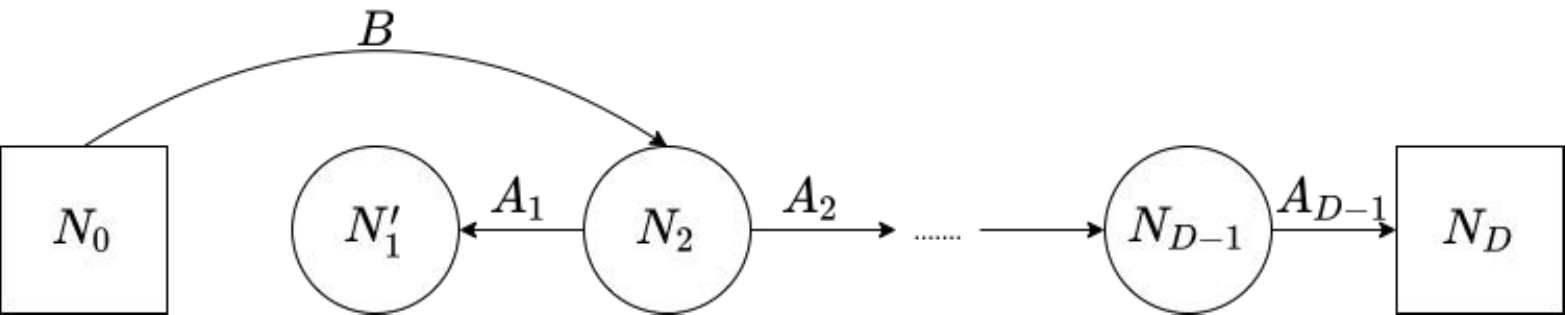}
    \caption{}
   \label{pic-dualoftautflag6}
\end{figure}
In the character \eqref{eqn:quiverphaseofdual}, one can find the semistable locus consists of 
 \begin{equation}
     \{A_1,\ldots,A_{D-1} \text{ nondegenerate}\}\,. 
 \end{equation}
 Denote the quiver variety by $\mathcal X^4$, and we notice that $\mathcal X^4$ is the total space of the vector bundle $S_2^{\oplus N_0}\rightarrow \mathcal X^2$.
 Consider the critical locus $Z^4=Z(d\tilde W)$ which is equivalent to 
\begin{equation}
    BA_1=0,\, A_1A_0=0, \,A_0B=0\,.
\end{equation}
Then we have 
\begin{equation}
    Z^4\cap \tilde V^{ss}(\tilde G)=\{A_0=0,\,BA_1=0, A_1,\ldots,A_{D-1} \text{ nondegenerate}\}\,.
\end{equation}
 The GIT quotient of critical locus
\begin{equation}
     \mc Z^4: =\tilde V^{ss}(\tilde G)\cap Z^4/\tilde G=\{ BA_1=0\}\sslash \tilde G\,
\end{equation} 
is a subvariety in $\mathcal X^4$.
\end{itemize}
\end{ex}

\section{Gromov-Witten invariants and wall-crossing theorem}\label{Sec:GW}
\subsection{Gromov-Witten invariants}
Since the examples we are interested in are all smooth quasiprojective varieties, we will only introduce the Gromov-Witten theory of smooth varieties. We refer to the beautiful book \cite{GW:mirrorsym} about the basic properties of GW theory.

\begin{dfn}
Let $\mc X$ be a smooth quasiprojective variety.
A stable map to $\mc X$ denoted by $(C, p_1,\ldots,p_n;f)$ consists of the following data:
\begin{itemize}
    \item $(C,p_1,\ldots,p_n)$ or simply $C$ is a connected reduced curve with $n\geq 0$ distinct marked non-singular points and at most ordinary double singular points; 
    \item every component of $C$ of genus 0 which is contracted by $f$ must have at least 3 special (marked or singular) points, and every component of $C$ of genus 1 which is contracted by $f$ must have at least 1 special point. 
    \end{itemize}
\end{dfn}
The class or the degree of a stable map $(C, p_1,\ldots, p_n; f)$ is defined as the homology class of the image $\beta=f_*[C]$.
For a fixed curve class $\beta\in H_2(\mc X, \Z)$, let $\overline{M}_{g,n}(\mc X, \beta)$ denote the stack of stable maps from $n$-marked and genus-g curves $C$ to $\mc X$ such that $f_*[C]=\beta$.
When $\mc X$ is projective, $\overline{M}_{g,n}(\mc X, \beta)$ is a proper separated DM stack and admits a perfect obstruction theory. Hence we can construct the virtual fundamental class $[\overline{M}_{g,n}(\mc X, \beta)]^{vir}\in A_{\text{vdim}}(\overline{M}_{g,n}(\mc X, \beta))$ where $\text{vdim}=\int_{\beta}c_1(X)+(\dim(\mc X)-3)(1-g)+n$. See  \cite{GW:LT,GW:BFintrinsic,GW:B}.

Define the universal curve by
\begin{equation}
\pi:\mc C_{g,n}\rightarrow \overline{M}_{g,n}(\mc X,\beta)\,,
\end{equation}
and $s_i$ are sections of $\pi$ sending each $(C, p_1,\ldots,p_n;f)$ to $p_i$. 
Let $\omega_\pi$ be the relative dualizing sheaf and $\mc P_i=s_i^*(\omega_\pi)$ be the cotangent bundle at the $i$-th marking.
Define the $\psi$-class by $\psi_i:=c_1(\mc P_i) \in H^2(\overline M_{g,n}(\mc X,\beta))$. 
Define evaluation maps by
\begin{align}
   ev_i:\overline M_{g,n}(\mc X,\beta)&\longmapsto \mc X\nonumber\\
   (C, p_1,\ldots,p_n;f)&\longmapsto f(p_i)\,.
\end{align}
Let $\gamma_1,\ldots,\gamma_n\in H^*(\mc X)$ be cohomology classes and $d_i$ $i=1,\ldots,n$ be some positive integers. The GW invariant is defined by
\begin{equation}\label{eqn:GWinv}
    \langle\tau_{d_1}\gamma_1,\ldots,\tau_{d_n}\gamma_n \rangle_{g,n,\beta}:=\int_{[\overline M_{g,n}(\mc X, \beta)]^{vir}}\prod_{i=1}^{n}\psi_i^{d_i}ev_i^*(\gamma_i)\,.
\end{equation}
Let $\alpha_0=1,\alpha_1,\ldots, \alpha_m\in H^*(\mc X)$ be a set of generators of cohomology group, and $\alpha^0,\alpha^2,\ldots, \alpha^m\in H^*(\mc X)$ be the Poincar\'e dual. 
The small $\mc J$-function of $\mc X$ which comprises genus-zero GW invariants is defined by
\begin{equation}\label{eqn:GWJ}
    \mc J^{\mc X}(q,\mathbf{t},z)= 
  \sum_{i=0}^{m} \sum_{(k\geq 0,\beta )} \alpha^i \langle \frac{\alpha_i}{z(z-\psi_{\bullet})}\mathbf{t}\ldots \mathbf{t}\rangle_{g,k+1,\beta}\frac{q^{\beta}}{k!} \,.
\end{equation}
where $\mathbf{t}\in H^{\leq 2}(\mc X)$.
When $\mc X$ is noncompact, the virtual class $[\overline M_{g,n}(\mc X, \beta)]^{vir}$ is not proper and the intersection again the virtual class doesn't make sense. 
However when $\mc X$ admits a torus action, denoted by $S$, then $S$ induces an action on $\overline M_{g,n}(\mc X, \beta)$ by sending a stable map $(C, p_1,\ldots, p_n;\,f)$ to $(C, p_1,\ldots, p_n;\, s\circ f)$ for each $s\in S$.
Suppose that the $S$-fixed substack $\overline M_{g,n}(\mc X, \beta)^S$ is proper. We can define the equivariant GW invariants as follows.
Let $H^*_S(\mc X):=H^*(\mc X\times_G EG)$ be equivariant cohomology of $\mc X$.
For $\gamma_1,\ldots,\gamma_n\in H^*_S(\mc X)$, the equivariant GW invariants are defined as,
\begin{equation}
      \langle\tau_{d_1}\gamma_1,\ldots,\tau_{d_n}\gamma_n \rangle_{g,n,\beta}^S:=\sum_{F}\int_{F}
      \frac{i_F^*\left(\prod_{i=1}^{n}\psi_i^{d_i}ev_i^*(\gamma_i)\right)}{e^S(N^{vir}_F)}\,,
\end{equation}
where the summation is over all torus fixed locus $F$, the map $i_F: F\rightarrow \overline M_{g,n}(\mc X, \beta)$ is the embedding, and $N_F^{vir}$ is the virtual normal bundle of $F$.
When $\mc X$ is projective, the nonequivariant limit is equal to the GW invariant defined in \eqref{eqn:GWinv}, see \cite{GW:GPequiv}.

Similarly,  we can define the equivariant small $\mc J$-function by changing each correlator in \eqref{eqn:GWJ} to equivariant invariant. We denote the equivariant smalll $\mc J$ function by $\mc J^{\mc X, S}( q,\mathbf{t},  z)$.

\subsection{genus-zero wall-crossing theorem}
In this subsection, we introduce the genus-zero wall-crossing theory in the context of Cheong, Ciocan-Fontanine, Kim, and Maulik  \cite{MR3126932,MR3412343,MR3272909,MR3586512}.  
We only involve the necessary parts for our purpose. 

Fix a valid input data for a GIT quotient $(V, G, \theta)$. 
\begin{dfn}\label{dfn:stablequasigraphmap}
A quasimap from $\mathbb P^1$ to $V\sslash G$ consists of the data $(P,u)$ where
\begin{itemize}
\item $P$ is a principle $G$-bundle on $\mbb P^1$,
\item $u$ is a section of the induced bundle $P\times_G V$ with the fiber $V$ on $\mbb P^1$.
\end{itemize}
The class of a quasimap is defined as $\beta\in \op{Hom}(\op{Pic}^G(V), \Z)$, such that for each line bundle $L\in \op{Pic}^G(V)$, 
\begin{equation}
    \beta(L)=\deg_{\mbb P^1}(u^*(P\times_G L))\,.
\end{equation}
\end{dfn}
\begin{dfn}\label{dfn:effectiveclass}
An element
$\beta\in \op{Hom}(\op{Pic}^G(V),\Z)$
is called an $I$-effective class if it is the class of a quasimap from $\mbb P^1$ to $V\sslash G$. These $I$-effective classes form a semigroup, denoted by $\op{Eff}(V,G,\theta)$.
\end{dfn}
\begin{dfn}
A quasimap $(P, u)$ from $\mbb P^1$ to $V\sslash G$ is stable if
\begin{enumerate}[label=\roman*]
    \item the set $B:=u^{-1}(V^{us})\subset \mbb P^1$ is finite, and points in $B$ are called base points of the quasimap. 
    \item $\mathbf L_\theta:=u^*(P\times_GL_\theta)$ is ample, where $L_{\theta}=V\times \C_\theta$.
\end{enumerate}
\end{dfn}
Denote the moduli stack of all stable qusimaps from $\mathbb P^1$ to $V\sslash G$ of class $\beta$ as $QG_{\beta}(V\sslash G)$. This moduli stack is the stable quasimap graph space in \cite{MR3126932}.
\begin{thm}[\cite{MR3126932}]\label{thm:quasimappot}
The stack $QG_{\beta}(V\sslash G)$ is a separated Deligne-Mumford stack of finite type, proper over the affine qouotient $Spec(H^0(V, \mc O_V)^G)$. It admits a canonical perfect obstruction theory if $V$ has at most lci singularities.
\end{thm}

Let  $[\zeta_0,\zeta_1]$ be coordinates on $\mathbb P^1$, then there is a standard $\C^*$ action, given by 
\begin{equation}
    t[\zeta_0,\zeta_1]=[t\zeta_0,\zeta_1], t\in \C^* \,.
\end{equation}
The $\C^*$-action on $\mbb P^1$ induces an action on $QG_{\beta}(V\sslash G)$. If a quasimap 
$(P,u)\in QG_{\beta}(V\sslash G)$ is $\C^*$-fixed, then all base points and the entire degree $\beta$ must be supported over the torus fixed points $[0:1]$ or $[1:0]$.

Consider the $\C^*$-fixed locus $F_\beta$ where everything is supported over the point $[0:1]\in \mbb P^1$ and the map $ev_\bullet: \mathbb P^1\backslash \{[0:1]\}\rightarrow V\sslash G$ is constant. 

 \begin{dfn}\label{dfn:smallIfunction}
Define the quasimap small $I$-function of a projective GIT quotient $V\sslash G$ as
 \begin{equation}
     I^{V\sslash G}(q,z)=1+\sum_{\beta\neq 0}q^{\beta} I_{\beta}^{V\sslash G}(z) \,\,, I_{\beta}^{V\sslash G}(z)=(ev_{\bullet})_*\left( \frac{[F_{\beta}]^{vir}}{e^{\C^*}(N_{F_{\beta}}^{vir})}  \right)\,,
 \end{equation}
 where the sum is over all $I$-effective classes of $(V, G,\theta)$.
 \end{dfn}

Assume $V\sslash G$ is quasi-projective, and $V$ admits a torus action ${S}$ which commutes with the action of $G$ on $V$. Hence the $S$ acts on $V\sslash G$.
The torus action is good if the torus fixed locus  $(V\sslash G)^S$ is a finite set.
There is an induced action of $S$ on $QG_\beta(V\sslash G)$ by sending $(P,u)\in QG_\beta(V\sslash G)$ to $s\circ u$ for each $s\in S$. 
 Moreover, the perfect obstruction theory is canonical $S$-equivariant \cite{MR3126932}. 
The same formula defines the equivariant small $I$-function of $V\sslash G$ as Definition \ref{dfn:smallIfunction} with all characteristic classes and pushforwards replaced by the equivariant version. 
We denote the equivariant quasimap small $I$-function of $V\sslash G$ by $I^{V\sslash G, S}(q, z)$.
\begin{thm}[\cite{MR3272909}]\label{thm:wallcrossing}
Assume $V\sslash G$ is a (quasi-)projective variety with a good torus action, and $V$ admits at most lci singularities. Then the following (equivariant) wall-crossing formula holds when $(V,G,\theta)$ is semi-positive,
\begin{equation}
    \mc J^{V\sslash G, S}( q, \mbf{t},z)=\frac{I^{V\sslash G,S}(q, z)}{I_0(q)}\,,
\end{equation} via  mirror map,
\begin{equation}
    \mbf{t}=\frac{I_1(q)}{I_0(q)} \in H^{\leq 2}(V\sslash G)\,,
\end{equation}
where the $I_0(q)$, $I_1(q)$ are defined as coefficients of $1$ and $z^{-1}$ in the following expansion,
\begin{equation}
    I^{V\sslash G,S}(q, z)=I_0(q)+\frac{I_1}{ z}(q)+O(\frac{1}{ z^2})\,.
\end{equation}
Further when $V\sslash G$ is Fano of index at least 2, $I_0(q)=1$ and $\mathbf{t}=0$. Hence in this situation
\begin{equation}
    \mc J^{V\sslash G,S}(q,z)=I^{V\sslash G,S}(q,z)\,.
\end{equation}
\end{thm}
\subsection{twisted  $I$-function}\label{sec:twistedIfunction}
Fix a valid input for GIT quotient $(V, G, \theta)$ and assume that $V$ has at most lci singularities.  Assume that $V\sslash G$ is quasiprojective  with a good torus action $S$. 
 Let $E$ be an equivariant $G$ bundle over $V$ with trivial $S$ action. Consider another torus $S^1$ acting on the fiber. Then we get an $S\times S^1$-equivariant bundle $E_G:=V^s\times_G E$ over $V\sslash_\theta G$. 
 
Assume $E$ is concave, which means $H^{0}(\mbb P^1, u^*(P\times_GE))=0$ for any stable quasimap $(P, u)$.
 Let
 \begin{equation}
     \pi: F_{\beta}(V\sslash_\theta G)\times \mbb P^1 \rightarrow F_{\beta}(V\sslash_\theta G)
 \end{equation}
 be the universal curve over $\C^*$-fixed locus supported at $[0:1]\in \mathbb{P}^1$, let $ \mathfrak{P}$ be the universal bundle over it, and let 
 \begin{equation}
     u: F_{\beta}(V\sslash_\theta G)\times \mbb P^1 \rightarrow \mathfrak P\times_G V
 \end{equation}
be the universal section.  
 Then $\mathfrak{P}\times_G E_G$ is a bundle over $\mathfrak P\times_G V$ and $R\pi_*(u^*)(\mathfrak{P}\times_G E)$ is an element in $K_{S\times G\times S^1}^\circ(F_{\beta})$.

 \begin{dfn}
Define the $S\times S^1$-equivariant $E$-twisted $I$-function as
 \begin{equation}
     I^{V\sslash_\theta G, S\times S^1, E}(q,z)=1+\sum_{\beta\neq 0}q^{\beta}I_\beta^{V\sslash_\theta G, S\times S^1, E}(z)\,,
 \end{equation}
where 
\begin{equation}
     I_{\beta}^{V\sslash_\theta G, S\times S^1, E}(z)=(ev_{\bullet})_*\left( \frac{[F_{\beta}]^{S\times S^1, vir}\cap e^{S\times S^1}(R^1\pi_*u^*E_G)}{e^{S\times S^1}(N_{F_{\beta}}^{vir})}  \right)\cdot e^{S\times S^1}(V^s\times_G E)\,.
     \end{equation}
 \end{dfn}
 Alternatively,  we may view the total space of the bundle $E_G\rightarrow V\sslash G$ as a quasiprojective variety $E_G$ admitting a good torus action $S^1\times S$.
The $S^1\times S$-fixed locus on the total space $E_G$ is the same as the $S$-fixed locus on $V\sslash G$.
 The twisted $I$-function defined above coincides with the quasimap $I$-function of the total space $E_G$, see \cite{MR3272909}.
 
 If the total space $E_G$ is semi-positive: $\beta(\det(T_{V\sslash G}))+\beta(\det{E})\geq 0$, the wall-crossing theorem holds for the above twisted $I$-function $I^{V\sslash G, S\times S^1, E}$.

\section{Quasimap  $I$-Functions}\label{Sec:quasimapI}
\subsection{Abelian/nonabelian correspondence for  $I$-functions}\label{Sec:abelian/nonabelianforI}
We will mainly follow the work of Rachel Webb about the abelian-nonablian correspondence to display the quasimap $I$-functions of our examples, see \cite{abeliannonabelian:Webb,abelianizationlef:Webb}. 

Fix a valid input $(V, G, \theta)$ for a GIT quotient  $V\sslash_\theta G$, and we assume that $V$ has at most lci singularities.
Let $T=(\C^*)^{r}\subseteq G$ denote the maximal torus of $G$ and let $W_T=N_T/T$ denote the Weyl group. 
We will use a letter $w$ to represent a general element in the Weyl group $W_T$ and its representative in $N_T$ with abuse of notation. 
Notice that any character $\theta$ of $G$ is also a character of $T$ by the inclusion $\chi(G)\subseteq \chi(T)$. 
We denote the semistable, stable and unstable locus of $V$ under the action of $T$ in character $\theta$ by $V^{ss}_\theta(T)$, $V^s_\theta(T)$, and $V^{us}_\theta(T)$.
We may instead use notations $V^{ss}(T)$, $V^s(T)$, and $V^{us}(T)$ when there is no confusion for the character. 
Assume that $V^{ss}(T)=V^{s}(T)$ and $T$ acts freely on $V^{ss}(T)$, so that we obtain a smooth variety $V\sslash T:=V^{s}(T)/T$. 
Assume that a torus $S$ acts on $V$ and commutes with the action of G. Hence $S$ acts on $V\sslash G$ and $V\sslash T$. 
Assume that the torus action $S$ on $V\sslash G$ and $V\sslash T$ is good. 

The relations of $H^*(V\sslash G)$ and $H^*(V\sslash T)$ are studied by \cite{Ab_nab:EG,Ab_nab:Mar,Kir}. 
The rational map $V\sslash G\dashrightarrow V\sslash T$ is realized as follows
\begin{equation}\label{diag}
\begin{tikzcd}
V^s(G)/T \arrow[hookrightarrow]{r}{j} \arrow[]{d}{g}
  & V^s(T)/T \\
V^s(G)/G & 
\end{tikzcd}
\end{equation}
The Weyl group $W$ acts on $V^s(G)\slash T$, and therefore on $H^*(V^s(G)\slash T)$. 
 The above diagram induces the following classical identification for the cohomology groups
\begin{equation}\label{eqn:cohabelianization}
    H_S^*(V\sslash G, \Q)\iso H_S^*(V^s(G)\slash T,\Q)^W\,.
\end{equation}
See \cite[Proposition 2.4.1]{abeliannonabelian:Webb} for a proof for chow groups.
For each $\gamma \in H_S^*(V\sslash G,\Q)$, we say $\tilde \gamma \in H_S^*(V\sslash T, \Q)^W$ is a lifting of $\gamma$ if $j^*(\tilde\gamma)=g^*(\gamma)$. Such a lifting is usually not unique. 
For each $\eta\in \chi(G)\subset \chi(T)$, there are line bundles $V\times \C_\eta\in \op{Pic}^G(V)$ and $V\times \C_\eta\in \op{Pic}^T(V)$.
Also there is a natural map from $\op{Pic}^G(V)$ to $\op{Pic}^T(V)$ by restriction.
Therefore we have the following commutative diagram 
\begin{equation}
  \begin{tikzpicture}
\node (VG) at (-1,1) {$\op{Pic}^G(V)$};
\node[right=of VG] (VT) {$\op{Pic}^T(V)$};
\node[below=of VT] (XT) {$\chi(T)$};
\node [below=of VG] (XG) {$\chi(G)$};
\draw[->] (VG)--(VT) node [] {};
\draw[->] (XT)--(VT) node [] {};
\draw[->] (XG)--(VG) node [] {};
\draw[->] (XG)--(XT) node [] {};
\end{tikzpicture}  
\end{equation}
Taking $\op{Hom}(-,\Z)$ to the above diagram, we get the following commutative diagram, 
\begin{equation}\label{diag:pic}
  \begin{tikzpicture}
\node (VT) at (-1,1) {$\op{Hom}(\op{Pic}^T(V),\Z)$};
\node[right=of VT] (VG) {$\op{Hom}(\op{Pic}^G(V),\Z)$};
\node[below=of VT] (XT) {$\op{Hom}(\chi(T),\Z)$};
\node [below=of VG] (XG) {$\op{Hom}(\chi(G),\Z)$};
\draw[->] (VT)--(VG) node [midway,above] {$r_1$};
\draw[->] (VT)--(XT) node [midway, right] {$v_1$};
\draw[->] (VG)--(XG) node [midway, right] {$v_2$};
\draw[->] (XT)--(XG) node [midway,above] {$r_2$};
\end{tikzpicture}  
\end{equation}
For any $\xi\in \chi(T)$, denote by $\mc L_\xi:=V^s(T)\times_T\C_\xi$ the line bundle over $V\sslash T$. For any $\tilde\beta\in \op{Hom}(\op{Pic}^T(V),\Z)$, denote by $\tilde\beta(\xi):=\tilde\beta(c_1(\mc L_\xi))$, and it also equals $v_1(\tilde \beta)(\xi)$ by the above diagram. 
 
\begin{lem}(\cite{MR3126932})
When $r_1$ restricts to $I$-effective classes $\op{Eff}(V, T,\theta)\subseteq \op{Hom}(\op{Pic}^T(V),\Z)$ in the source and $\op{Eff}(V, G, \theta)\subseteq \op{Hom}(\op{Pic}^G(V),\Z)$ in the target, it has finite fibers. 
\end{lem}
 
 \begin{thm}[\cite{abeliannonabelian:Webb}]\label{thm:abelian-nonabelianequiv}
The equivariant quasimap small $I$-functions of $V\sslash G$ and $V\sslash T$ satisfy 
 \begin{equation}\label{eqn:abeliannonabelianforI}
     g^*I_{\beta}^{V\sslash G,S}(z)=j^*\left[ \sum_{\tilde\beta\rightarrow \beta} \prod_{\rho}\frac{\prod_{k\leq \tilde\beta(\rho)}(c_1(\mathcal{L}_{\rho})+kz)}{\prod_{k\leq 0}(c_1(\mathcal{L}_{\rho})+kz)}I_{\tilde \beta}^{V\sslash T,S}(z)\right]
 \end{equation}
 where the sum is over all preimages $\tilde\beta$ of $\beta$ under the map $r_1$ in above diagram \eqref{diag:pic} and the product is over all roots $\rho$ of $G$. 
 \end{thm}

Consider a $G$-equivariant bundle $E$ over $V$, and assume $s$ is a $G$-equivariant regular section of the bundle $E\times V\rightarrow V$. Let $Z:=Z(s)\subseteq V$ be the zero locus of $s$.
Taking $Z$ into consideration, we can extend the diagram \eqref{diag} to
 \begin{equation}
\begin{tikzcd}
Z^{s}(G)\slash T \arrow[hookrightarrow]{r}{k} \arrow[]{d}{\phi}&V^s(G)/T \arrow[hookrightarrow]{r}{j} \arrow[]{d}{g} & V\sslash T \\
Z\sslash G \arrow[hookrightarrow]{r}{i}  & V\sslash G&
\end{tikzcd}
\end{equation}
 and extend the diagram \eqref{diag:pic}  to
 \begin{equation}
\begin{tikzcd}
\op{Hom}(\op{Pic}^T(Z),\Q) \arrow[hookrightarrow]{r}{k_*} \arrow[]{d}{}&\op{Hom}(\op{Pic}^T(V),\Q)  \arrow[]{d}{r_1}  \\
\op{Hom}(\op{Pic}^G(Z),\Q) \arrow[hookrightarrow]{r}{i_*}  & \op{Hom}(\op{Pic}^G(V),\Q)
\end{tikzcd}
\end{equation}
For each $\xi \in \chi(T)$, and $\beta\in \op{Home}(\op{Pic}^T(V),\Z)$,
denote
\begin{equation}
    C(\beta,\xi):=\frac{\prod_{k\leq 0}(c_1(\mathcal{L}_\xi)+kz)}{\prod_{k\leq \beta(\xi)}(c_1(\mathcal{L}_\xi)+kz)}\,.
\end{equation}
Assume that the torus $S$ acts on $Z$ and is good. 
The equivariant $I$-functions of $Z\sslash G$ and $V\sslash T$ satisfy the following relation, which can be viwed as an abelian/nonabelian lefchetz theorem. 
\begin{thm}[\cite{abeliannonabelian:Webb,abelianizationlef:Webb}]\label{thm:abeliannonabelianlefchetz}
 Assume that weights of $E$ with respect to the action of $T$ are $\epsilon_j$, for $j=1,\ldots,m$, and $\rho_i$ for $i=1,\ldots,r$ are roots of $G$. Then for a fixed $\delta\in \op{Hom}(\op{Pic}^G(V),\Q) $,  we have the following relation between $I$-functions of $Z\sslash G$ and $V\sslash T$,
 \begin{equation}
    \sum_{\beta\rightarrow \delta} \phi^*I_{\beta}^{Z\sslash G, S}(z)=\sum_{\tilde\delta\rightarrow \delta}
    \left(\prod_{i=1}^m C(\tilde \delta, \epsilon_i)^{-1} \right)\left(\prod_{i=1}^r  C(\tilde \delta, \rho_i)^{-1}  \right) k^*j^* I_{\tilde\delta}^{V\sslash T, S}(z)
 \end{equation}
  where $\tilde\delta\in \op{Hom}(\op{Pic}^T(V),\Q)$ are preimages of $\delta$ via $r_1$, and $\beta\in \op{Hom}(\op{Pic}^G(Z),\Q)$.
\end{thm}

 \subsubsection{$I$-functions of $A_n$-type quivers}\label{subsec:Ifunctionbm}
 In this section, we mainly apply the abelian-nonabelian correspondence for $I$-functions to $A_n$-type quivers in Example \ref{ex:flag} and generalized $A_n$-type quivers in Example \ref{ex:tautflag}. We would use the notations in Section \ref{Sec:quiver} about quiver varieties.
 Let $(V, G, \theta)$ be the input data for $F(N_1,\ldots, N_D)$. There is a good torus action $S=(\C^*)^{N_D}$ on $F(N_1,\ldots, N_D)$ as follows,
\begin{equation}\label{eqn:Saction}
    (A_1,\ldots,A_{D-1})t=(A_1,\ldots,A_{D-1}t^{-1}), \,\,t=(t_1,\ldots,t_{N_D})\in S\,,
\end{equation}
where we have identified an element $t\in (\C^*)^{N_D}$ with a diagonal matrix $\op{diag}(t_1,\ldots,t_{N_D}) \in GL(N_D)$.
One can check that the $S$-fixed points are in one to one correspondence with the following sequences of subsets,
\begin{equation}\label{fixed:flag}
    \mathfrak{F}^{bm}=\{\vec C_{[N_1]}\subset \vec C_{[N_2]}\subseteq\cdots \subseteq [N_D]\}\,,
\end{equation}
where $[N_i]$ denotes a set of integers $\{1,\ldots, N_i\}$ and $\vec C_{[N_i]}=\{f_1<f_2<\ldots<f_{N_i}\}$ denotes a set of arbitrary distinct $N_i$ integers in $[N_{i+1}]$.
We denote the equivariant parameters of the $S$-action by $\lambda_1,\ldots, \lambda_{N_D}$.

Let $i_Q: Q\rightarrow F(N_1,\ldots,N_D)$ be an inclusion map from an $S$-fixed point $Q\in \mathfrak F^{bm}$ to $F(N_1,\ldots, N_D)$. Then the localization theorem of cohomology \cite{ATIYAH19841} states that
\begin{equation}\label{eqn:localization}
    {H}_{S}^*(F(N_1,\ldots, N_D))\iso \bigoplus_{Q\in \mathfrak F^{bm}}\op{H}^*_S(Q)\,.
\end{equation}

A stable quasimap $(\mathcal P, \sigma)$ from $\mathbb{P}^1$ to $F(N_1,\ldots,N_D)$ is equivalent to the following ingredients: 
 \begin{itemize}
     \item a bundle
     \begin{equation}
         \oplus_{i=1}^{D-2}\oplus_{I=1}^{N_i}\oplus_{J=1}^{N_{i+1}}\mc O_{\mbb{P}^1}(m_I^{(i)}-m_J^{(i+1)})\bigoplus\oplus_{I=1}^{N_{D-1}}\mc O_{\mbb P^1}(m_I^{(D-1)})^{\oplus N_D}
     \end{equation}
      over $\mathbb{P}^1$, where $m_{I}^{(i)}\in \Z$;
     \item a section $\sigma$ of the above bundle which maps $\mbb P^1$ to $V^{ss}(G)$ except for finite many points.
 \end{itemize}
 Denote $\vec m^{(i)}:=(m^{(i)}_1 \ldots m^{(i)}_{N_i})$. 
Then those integer vectors $(\vec m^{(1)}\ldots \vec m^{(D-1)})\in \oplus_{i=1}^{D-1}\Z^{N_i}$ that make the above two items hold must satisfy the following conditions,
\begin{enumerate}[label=\roman*]
     \item for a fixed $i$ and for each $I\in \{1,\ldots, N_i\}$, $\exists J_I\in \{1,\ldots, N_{i+1}\}$, s.t. $m_I^{(i)}-m_{J_J}^{(i+1)}\geq 0$, $i=1,\ldots, D-2$,
     \item for each fixed $i$ and for the index $(I,J_I)$ in the above item,  the $N_i\times N_{i+1}$ matrix whose $(I,J_I)$ entries are 1  and all other entries are  0 is nondegenerate.
 \end{enumerate}
 We denote the set of those $(\vec m^{(i)})$ by $\op{Eff}^T_{bm}$.
 
 The map $r_1: Hom(\op{Pic}^T(V),\Z)\rightarrow Hom(\op{Pic}^G(V),\Z)$ in commutative diagram \eqref{diag:pic} sends each $(\vec m^{(i)})\in \op{Eff}^T_{bm}$ to $(\abs{\vec m^{(i)}})\in \Z^{D-1}$ where $\abs{\vec m^{(i)}}=\sum_{I=1}^{N_i}m^{(i)}_I$. 
By the definition of effective classes, the images of $\op{Eff}^T_{bm}$ via $r_1$ are all $I$-effective classes for $F(N_1,\ldots, N_D)$ which we denote by $\op{Eff}_{bm}$.

Let $x^{(i)}_I$, $I=1,\ldots, N_i$, be chern roots of the dual bundles of the universal bundles $S_i$: 
\begin{equation}
    S_1\subseteq S_2\ldots\subseteq S_{D-1}\subseteq \C^{N_D}\,.
\end{equation}

 \begin{lem}\label{lem:Ifuncofflagbm}
The equivariant quasimap small $I$-function of a flag variety $F(N_1,\ldots,N_D)$, pulled back to $H_S^*(V^{ss}(G)/T,\Q)^{W_T}$ via $g^*$ in diagram \eqref{diag}, is given as follows,
\begin{align}
    &g^*I^{F,S}(\vec q,z)=\sum_{(m^{i})\in \op{Eff}_{bm}}I^{F,S}_{(m^{i})}(z)\prod_{i=1}^{D-1}q_{i}^{m^{(i)}}\,,\\ 
    & I^{F,S}_{(m^{i})}(z)=\sum_{\substack{\abs{\vec m^{(i)}}=m^{(i)}\\ (\vec m^{(i)})\in \op{Eff}_{bm}^T}} I^{F,S}_{(\vec m^{i})}(z)\,,
\end{align}
where 
\begin{equation}\label{I:flag}
I^{F,S}_{(\vec m^{i})}(z)= \prod_{i=1}^{D-1}\prod_{I\neq J}^{N_i}
 \frac{\prod_{l\leq m_I^{(i)}-m_J^{(i)}}(x_I^{(i)}-x_{J}^{(i)}+zl)}
 {\prod_{l\leq 0}(x_I^{(i)}-x_{J}^{(i)}+zl)}\prod_{I=1}^{N_i}\prod_{J=1}^{N_{i+1}}
 \frac{\prod_{l\leq 0}(x_I^{(i)}-x_J^{(i+1)}+lz)}
 {\prod_{l\leq m_I^{(i)}-m_J^{(i+1)}}(x_I^{(i)}-x_J^{(i+1)}+lz)}\,.
 \end{equation}
In the above formula, in order to simplify the expression, we have made the assumptions that $x^{(N)}_J$ are equivariant parameters $\lambda_{J}$ and $\vec m^{(D)}_J=\vec 0$. 
 \end{lem}
 \begin{proof}
 We only have to prove that the expression $I^{F,S}_{(\vec m^{i})}(z)$ is equal to the right hand side of the Equation \eqref{eqn:abeliannonabelianforI}. Roots $\rho$ of the group $G$ can viewed as elements of $\chi(T)=\prod_{i=1}^{D-1}\chi((\C^*)^{N_i})$. 
 Let $\vec e^{(i)}_I=(0\,\ldots\, 1\,\ldots\,0)\in \Z^{N_i}$ be the unit vector with the $I$-th component being 1 and all other components being zero. 
 It can be viewed as an element in $\oplus_{i=1}^{D-1}\Z^{N_i}$ by natural embedding: $(\vec 0\ldots \vec e^{(i)}_I\ldots\vec 0)$.  
 Roots of $G$ can then be expressed as $\vec e^{(i)}_I-\vec e^{(i)}_J$ in $\chi(T)=\oplus_{i=1}^{D-1}\Z^{N_i}$, for $i=1,\ldots, D-1$, and $I, J\in [N_i]$.  
 Then we can find the first factor in the expression of $I^{F,S}_{(\vec m^{i})}(z)$ is the product over roots in \eqref{eqn:abeliannonabelianforI}, and the second factor is $j^*I^{V\sslash T, S}_{\tilde \beta}(z)$. 
 \end{proof}

We want to mention that 
a flag variety's equivariant small $I$-function has been investigated in \cite{abelian/nonabelian:CKB} via a different method.

 Let's continue with the equivariant quasimap small $I$-function of the total space of $N_0$ copies of the tautological bundle over a flag variety: $S_1^{\oplus N_0}\rightarrow F(N_1,\ldots, N_D)$. We still use $(V, G, \theta)$ to denote the input data of the GIT quotient by abuse of notation. 
 Now we have to consider the torus action $S^2=(\C^*)^{N_0}\times (\C^*)^{N_D}$ on the quasiprojective variety  $S_1^{\oplus N_0}\rightarrow F(N_1,\ldots, N_D)$ in the following way,
\begin{equation}\label{eqn:torusS2}
    (A_0\ldots, A_{D-1})(s,t)=(sA_0,A_1,\ldots, A_{D-1}t^{-1})\,,\, (s, t)\in (\C^*)^{N_0}\times (\C^*)^{N_D}\,.
\end{equation}
The $S^2$-fixed locus in $S_1^{\oplus N_0}\rightarrow F(N_1,\ldots, N_D)$ is exactly the same with the $S$-fixed locus in $F(N_1,\ldots,N_D)$ which is $\mathfrak F^{bm}$ in the Equation \eqref{fixed:flag}. 
Furthermore, one can check that the $I$-effective classes of  $S_1^{\oplus N_0}\rightarrow F(N_1,\ldots, N_D)$ are also the same with those of $F(N_1,\ldots,N_D)$.

Denote equivariant parameters of the torus $(\C^*)^{N_0}$ by $\eta_A$, $A=1,\ldots, N_0$ and equivariant parameters of the torus $(\C^*)^{N_D}$ by $\lambda_i,\, i=1,\ldots,{N_D}$. 
\begin{lem}\label{lem:Ifuncoftautflag}
Applying the abelian-nonabelian correspondence for the $I$-function in Theorem \ref{thm:abelian-nonabelianequiv}, we obtain that the equivariant quasimap small
 $I$-function of $S_1^{\oplus N_0}\rightarrow F(N_1,\ldots,N_D)$ pulled back to $H^*(V^s(G)\slash T)^{W_T}$ is as follows,
 \begin{align}
      &g^*I^{tF,S^2}(\vec q,z)=\sum_{( m^{(i)})\in \op{Eff}_{bm}} I^{tF, S^2}_{( m^{(i)})}(z) \prod_{i=1}^{D-1}q_i^{m^{(i)}}\,,\\
&
     I^{tF, S^2}_{( m^{(i)})}(z)=\sum_{\substack{\abs{\vec m^{(i)}}=m^{(i)}\\ (\vec m^{(i)})\in \op{Eff}_{bm}^T}}I^{tF, S^2}_{( \vec m^{(i)})}(z)\,,
 \end{align}
 where 
\begin{align}\label{eqn:Itautflag}
I^{tF, S^2}_{(\vec m^{(i)})}(z)=
 &\prod_{i=1}^{D-1}\prod_{I\neq J}^{N_i}
 \frac{\prod_{l\leq m_I^{(i)}-m_J^{(i)}}(x_I^{(i)}-x_{J}^{(i)}+lz)}
 {\prod_{l\leq 0}(x_I^{(i)}-x_{J}^{(i)}+lz)}
 \prod_{A=1}^{N_0}\prod_{J=1}^{N_1}
 \frac{\prod_{l\leq 0}(\eta_A-x_J^{(1)}+lz)}{\prod_{l\leq -m_J^{(1)}}(\eta_A-x_J^{(1)}+lz)}\nonumber\\
 &\prod_{i=1}^{D-1}\prod_{I=1}^{N_i}\prod_{J=1}^{N_{i+1}}
 \frac{\prod_{l\leq 0}(x_I^{(i)}-x_J^{(i+1)}+lz)}
 {\prod_{l\leq m_I^{(i)}-m_J^{(i+1)}}(x_I^{(i)}-x_J^{(i+1)}+lz)}\,.
\end{align}
\end{lem}
\begin{proof}
The proof is exactly the same with that for $F(N_1,\ldots,N_D)$ in Lemma \ref{lem:Ifuncofflagbm}, which is omitted.
\end{proof}
A flag variety is always semi-positive, so the wall-crossing Theorem \ref{thm:wallcrossing} holds for $F(N_1,\ldots, N_D)$. 
 In particular, a  flag variety is Fano of index at least 2 if 
\begin{equation}
    N_1<N_2\ldots <N_D\,.
\end{equation} 
In this situation, the wall-crossing Theorem \ref{thm:wallcrossing} holds with trivial variable change.
The local target $S_1^{\oplus N_0}\rightarrow F(N_1,\ldots, N_D)$ is semi-positive when $N_2\geq N_0$ and Fano of index at least 2 when 
\begin{equation}
    1\leq N_1<\ldots<N_D\,, \,\,
    N_0+2\leq N_2\,. 
    \end{equation}
In this situation, the wall-crossing theorem holds for $S_1^{\oplus N_0}\rightarrow F(N_1,\ldots, N_D)$ via the trivial mirror map. 

\subsubsection{$I$-functions of varieties after quiver mutations}\label{subsec:Ifunctionam}
Revisit Example \ref{ex:dualofflag} and use the notations in that example. 
$\mc Z^1$ and $\mc X^2$ are the two varieties we have constructed after we perform a quiver mutation at a gauge node $k\neq 1$ and the node $k=1$ respectively to the quiver in Figure \ref{diag:flag}. 

Denote the input data of the quiver variety $\mc X^1$ corresponding to the quiver in Figure \ref{pic:dualofflag2D}  by $(\tilde V, \tilde G, \tilde\theta)$.
There is a good torus action $S=(\C^*)^{N_D}$ on $\mc X^1$ coming from the rightmost frame node as follows.
For any $(A_1,\ldots, A_{k-1},B,A_{k+1},\ldots, A_{D-1})\in \tilde V$, and $t\in S$, 
\begin{equation}\label{eqn:torusactionondualflag}
    (A_1,\ldots, A_{k-1},B,A_{k+1},\ldots, A_{D-1})t\cdot=
    (A_1,\ldots, A_{k-1},B,A_{k+1},\ldots, A_{D-1}t^{-1})\,.
\end{equation}
The above torus action commutes with $G$-action and preserves the relation $BA_{k}=0$, so it acts on $\mc Z^1$.
One can check that under the above torus action, torus fixed points in $\mc Z^1$ are in one to one with sequences of subsets,
\begin{align}\label{eqn:fixedpointsam}
   \mathfrak{F}^{am}=\{ \vec C_{[N_1]}\subset \cdots\subset \vec C_{[N_{k-1}]}\subset \vec C_{[N_{k+1}]}\subset\cdots\subseteq [N_D],\,\vec C_{[N_k']}\subseteq \vec C_{N_{k+1}},\, \vec C_{[N_k']}\cap \vec C_{[N_{k-1}]}=\emptyset\}\,.
\end{align}
The last requirement $\vec C_{[N_k']}\cap \vec C_{[N_{k-1}]}=\emptyset$ arises to make those points in $\mc Z^1$.

 A stable quasimap $(\mathcal P, \sigma)$ from $\mathbb{P}^1$ to $\mc X^1$ is equivalent to the following ingredients,
 \begin{itemize}
     \item a bundle 
     \begin{align}
         &\oplus_{i\neq k-1,k}^{D-2}\oplus_{I=1}^{N_i}\oplus_{J=1}^{N_{i+1}}\mc O_{\mbb{P}^1}(m_I^{(i)}-m_J^{(i+1)})
     \bigoplus \oplus_{I=1}^{N_{k-1}}\oplus_{J=1}^{N_{k+1}}\mc O_{\mbb P^1}(m_{I}^{(k-1)}-m_{J}^{(k+1)}) \nonumber\\
     &\bigoplus\oplus_{I=1}^{N_{k+1}}\oplus_{J=1}^{N_{k}'}\mc O_{\mbb P^1}(m_{I}^{(k+1)}-m_{J}^{(k)}) 
     \bigoplus\oplus_{I=1}^{N_{D-1}}\mc O_{\mbb P^1}(m_I^{(D-1)})^{\oplus N_D}
     \end{align}
     over $\mathbb{P}^1$, where $m_{I}^{(i)}\in \Z$,
     \item a section $\sigma$ of the above bundle such that it maps $\mbb P^1$ to $V^{ss}(G)$ except for finite points. 
 \end{itemize}
By the similar reasoning with the last subsection \ref{subsec:Ifunctionbm},
the vector $(\vec m^{(1)},\ldots, \vec m^{(D-1)})\in \oplus_{i=1}^{k-1}\Z^{N_i}\oplus \Z^{N_k'}\oplus_{i=k+1}^{D-1}\Z^{N_i}$ that makes the above two items hold if and only if the following conditions are satisfied.
\begin{enumerate}[label=\roman*]
    \item For each $i\in \{1,\ldots, D-1\}\backslash \{k-1,k\}$ and $I\in [N_i]$, $\exists J_I\in [N_{i+1}]$,
    such that $m_{I}^{(i)}-m^{(i+1)}_{J_I}\geq 0$; 
    for each $I\in [N_{k-1}]$, $\exists\, J_I\in [N_{k+1}]$, such that $m_I^{(k-1)}-m_{J_I}^{(k+1)}\geq 0$; 
    for each $I\in [N_{k}']$, $\exists\, J_I\in [N_{k+1}]$, such that $-m_{I}^{(k)}+m_{J_I}^{k+1}\geq 0$.
    \item For each $i\in \{1,\ldots, D-1\}$, and the related index $(I,J_I)$ above, the matrix whose $(I,J_I)$-entries are 1 and all other entries are zero is nondegenerate.
\end{enumerate}
We denote the set of vectors $(\vec m^{(1)},\ldots, \vec m^{(D-1)})$ satisfying the above two conditions by $\op{Eff}^{T}_{am}$.

The map $r_1: Hom(\op{Pic}^T(V),\Z)\rightarrow Hom(\op{Pic}^G(V),\Z)$ sends each $(\vec m^{(1)},\ldots, \vec m^{(D-1)})$ to $(\abs{\vec m^{(1)}},\ldots, \abs{\vec m^{(D-1)}})$, and the images of $\op{Eff}^{T}_{am}$ are $I$-effective classes of $\mc X^1$, which we denote  by $\op{Eff}_{am}$.

 \begin{lem}\label{lem:IfunctionofZ1}
Applying the abelian-nonabelian correspondence in Theorem \ref{thm:abeliannonabelianlefchetz}, we obtain the equivariant quasimap small $I$-function of $\mc Z^1$.
 \begin{align}\label{eqn:Iflagam1}
     &g^*I^{\mc Z^1,S}(\vec q', z)=\sum_{m^{(i)}\in \op{Eff}_{am}}I^{\mc Z^1, S}_{(m^{(i)})}(z)\prod_{i=1}^{D-1}(q_{i}')^{m^{(i)}}\,,\\
     & I^{\mc Z^1, S}_{(m^{(i)})}(z)=\sum_{\substack{\abs{\vec m^{(i)}}=m^{(i)}\\
    (\vec m^{(i)})\in \op{Eff}^T_{am} }} I^{\mc Z^1, S}_{(\vec m^{(i)})}(z)\,,
 \end{align}
 with 
\begin{align}\label{eqn:Iflagam}
    I^{\mc Z^1,S}_{(\vec m^{(i)})}(z)=&\prod_{\substack{i=1,\\i\neq k-1,k}}^{D-1}\prod_{I\neq J}^{N_i}
    \frac{\prod_{l\leq m_{I}^{(i)}-m_{J}^{(i)}}(x_{I}^{(i)}-x_{J}^{(i)}+lz)}
    {\prod_{l\leq 0}(x_{I}^{(i)}-x_{J}^{(i)}+lz)}
    \prod_{I=1}^{N_i}\prod_{J=1}^{N_{i+1}}
    \frac{\prod_{l\leq 0}(x_{I}^{(i)}-x_{J}^{(i+1)}+lz)}
    {\prod_{l\leq m_{I}^{(i)}-m_{J}^{(i+1)}}(x_{I}^{(i)}-x_{J}^{(i+1)}+lz)}\nonumber\\
    &
    \prod_{I\neq J}^{N_{k-1}}
    \frac{\prod_{l\leq m_{I}^{(k-1)}-m_{J}^{(k-1)}}(x_{I}^{(k-1)}-x_{J}^{(k-1)}+lz)}
    {\prod_{l\leq 0}(x_{I}^{(k-1)}-x_{J}^{(k-1)}+lz)}
    \prod_{I=1}^{N_{k-1}}\prod_{J=1}^{N_{k+1}}
    \frac{\prod_{l\leq 0}(x_{I}^{(k-1)}-x_{J}^{(k+1)}+lz)}
    {\prod_{l\leq m_{I}^{(k-1)}-m_{J}^{(k+1)}}(x_{I}^{(k-1)}-x_{J}^{(k+1)}+lz)} \nonumber\\
    &
    \prod_{I\neq J}^{N_{k}'}
    \frac{\prod_{l\leq m_{I}^{(k)}-m_{J}^{(k)}}(x_{I}^{(k)}-x_{J}^{(k)}+lz)}
    {\prod_{l\leq 0}(x_{I}^{(k)}-x_{J}^{(k)}+lz)}
    \prod_{I=1}^{N_{k+1}}\prod_{J=1}^{N_{k}'}
    \frac{\prod_{l\leq 0}(x_{I}^{(k+1)}-x_{J}^{(k)}+lz)}
    {\prod_{l\leq m_{I}^{(k+1)}-m_{J}^{(k)}}(x_{I}^{(k+1)}-x_{J}^{(k)}+lz)}\nonumber\\
    &
    \prod_{I=1}^{N_k'}\prod_{J=1}^{N_{k-1}}
    \frac{\prod_{l\leq -m_{I}^{(k)}+m_{J}^{(k-1)}}(-x_{I}^{(k)}+x_{J}^{(k-1)}+lz)}
    {\prod_{l\leq 0}(-x_{I}^{(k)}+x_{J}^{(k-1)}+lz)}\,,
\end{align}
where we have made the assumptions that $x^{(D)}_J=\lambda_J$ and $m^{(D)}_J=0$.
\end{lem}
\begin{proof}
The variety $\mc Z^1$ is the zero locus of a section of the bundle $S_{k-1}^\vee\otimes S_{k}^\vee$ over $\mc X^1$.
The group $GL(N_i)$ for $i\neq k,k-1$ act trivially on the bundle.
Utilizing the notations in the proof of Lemma \ref{lem:Ifuncofflagbm}, the weights of the bundle $S_{k-1}^\vee\otimes S_{k}^\vee$ under the action of $T$ are vectors  $(\vec 0\ldots \vec e^{(k-1)}_{I}\, -\vec e^{(k)}_{J}\ldots\vec 0)$, for $I\in [N_{k-1}]$ and $J\in [N_{k}']$.
Then we are able to get the $I$-function by Theorem \ref{thm:abeliannonabelianlefchetz}.
\end{proof}
After we apply a quiver mutation to the gauge node $k=1$, we have obtained the quiver variety $\mc X^2$ in Example \ref{ex:dualofflag}.
Similar with $\mc Z^1$, $\mc X^2$ admits a good torus action $S$ as \eqref{eqn:torusactionondualflag}, and the torus fixed points are those in $ \mathfrak{F}^{am}$ in $\eqref{eqn:fixedpointsam}$ with  $ \vec C_{[N_{k-1}]}=\emptyset$. 
The semigroup of $I$-effective classes is also $\op{Eff}_{am}$, and the lifting to $Hom(\op{Pic}^T(\tilde V),\Z)$ via the $r_1$ map is $\op{Eff}^T_{am}$ by omitting a condition: for any $I\in [N_{k-1}]$, $\exists\, J_I\in [N_{k+1}]$, such that $m_I^{(k-1)}-m_{J_I}^{(k+1)}\geq 0$.

The equivariant quasimap small $I$-function of $\mc X^2$ is as follows, 
\begin{align}
    &I^{\mc X^2,S}(\vec q',z)=\sum_{(m^{(i)})\in  \op{Eff}_{am}} I^{\mc X^2,S}_{(m^{(i)})}(z)\prod_{i=1}^{D-1}(q_i')^{m^{(i)}}\,,\\
    &I^{\mc X^2,S}_{(m^{(i)})}(z) =\sum_{\substack{\abs{\vec m^{(i)}}=m^{(i)}\\
     (\vec m^{(i))})\in \op{Eff}_{am}^T}} I^{\mc X^2,S}_{(\vec m^{(i)})}(z)\,,
\end{align}
and 
 \begin{align}
     I^{\mc X^2,S}_{(\vec m^{(i)})}(z)=&\prod_{\substack{i=2}}^{D-1}\prod_{I\neq J}^{N_i}
    \frac{\prod_{l\leq m_{I}^{(i)}-m_{J}^{(i)}}(x_{I}^{(i)}-x_{J}^{(i)}+lz)}
    {\prod_{l\leq 0}(x_{I}^{(i)}-x_{J}^{(i)}+lz)}
    \prod_{I=1}^{N_i}\prod_{J=1}^{N_{i+1}}
    \frac{\prod_{l\leq 0}(x_{I}^{(i)}-x_{J}^{(i+1)}+lz)}
    {\prod_{l\leq m_{I}^{(i)}-m_{J}^{(i+1)}}(x_{I}^{(i)}-x_{J}^{(i+1)}+lz)}\nonumber\\
    &
    \prod_{I\neq J}^{N_{1}'}
    \frac{\prod_{l\leq m_{I}^{(1)}-m_{J}^{(1)}}(x_{I}^{(1)}-x_{J}^{(1)}+lz)}
    {\prod_{l\leq 0}(x_{I}^{(1)}-x_{J}^{(1)}+lz)}
    \prod_{I=1}^{N_{2}}\prod_{J=1}^{N_{1}'}
    \frac{\prod_{l\leq 0}(x_{I}^{(2)}-x_{J}^{(1)}+lz)}
    {\prod_{l\leq m_{I}^{(2)}-m_{J}^{(1)}}(x_{I}^{(2)}-x_{J}^{(1)}+lz)}\,.
 \end{align}
 
Performing a quiver mutation at one gauge node of the quiver diagram of
 $S_1^{\oplus N_0}\rightarrow F(N_1,\ldots,N_D)$, we have constructed two varieties $\mc Z^3$ and $\mc Z^4$ in Example \ref{ex:dualoftautflag} depending on nodes we apply the mutation to.
We find $\mc Z^3$ is the total space of $N_0$-copies of the tautological bundle $S_1$ over $\mc Z^1$. 
The good torus action on $\mc Z^3$ is $S^2=(\C^*)^{N_0}\times (\C^*)^{N_D}$ as \eqref{eqn:torusS2}, which has torus fixed locus $\mathfrak{F}^{am}$ in  \eqref{eqn:fixedpointsam}.
The semigroup of $I$-effective classes of $\mc Z^3$ is $\op{Eff}_{am}$. 
By the same reasoning as Lemma \ref{lem:IfunctionofZ1}, we get the equivariant small I-function of $\mc Z^3$.
\begin{lem}\label{lem:Itauflagam}
The equivariant quasimap small $I$-function of $\mc Z^3$ can be written as
\begin{align}\label{eqn:Itautflagamk}
    &I^{\mc Z^3,S^2}(\vec q',z)=\sum_{(\vec m^{(i))}\in \op{Eff}^T_{am}}
    I^{\mc Z^1, S}_{(\vec m^{(i)})}(z)
    \prod_{A=1}^{N_0}\prod_{J=1}^{N_1}
 \frac{\prod_{l\leq 0}(\eta_A-x_J^{(1)}+lz)}{\prod_{l\leq -m_J^{(1)}}(\eta_A-x_J^{(1)}+lz)}\prod_{i=1}^{D-1}(q')^{\abs{\vec m^{(i)}}}\,.
\end{align}
\end{lem}
The situation for $\mc Z^4$ is a little different.
We may view $\mc Z^4$ as a subvariety in $S_2^{\oplus N_0}\rightarrow \mc X^2$ defined by $A_0B=0$. 
The semigroup of $I$-effective classes of $S_2^{\oplus N_0}\rightarrow \mc X^2$ is also $\op{Eff}_{am}$. 
The local target $S_2^{\oplus N_0}\rightarrow \mc X^2$ has good torus action $S^2$ whose torus fixed locus is $\mathfrak{F}^{am}$. 

\begin{lem}\label{lem:IofZ4}
 Applying the abelian-nonabelian correspondence in Theorem \ref{thm:abeliannonabelianlefchetz}, we get the $I$-function of $\mc Z^4$,
\begin{equation}
    g^*I^{\mc Z^4, S^2}(\vec q',z)=\sum_{(\vec m^{(i)})\in \op{Eff}_{am}^T}I^{\mc Z^4, S^2}_{(\vec m^{(i)})}(z) \prod_{i=1}^{D-1}(q_i')^{\abs{\vec m^{(i)}}}\,,
\end{equation}
and 
\begin{equation}\label{eqn:Itautflagam1}
   I^{\mc Z^4, S^2}_{(\vec m^{(i)})}(z)=I^{\mc X^2, S}_{(\vec m^{(i)})}(z)
    \prod_{A=1}^{N_{0}}\prod_{J=1}^{N_{2}}
    \frac{\prod_{l\leq 0}(\eta_A-x_{J}^{(2)}+lz)}
    {\prod_{l\leq -m_{J}^{(2)}}(\eta_A-x_{J}^{(2)}+lz)}
    \prod_{I=1}^{N_1'}
    \frac{\prod_{l\leq -m_{I}^{(1)}}(-x_{I}^{(1)}+\eta_A+lz)}
    {\prod_{l\leq 0}(-x_{I}^{(1)}+\eta_A+lz)}\,.
\end{equation}
\end{lem}

One can check that $\mc Z^1$ and $\mc X^2$ are Fano of index at least 2 when 
\begin{equation}\label{eqn:flagdualfano}
    1\leq N_1<N_2<\ldots <N_D\,,
\end{equation}
 and $\mc Z^3$ and $\mc Z^4$ are Fano of index at least $2$ when
\begin{equation}\label{eqn:dualtautfano}
    1\leq N_1<N_2<\ldots<N_D,\, N_0+2\leq N_2\,.
\end{equation}
Therefore, in this situation, the Theorem \ref{thm:wallcrossing} holds for all our examples with the trivial mirror map.

\section{Main theorems and proofs}\label{sec:proof}
We are ready to state our main theorems which induces the genus 0 Seiberg duality conjecture. In this section, we will always make $z=1$ in the expression of $I$-functions. Denote $I^{\mc X, S}(\vec q):=I^{\mc X, S}(\vec q, 1)$, and $I_\beta^{\mc X, S}:=I_\beta^{\mc X, S}(1)$.
\subsection{Statement of main theorem}
When a quiver mutation is performed to the quiver diagram of 
$F(N_1,\ldots, N_D)$ at a node $k$ in Example \ref{ex:dualofflag}, we have constructed $\mc Z^1$ if $k\neq 1$ and $\mc X^2$ if $k=1$.

\begin{thm}\label{thm:main} 
 The $I$-functions of pairs of varieties $F(N_1,\ldots, N_D)$ and $\mc Z^1$, $F(N_1,\ldots, N_D)$ and $\mc X^2$ satisfy the following relations.
\begin{enumerate}
    \item If $N_{k+1}\geq N_{k-1}+2$,
    \begin{equation}
    I^{F, S}(\vec q)=I^{\mc Z^1,S}(\vec q')\,,\, I^{F,S}(\vec q)=I^{\mc X^2,S}(\vec q')\,,
    \end{equation}
via the cluster transformations on k\"ahler coordinates
\begin{equation}\label{eqn:variablechangek}
    q_{k}'=q_k^{-1},\, q_{k+1}'=q_{k+1}q_k,\, q_i'=q_i\, \text{ for } i\neq k,\,k+1\,.
\end{equation}
\item If $N_{k+1}=N_{k-1}+1$, then we have $N_{k-1}=0$, $N_1=1,\, N_2=1$ in the case $k=1$. The $I$-functions satisfy 
\begin{equation}
    I^{F,S}(\vec q)={e}^{(-1)^{N_{k}'}q_k}I^{\mc Z^1,S}(\vec q')\,,\, I^{F, S}(\vec q)={e}^{q_1}I^{\mc X^2,S}(\vec q')\,,
\end{equation}
via the cluster transformations on k\"ahler coordinates
\begin{equation}
    q_{k}'=q_k^{-1},\, q_{k+1}'=q_{k+1}q_k,\, q_i'=q_i\, \text{ for } i\neq k,\,k+1\,;
\end{equation}
\item If $N_{k+1}=N_{k-1}$, the $k=1$ case is trivial. In the $k\neq 1$ case, we have 
 \begin{equation}
    I^{F, S}(\vec q)=I^{\mc Z^1,S}(\vec q')\,,
    \end{equation}
  via
    \begin{equation}
        q_{k+1}'=\frac{q_kq_{k+1}}{1+q_k},\, q_{k-1}'={q_{k-1}}({1+q_k}),\,q_i'=q_i\, \text{  for } i\neq k,\,k+1\,.
    \end{equation}

\end{enumerate}

\end{thm}

\begin{thm}\label{thm:main2}
\begin{enumerate}
\item 
When a quiver mutation is performed to the quiver diagram of 
$S_1^{\oplus N_0}\rightarrow F(N_1,\ldots, N_D)$ at a gauge node $k\neq 1$, we have constructed $\mc Z^3$ in Example \ref{ex:dualoftautflag}. Then
$I^{tF,S^2}(q)$ and $I^{\mc Z^3,S^2}(\vec q')$ satisfy exactly the same relations as Theorem \ref{thm:main}. 
\item  When a quiver mutation is applied to the quiver diagram of $S_1^{\oplus N_0}\rightarrow F(N_1,\ldots, N_D)$ at the gauge node $k=1$, we have obtained the variety $\mc Z^4$ in the dual side.  The equivariant quasimap small $I$-functions $I^{tF,S^2}(\vec q)$ and $I^{\mc Z^4,S^2}(\vec q')$ satisfy the following relations:
\begin{enumerate}
    \item when $N_{2}\geq N_0+2$ and $N_2\geq N_0+1$, $I^{tF,S^2}(q)$ and $I^{\mc Z^4,S^2}(q')$ satisfy the same relations with item 1 and item $2$ in Theorem \ref{thm:main};
    \item when $N_0=N_2$, 
    \begin{equation}
    I^{tF,S^2}(\vec q)=(1+(-1)^{N_1'}q_1)^{\sum_{A=1}^{N_0}\eta_A-\sum_{F=1}^{N_2}x_{F}^{(2)}+N_1'}I^{\mc Z^4,S^2}(\vec q')\,,
    \end{equation}
    via the cluster transformations on k\"ahler coordinates,
         \begin{equation}
         q_1'=q_1^{-1}, q_{2}'=\frac{q_1q_{2}}{1+(-1)^{N_1'}q_1}\,.
    \end{equation}
In the above formula, $(1+(-1)^{N_1'}q_1)^{\sum_{A=1}^{N_0}\eta_A-\sum_{F=1}^{N_2}x_{F}^{(2)}+N_1'}$ represents a formal power series
  \begin{equation}
\sum_{m\geq 0}\frac{\prod_{l=0}^{m-1}(\sum_{A=1}^{N_0}\eta_A-\sum_{F=1}^{N_2}x_{F}^{(2)}+N_1'-l)}{m!} ((-1)^{N_1'}q_1)^m\,.
  \end{equation}
\end{enumerate}
\end{enumerate}
\end{thm}
Since all of our varieties $F(N_1,\ldots,N_D)$, $S_1^{\oplus N_0}\rightarrow F(N_1,\ldots, N_D)$, $\mc Z^1$, $\mc Z^2$, $\mc Z^3$, and $\mc Z^4$ are Fano varieties of index at least 2 
 under conditions
\begin{equation}
    1\leq N_1<N_2<\ldots<N_D,\,\, N_0+2\leq N_2\,,
\end{equation}
then we are led to the genus-zero Seiberg duality conjecture.
\begin{cor}\label{cor:seiberg}
When $1\leq N_1<\ldots<N_D$ and $N_0+2\leq N_2$, the equivariant small $\mc J$-functions  of $A_n$-type quivers in Example \ref{ex:flag} and Example \ref{ex:tautflag} and their dualities satisfy
the genus 0 Seiberg duality Conjecture \ref{conj}. 
\end{cor}

\subsection{Seiberg duality conjecture in the level of equivariant cohomology groups}
In this subsection, we will prove that the equivariant cohomology groups of two varieties before and after a quiver mutation are isomorphic. 
\begin{lem}\label{lem:fixedpoints1to1}
    There exists a bijection map
    \begin{equation}
        \iota: \mathfrak{F}^{bm}\rightarrow \mathfrak F^{am}\,.
    \end{equation}
\end{lem}
\begin{proof}
For any torus fixed point $Q=(\vec C_{[N_1]}\subset \vec C_{[N_2]}\subseteq\cdots \subseteq [N_D] )\in \mathfrak{F}^{bm}$, $\iota$ keeps each integer subset $\vec C_{N_{i}}$ for $i\neq k$ unchanged, and sends $\vec C_{N_k}$ to $\vec C_{N_k'}=\vec C_{N_{k+1}}\backslash \vec C_{N_k}$. One can check that this map $\iota$ is bijective. 
\end{proof}
\begin{thm}\label{thm:seibergforcohomology}
We have the following isomorphism among equivariant cohomology groups,
\begin{equation}
    H^*_S(F(N_1,\ldots, N_D))\iso H^*_S(\mc Z^1) \iso H^*_S(\mc Z^2)\,,
\end{equation}
and 
\begin{equation}
    H^*_{S^2}(S^{\oplus N_0}\rightarrow F(N_1,\ldots, N_D))\iso H^*_{S^2}(\mc Z^3) \iso H^*_{S^2}(\mc Z^4)\,.
\end{equation}
\end{thm}
\begin{proof}
Both $F(N_1,\ldots,N_D)$ and $\mc Z^1$ admit a good torus action $S=(\C^*)^{N_D}$, and the torus fixed loci for them are $\mathfrak{F}^{bm}$ and $\mathfrak{F}^{am}$. The localization theorem \cite{ATIYAH19841} states that
\begin{equation}
    {H}_{S}^*(F(N_1,\ldots, N_D))\iso \bigoplus_{Q\in \mathfrak F^{bm}}\op{H}^*_S(Q)\,,
\end{equation}
and 
\begin{equation}
    {H}_{S}^*(\mc Z^1)\iso \bigoplus_{P\in \mathfrak F^{am}}\op{H}^*_S(P)\,.
\end{equation}
Since $\mathfrak{F}^{bm}$ and $\mathfrak{F}^{am}$ are in one to one correspondence by Lemma \ref{lem:fixedpoints1to1}, 
we obtain that $H^*_S(F(N_1,\ldots, N_D))\iso H^*_S(\mc Z^1)$. Similar reasoning, we can get the isomorphism between  $H^*_S(F(N_1,\ldots, N_D))$ and $H^*_S(\mc X^2)$, and isomorphism among $H^*_{S^2}(S^{\oplus N_0}\rightarrow F(N_1,\ldots, N_D))$, $H^*_{S^2}(\mc Z^3)$ and $H^*_{S^2}(\mc Z^4)$.
\end{proof}
At the end of this subsection, we would like to outline our strategies to prove the main theorems. 
Let $(V,T\subset G,\theta)$ denote the input data of a quiver variety before a quiver mutation which represents $F(N_1, \ldots, N_D)$ and $S_1^{\oplus N_0}\rightarrow F(N_1,\ldots, N_D)$ in our examples, 
and let $(\tilde Z, \tilde T\subseteq \tilde G, \tilde \theta)$ denote the input data of a GIT quotient after a quiver mutation which represents $\mc Z^1,\,\mc Z^2,\, \mc Z^3,\, \mc Z^4$.  Assume that $V$ and $\tilde Z$ admit a common good torus action $S$. 

Our goal is to to prove that $I^{V\sslash G,S}(\vec q)=I^{\tilde Z\sslash {\tilde G},S}(\vec q')$ under a suitable variable change. 
Since we have proved that equivariant cohomology groups of $V\sslash G$ and $\tilde Z\sslash {\tilde G}$ are isomorphic, we only have to prove that for each $Q\in \mathfrak F^{bm}$ and $P:=\iota(Q)\in \mathfrak F^{am}$
\begin{equation}\label{eqn:Ifunctionlocalization}
    i_Q^*I^{V\sslash G,S}(\vec q)=i_P^*I^{\tilde Z\sslash {\tilde G},S}(\vec q')\,,
\end{equation}
where $i_Q:Q\rightarrow V\sslash G$ and $i_P: P\rightarrow \tilde Z\sslash{\tilde G}$ are inclusions.

 Consider the following commutative diagram,
\begin{equation}
    \begin{tikzcd}
 &H^*_S(V^s(G)/T)^{W_T} \arrow[]{r}{} & H^*_{S}(\tilde Z^s(
\tilde G)/{\tilde T})^{W_{\tilde T}}  &  \\
\oplus_{Q\in \mathfrak F^{bm}} H^*_S(Q)\arrow[]{r}{i_*}&H^*_S(V^s(G)/G) \arrow[]{r}{}\arrow[]{u}{g^*}  & H^*_{S}(\tilde Z\sslash \tilde G)\arrow[]{u}{g^*} &\oplus_{P\in \mathfrak F^{am}}H^*_S(P)\arrow[]{l}[swap]{j_*}
\end{tikzcd}
\end{equation}
Each arrow in the above diagram is an isomorphism.

Let $Q'\in (V^s(G)\slash T)^S$ be any lifting of $Q$ via the map $g$ in diagram \eqref{diag}. This lifting is not unique, and one can show that for all of our examples, any two distinct liftings are connected by an element $w\in W_T$.

Let $g^*(I^{V\sslash G,S})$ be the lifting of $I^{V\sslash G,S}$. By the isomorphism \eqref{eqn:cohabelianization}, $g^*(I^{V\sslash G,S}) \in H^*_S(V^s(G)\slash T)^{W_T}$ is $W_T$ invariant, so $i_{Q'}^*(g^*I^{V\sslash G,S})$ is independent of the choice of $Q'$. 

Let $Q' \in (V^s(G)\slash T)^S$ and $P'\in (\tilde Z^s(\tilde G)\slash \tilde T)^S$ be arbitrary liftings of $Q$ and $P$. 
Then the Equation \eqref{eqn:Ifunctionlocalization} is equivalent to
\begin{equation}
     i_{Q'}^*(g^*I^{V\sslash G,S})=i_{P'}^*(g^*I^{\tilde Z\sslash {\tilde G},S})\,.
\end{equation}
We will prove this equation for all pairs of our varieties before and after a quiver mutation. 
In the following subsections, when we talk about the restriction of a small $I$-function $I^{V\sslash G, S}$ to some torus fixed point $Q$, we mean $i_{Q'}^*(g^*I^{V\sslash G,S})$. 

\subsection{Proof for a fundamental building block}\label{sec:taugr}

In this section, we consider the fundamental building block. Consider the quiver diagram below with $N_1\leq N_2,\, N_0\leq N_2$,  which is a special case of Example \ref{ex:tautflag},
\begin{figure}[H]
    \centering
     \includegraphics[width=2in]{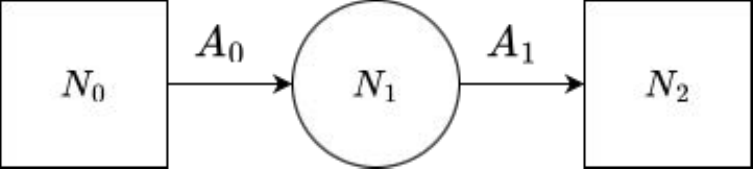}
\end{figure}
\noindent 
The quiver variety is the total space of $N_0$ copies of the tautological bundle  over a Grassmannian:  $S_1^{ N_0}\rightarrow Gr(N_1, N_2)$. Apply a quiver mutation, 
and we get a quiver diagram below with a potential $W=tr(BA_1A_0)$, where $N_1'=N_2-N_1$.
\begin{figure}[H]
    \centering
    \includegraphics[width=2in]{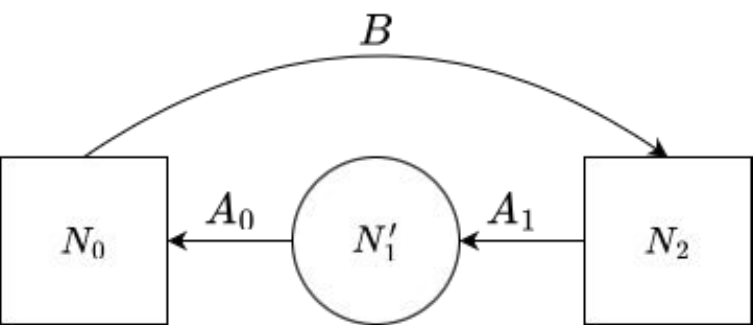}
\end{figure}
\noindent 
By the argument in Example \ref{ex:dualoftautflag}, we need to consider the negative phase of the gauge group: $\theta(g)=\det(g)^{\sigma}$, $\sigma<0$, for $g\in GL(N_1')$. The corresponding variety is the total space of $N_0$ copies of the dual tautological bundle over the dual Grassmannian $ (S_1^{\vee})^{\oplus N_0}\rightarrow Gr(N_1',N_2)$.

The genus-zero Seiberg duality conjecture holds for the fundamental building block if we can prove that 
 the equivariant quasimap small $I$-functions of $S_1^{\oplus N_0}\rightarrow Gr(N_1, N_2)$ and $(S_1^\vee)^{\oplus N_0}\rightarrow Gr(N_1', N_2)$ are equal.

Both $S_1^{\oplus N_0}\rightarrow Gr(N_1, N_2)$ and $(S_1^\vee)^{\oplus N_0}\rightarrow Gr(N_1', N_2)$ admit a good torus action $S^2=(\C^*)^{N_0}\times(\C^{*})^{N_2}$. 
The torus fixed locus in $S_1^{\oplus N_0}\rightarrow Gr(N_1, N_2)$ is $\mathfrak{F}^{bm}:=\{\vec C_{[N_1]}\subseteq [N_2]\}$ and that in $(S_1^\vee)^{\oplus N_0}\rightarrow Gr(N_1', N_2)$ is $\mathfrak{F}^{am}=\{\vec C_{[N_1']}\subseteq [N_{2}]\}$. 

The equivariant quasimap small ${I}$-function of $S_1^{\oplus N_0}\rightarrow Gr(N_1,N_2)$ denoted by ${I}^{Gr,S^2}(q)$ can be written as follows
\begin{equation}
   {I}^{Gr,S^2}(q)=\sum_{\vec d\in \Z^N_{\geq 0}} 
   \prod_{I\neq J}^{N_1}\frac{\prod_{l\leq d_I-d_J}(x_I-x_J+l)}{\prod_{l\leq 0}(x_I-x_J+l)}
    \prod_{I=1}^{N_1}\frac{\prod_{A=1}^{N_0}\prod_{l=0}^{d_I-1}(-x_I+\eta_A-l)}{ \prod_{F=1}^{N_2}\prod_{l=1}^{d_I}(x_I-\lambda_F+l)}q^{\abs{\vec d}}\,.
\end{equation}
The equivariant quasimap small $I$-function of $(S^\vee)^{\oplus N_0}\rightarrow Gr(N_1',N_2)$ denoted by $I^{Gr^\vee, S^2}(q')$ is
\begin{equation}
    {I}^{Gr^\vee,S^2}(q')=\sum_{\vec d\in \Z^{N_1'}_{\leq 0}}
    \prod_{I\neq J}^{N_1'}\frac{\prod_{l\leq d_I-d_J}(x_I-x_J+l)}{\prod_{l\leq 0}(x_I-x_J+l)}
    \prod_{I=1}^{N_1'}\frac{\prod_{A=1}^{N_0}\prod_{l=1}^{-d_I}(-x_I+\eta_A+l)}{ \prod_{F=1}^{N_2}\prod_{l=1}^{-d_I}(-x_I+\lambda_F+l)}(q')^{\abs{\vec d}}\,.
\end{equation}
For an arbitrary point $Q=(\{f_1<\ldots<f_{N_1}\}\subseteq [N_2])  \in \mathfrak F^{bm}$, denote the image $\iota(Q)$ by $P=(\{f_1'<\ldots<f_{N_1'}'\}\subseteq [N_2])\in \mathfrak F^{am}$. The restriction of $I^{Gr,S^2}$ to $Q$ is 
\begin{equation}\label{eqn:Ifungrbmres}
   i_{Q}^*I^{Gr,S^2}(q)=\sum_{\vec d\in \Z^N_{\geq 0}}
    \prod_{I\neq J}^{N_1}\frac{\prod_{l\leq d_I-d_J}(\lambda_{f_I}-\lambda_{f_J}+l)}{\prod_{l\leq 0}(\lambda_{f_I}-\lambda_{f_J}+l)}
    \prod_{I=1}^{N_1}\frac{\prod_{A=1}^{N_0}\prod_{l=0}^{d_I-1}(-\lambda_{f_I}+\eta_A-l)}
    {\prod_{F=1}^{N_2}\prod_{l=1}^{d_I}(\lambda_{f_I}-\lambda_F+l)}q^{\abs{\vec d}}\,.
\end{equation}
The restriction of $I^{Gr^\vee,S^2}$ to $P$ is  
\begin{equation}\label{eqn:Ifungramres}
   i^*_{P}I^{Gr^\vee,S^2}(q')=\sum_{\vec d\in \Z^{N_1'}_{\leq  0}}
    \prod_{I\neq J}^{N_1'}\frac{\prod_{l\leq d_I-d_J}(\lambda_{f_I'}-\lambda_{f_J'}+l)}{\prod_{l\leq 0}(\lambda_{f_I'}-\lambda_{f_J'}+l)}
    \prod_{I=1}^{N_1'}\frac{\prod_{A=1}^{N_0}\prod_{l=1}^{-d_I}(-\lambda_{f_I'}+\eta_A+l)}{ \prod_{F=1}^{N_2}\prod_{l=1}^{-d_I}(-\lambda_{f_I'}+\lambda_F+l)}{q'}^{\abs{\vec d}}\,.
\end{equation}

\begin{thm}[\cite{donghai,benini2015cluster}]\label{thm:Haidong}
\begin{enumerate} 
    \item When $N_2\geq N_0+2$, $i_{Q}^*{I}^{Gr,S^2}(q)=i^*_{P}{I}^{Gr^\vee,S^2}(q^{-1})$.
    \item When $N_2=N_0+1$, $i_{Q}^*{I}^{Gr,S^2}(q)={e}^{(-)^{N_1'}q}\, i^*_{P}{I}^{Gr^\vee,S^2}(q^{-1})$.
    \item When $N_2=N_0$, 
    $i_{Q}^*{I}^{Gr,S^2}(q)=(1+(-1)^{N_1'}q)^{\sum_{A=1}^{N_0}\eta_A-\sum_{F=1}^{N_2}\lambda_{F}+N_1'}i^*_{P}{I}^{Gr^\vee,S^2}(q^{-1})$.
\end{enumerate}
\end{thm}
\begin{proof}
The situations $N_2\geq N_0+2$ and $N_2=N_0+1$ have been proved by Hai Dong in his thesis \cite{donghai} which has no online resources, and the proof of the case $N_2=N_0$ can be found in \cite[appendix]{benini2015cluster}. We include the same proofs in appendix for readers' convenience. One can also find proofs in physical papers \cite[appendix]{MR3296161}\cite{gomis2016m2}.
\end{proof}
The above theorem implies that $I^{Gr, S^2}(q)$ and $I^{Gr^\vee, S^2}(q')$ satisfy the relations in Theorem \ref{thm:main2} item 2. 
When $N_2\geq N_0+2$, the wall-crossing Theorem \ref{thm:wallcrossing} holds with the trivial mirror map. 
Therefore we have proved the genus-zero Seiberg duality conjecture for the fundamental building block: $\mc J^{Gr,S^2}(q)=\mc J^{Gr^\vee, S^2}(q^{-1})$ when $N_2\geq N_0+2$.

\subsection{Proofs for Theorem \ref{thm:main}}
We will follow the outline in the above subsection to prove this Theorem, and to prove the Equation \eqref{eqn:Ifunctionlocalization} for $I^{F, S}(\vec q)$ and $I^{\mc Z^1, S}(\vec q')$.
Without loss of generality, we consider the special pair of torus fixed points 
\begin{equation}
    Q_0=\{[N_1]\subseteq [N_2]\subseteq\ldots\subseteq [N_D] \}\,,
\end{equation}
 and $P_0=\iota(Q_0)$.
We firstly restrict $I^{F, S}(\vec q)$ in Lemma \ref{lem:Ifuncofflagbm} to the torus fixed point $Q_0$
Then $x_{I}^{(i)}\rvert_{Q_0}=\lambda_{I}$ for each $i$ and
 \begin{equation}
 i_{Q_0}^*I^{F,S}(\vec q)=\sum_{(\vec m^{(i)})\in \op{Eff}^T_{bm}}
 \prod_{i=1}^{D-1}\prod_{I\neq J}^{N_i}
 \frac{\prod_{l\leq m_I^{(i)}-m_J^{(i)}}(\lambda_I-\lambda_{J}+l)}
 {\prod_{l\leq 0}(\lambda_I-\lambda_J+l)}\prod_{I=1}^{N_i}\prod_{J=1}^{N_{i+1}}
 \frac{\prod_{l\leq 0}(\lambda_I-\lambda_J+l)}
 {\prod_{l\leq m_I^{(i)}-m_J^{(i+1)}}(\lambda_I-\lambda_J+l)}q_i^{\abs{\vec m^{(i)}}}\,.
 \end{equation}
 For each fixed degree $(\vec m^{(1)},\ldots, \vec m^{D-1})\in \op{Eff}_{bm}^{T}$, in the spirit of the Figure \ref{diag:ideaofproof} in the introduction, we split $i_{Q_0}^*I^{F,S}_{(\vec m^{(i)})}$ into two parts: 
 \begin{align}
 i_{Q_0}^*I^{F,S}_{(\vec m^{(i)})}=R^{(k)}I^{(k)}\,,
 \end{align}
 where 
 \begin{align}\label{eqn:Inodek}
    I^{(k)}=&
 \prod_{I\neq J}^{N_k}
 \frac{\prod_{l\leq m_I^{(k)}-m_J^{(k)}}(\lambda_I-\lambda_{J}+l)}
 {\prod_{l\leq 0}(\lambda_I-\lambda_J+l)}
 \prod_{I=1}^{N_k}\prod_{F=1}^{N_{k+1}}
 \frac{\prod_{l\leq 0}(\lambda_I-\lambda_F+l)}
 {\prod_{l\leq m_I^{(k)}-m_F^{(k+1)}}(\lambda_I-\lambda_F+l)}\nonumber\\
 & \prod_{A=1}^{N_{k-1}}\prod_{J=1}^{N_{k}}
 \frac{\prod_{l\leq 0}(\lambda_A-\lambda_J+l)}
 {\prod_{l\leq m_A^{(k-1)}-m_J^{(k)}}(\lambda_A-\lambda_J+l)}\,,
 \end{align}
 and $R^{{(k)}}$ is the remaining part.
In the expression of $ I^{(k)}$,
$m^{(k)}_{I}\geq m^{(k+1)}_{I}$ for $I=1,\ldots, N_k$, otherwise $I^{(k)}$  would vanish. 
Let $n^{(k)}_I=m^{(k)}_{I}- m^{(k+1)}_{I}$. 
Making a substitution $m^{(k)}_I=n^{(k)}_I+m^{(k+1)}_I$, and doing some combinatorics, 
we can rewrite $I^{(k)}$ as follows,

\begin{subequations}
\begin{align}
     I^{(k)}=&
 \prod_{I\neq J}^{N_k}
 \frac{\prod_{l\leq n^{(k)}_I-n^{(k)}_J}(\lambda_I-\lambda_{J}+m^{(k+1)}_I-m^{(k+1)}_J+l)}
 {\prod_{l\leq 0}(\lambda_I-\lambda_J+m^{(k+1)}_I-m^{(k+1)}_J+l)}\label{eqn:flagI1}\\
 &\frac{\prod_{A=1}^{N_{k-1}}\prod_{I=1}^{N_k} \prod_{l=0}^{n_I^{(k)}-1}(\lambda_A-\lambda_I+m_A^{(k-1)}-m_I^{(k+1)}-l) }
 {\prod_{I=1}^{N_k}\prod_{F=1}^{N_{k+1}}\prod_{l=1}^{n_I^{(k)}}(\lambda_I-\lambda_F+m^{(k+1)}_{I}-m_F^{(k+1)}+l)}\label{eqn:flagI2}\\
 &\prod_{I=1}^{N_k}\prod_{F\in [N_{k+1}]\backslash [N_k]}
 \frac{\prod_{l\leq 0}(\lambda_I-\lambda_F+l)}{\prod_{l\leq m_I^{(k+1)}-m_F^{(k+1)}}(\lambda_I-\lambda_F+l)}
 \prod_{A=1}^{N_{k-1}}\prod_{I=1}^{N_k}
 \frac{\prod_{l\leq 0}(\lambda_A-\lambda_I+l)}{\prod_{l\leq m_A^{(k-1)}-m^{(k+1)}_I}(\lambda_A-\lambda_I+l)}\,.\label{eqn:flagI3}
\end{align}
\end{subequations}
Observe the above formula and one can find sub-equations \eqref{eqn:flagI1} and \eqref{eqn:flagI2} together can be viewed as the degree $\vec n^{(k)}$ term of the equivariant quasimap small $I$-function of $S_1^{\oplus N_{k-1}}\rightarrow Gr(N_k,N_{k+1})$ in \eqref{eqn:Ifungrbmres}, pulled back to the $S^2$-fixed point $Q=\{1,\ldots, N_k\}$, 
if one pretends $\lambda_{F}+m^{(k+1)}_F$ are equivariant parameters of the torus $(\C^*)^{N_{k+1}}$ and $\lambda_A+m^{(k-1)}_A$ are equivariant parameters of the torus $(\C^*)^{N_{k-1}}$. 

We do the similar combinatorics to the $I$-function of the variety $\mc Z^1$ in Lemma \ref{lem:IfunctionofZ1}. 
After being restricted to the torus fixed point $P_0=\iota(Q_0)$, $x_I^{(k)}\rvert_{P_0}=\lambda_{N_k+I}$, $x_J^{(i)}\rvert_{P_0}=\lambda_J$ for $i\neq k$, and  
\begin{align}
        i_{P_0}^*I^{\mc Z^1}(\vec q')=&\sum_{(\vec m^{(i)})\in \op{Eff}^T_{am}}\prod_{\substack{i=1,\\i\neq k-1,k}}^{D-1}\left(\prod_{\substack{I,J=1\\I\neq J}}^{N_i}
    \frac{\prod_{l\leq m_{I}^{(i)}-m_{J}^{(i)}}(\lambda_{I}-\lambda_{J}+l)}
    {\prod_{l\leq 0}(\lambda_{I}-\lambda_{J}+l)}
    \prod_{I=1}^{N_i}\prod_{J=1}^{N_{i+1}}
    \frac{\prod_{l\leq 0}(\lambda_{I}-\lambda_{J}+l)}
    {\prod_{l\leq m_{I}^{(i)}-m_{J}^{(i+1)}}(\lambda_{I}-\lambda_{J}+l)}\right)\nonumber\\
    &
    \prod_{\substack{I,J=1\\I\neq J}}^{N_{k-1}}
    \frac{\prod_{l\leq m_{I}^{(k-1)}-m_{J}^{(k-1)}}(\lambda_{I}-\lambda_{J}+l)}
    {\prod_{l\leq 0}(\lambda_{I}-\lambda_{J}+l)}
    \prod_{A=1}^{N_{k-1}}\prod_{F=1}^{N_{k+1}}
    \frac{\prod_{l\leq 0}(\lambda_A-\lambda_{F}+l)}
    {\prod_{l\leq m_{A}^{(k-1)}-m_{F}^{(k+1)}}(\lambda_A-\lambda_F+l)}\nonumber\\
    &
    \prod_{\substack{I,J=1\\I\neq J}}^{N_k'}
    \frac{\prod_{l\leq m_I^{(k)}-m_J^{(k)}}(\lambda_{N_K+I}^{(k)}-\lambda_{N_K+J}^{(k)}+l)}{\prod_{l\leq 0}(\lambda_{N_K+I}^{(k)}-\lambda_{N_K+J}^{(k)}+l)}
    \prod_{F=1}^{N_{k+1}}\prod_{I=1}^{N_{k}'}
    \frac{\prod_{l\leq 0}(\lambda_F-\lambda_{N_k+I}+l)}
    {\prod_{l\leq m_{F}^{(k+1)}-m_{I}^{(k)}}(\lambda_F-\lambda_{N_k+I}+l)}\nonumber\\
    &
    \prod_{I=1}^{N_k'}\prod_{A=1}^{N_{k-1}}
    \frac{\prod_{l\leq -m_{I}^{(k)}+m_{A}^{(k-1)}}(-\lambda_{N_k+I}+\lambda_A+l)}
    {\prod_{l\leq 0}(-\lambda_{N_k+I}+\lambda_A+l)}\prod_{i=1}^{D-1}(q_{i}')^{\abs{\vec m^{(i)}}}\,.
\end{align}
For each fixed degree $(\vec m^{(1)},\ldots,\vec m^{(D-1)})\in \op{Eff}_{am}^T$, 
we split $i_{P_0}^*I^{\mc Z^1}_{(\vec m^{(i)})}$ into 
\begin{equation}
    i_{P_0}^*I^{\mc Z^1}_{(\vec m^{(i)})}=R_{am}^{(k)}I^{(k)}_{am}\,,
\end{equation}
where 

\begin{align}\label{eqn:Inodekdual}
    I^{(k)}_{am}=&
    \prod_{A=1}^{N_{k-1}}\prod_{F=1}^{N_{k+1}}
    \frac{\prod_{l\leq 0}(\lambda_A-\lambda_{F}+l)}
    {\prod_{l\leq m_{A}^{(k-1)}-m_{F}^{(k+1)}}(\lambda_A-\lambda_F+l)}
        \prod_{I=1}^{N_k'}\prod_{A=1}^{N_{k-1}}
    \frac{\prod_{l\leq -m_{I}^{(k)}+m_{A}^{(k-1)}}(-\lambda_{N_k+I}+\lambda_A+l)}
    {\prod_{l\leq 0}(-\lambda_{N_k+I}+\lambda_A+l)}\nonumber\\
    &
    \prod_{\substack{I,J=1\\I\neq J}}^{N_k'}
    \frac{\prod_{l\leq m_I^{(k)}-m_J^{(k)}}(\lambda_{N_K+I}^{(k)}-\lambda_{N_K+J}^{(k)}+l)}{\prod_{l\leq 0}(\lambda_{N_K+I}^{(k)}-\lambda_{N_K+J}^{(k)}+l)}
    \prod_{F=1}^{N_{k+1}}\prod_{I=1}^{N_{k}'}
    \frac{\prod_{l\leq 0}(\lambda_F-\lambda_{N_k+I}+l)}
    {\prod_{l\leq m_{F}^{(k+1)}-m_{I}^{(k)}}(\lambda_F-\lambda_{N_k+I}+l)}\,.
\end{align}
Denote $n_I^{(k)}:=m_{N_k+I}^{(k+1)}-m_I^{(k)}$. 
We are aware that $n_I^{(k)}\geq 0$, otherwise $ I^{(k)}_{am}$ would vanish. 
By making substitution $m_{I}^{(k)}=m_{N_k+I}^{(k+1)}-n_I^{(k)}$, we can transform $ I^{(k)}_{am}$ to
\begin{subequations}
\begin{align}
     I^{(k)}_{am}=&
     \prod_{\substack{I,J=1\\I\neq J}}^{N_k'}
    \frac{\prod_{l\leq n_J^{(k)}-n_I^{(k)}}(\lambda_{N_K+I}^{(k)}-\lambda_{N_K+J}^{(k)}+m^{(k+1)}_{N_k+I}-m^{(k+1)}_{N_k+J}+l)}{\prod_{l\leq 0}(\lambda_{N_K+I}^{(k)}-\lambda_{N_K+J}^{(k)}+m^{(k+1)}_{N_k+I}-m^{(k+1)}_{N_k+J}+l)}\label{eqn:flagdual1}\\
    &
    \prod_{I=1}^{N_k'}\frac{\prod_{A=1}^{N_{k-1}} \prod_{l=1}^{n^{(k)}_I}(-\lambda_{N_k+I}-m_{N_k+I}^{(k+1)}+\lambda_A+m_{A}^{(k-1)}+l) }
    {\prod_{F=1}^{N_{k+1}}\prod_{l=1}^{n_I^{(k)}}(\lambda_F-\lambda_{N_k+I}+m_{F}^{(k+1)}-m_{N_k+I}^{(k+1)}+l)}\label{eqn:flagdual2}\\
    &
    \prod_{I=1}^{N_k}\prod_{F\in [N_{k+1}]\backslash [N_k]}
 \frac{\prod_{l\leq 0}(\lambda_I-\lambda_F+l)}{\prod_{l\leq m_I^{(k+1)}-m_F^{(k+1)}}(\lambda_I-\lambda_F+l)}
 \prod_{A=1}^{N_{k-1}}\prod_{I=1}^{N_k}
 \frac{\prod_{l\leq 0}(\lambda_A-\lambda_I+l)}{\prod_{l\leq m_A^{(k-1)}-m^{(k+1)}_I}(\lambda_A-\lambda_I+l)}\,.\label{eqn:flagdual3}
\end{align}
\end{subequations}
Sub-equations \eqref{eqn:flagdual1} and \eqref{eqn:flagdual2} together can be viewed as the degree $\vec n^{(k)}$ term of the $I$-function of $(S_1^{\vee})^{\oplus N_{k-1}}\rightarrow Gr(N_k',N_{k+1})$ in \eqref{eqn:Ifungramres} being restricted to the $S^2$-fixed point $(\{N_k+1<\ldots< N_{k+1}\}\subseteq [N_{k+1}])$, 
if we pretend that $\lambda_F+m^{(k+1)}_{F}$ are the equivariant parameters of the torus $(\C^*)^{N_{k+1}}$, 
and $\lambda_A+m^{(k-1)}_A$ are the equivariant parameters of the torus $(\C^*)^{N_{k-1}}$. 

Compare the two formulae $I^{(k)}$ and $I^{(k)}_{am}$ for fixed integer vectors $\vec m^{(i)}$, $i\neq k$. 
Denote 
\begin{align}
    I^{(k)}(\vec q):= \sum_{\substack {\vec m^{(k)}:\\ m^{(k)}_I\geq m^{(k+1)}_I}}I^{(k)}q_{k}^{\abs{\vec m^{(k)}}}q_{k+1}^{\abs{\vec m^{(k+1)}}}q_{k-1}^{\abs{\vec m^{(k-1)}}}\,,
\end{align}
and denote 
\begin{equation}
    I^{(k)}_{am}(\vec q'):=\sum_{\substack{\vec {m}^{(k)}:\\m_{N_k+I}^{(k+1)}\geq m_I^{(k)} }}I^{(k)}_{am}(q_{k}')^{\abs{\vec { m}^{(k)}}}(q'_{k+1})^{\abs{\vec m^{(k+1)}}}(q'_{k-1})^{\abs{\vec m^{(k-1)}}}\,.
\end{equation} 
Notice that we have involved variables $q_k,\,q_{k-1},\,q_{k+1}$, and $q_k',\,q_{k-1}',\,q_{k+1}'$  in the summations in order to investigate the change of variables under the cluster transformation.
\begin{lem}\label{lem:mainlemmaflag}
For fixed integer vectors $\vec m^{(i)}$, $i\neq k$, $I^{(k)}(\vec q)$ and   $I^{(k)}_{am}(\vec q')$
 satisfy the following relations.
\begin{itemize} 
    \item When $N_{k+1}\geq N_{k-1}+2$, 
    \begin{equation}\label{eqn:lemflag1}
    I^{(k)}(\vec q)= I^{(k)}_{am}(\vec q')\,,
\end{equation}
via the variable change,
\begin{equation}
    q_{k}'=q_k^{-1}, q_{k+1}'=q_{k+1}q_k\,.
\end{equation}
\item When $N_{k+1}=N_{k-1}+1$,
\begin{equation}\label{eqn:lemflag2}
   I^{(k)}(\vec q) = 
    e^{(-1)^{N_k'} q_k} I^{(k)}_{am}(\vec q')\,,
\end{equation}
via the variable change,
\begin{equation}
    q_{k}'=q_k^{-1}, q_{k+1}'=q_{k+1}q_k\,.
\end{equation}
\item When $N_{k+1}=N_{k-1}$, then $N_k=N_{k+1}=N_{k-1}$ and $N_k'=0$.
 \begin{equation}\label{eqn:N2=N04}
 I^{(k)}(\vec q)= 
I^{(k)}_{am}(\vec q')\,,
    \end{equation}
    via the variable change, 
    \begin{equation}\label{eqn:variablechangeN2=N0k}
        q_{k+1}'=\frac{q_kq_{k+1}}{1+q_k},\, q_{k-1}'={q_{k-1}}(1+q_k)\,.
    \end{equation}
\end{itemize}
\end{lem}
\begin{proof}
This lemma is a straightforward application of Theorem \ref{thm:Haidong}.
Since the sub-equations \eqref{eqn:flagI3} and \eqref{eqn:flagdual3} are exactly equal no matter what $\vec m^{(k)}$ is, we only have to consider the sum of subequations  \eqref{eqn:flagI1} and \eqref{eqn:flagI2} and the sum of subequations \eqref{eqn:flagdual1} and \eqref{eqn:flagdual2}. One can find that
\begin{align}
I^{(k)}(\vec q)&=
    \sum_{\vec n^{(k)}\in \Z_{\geq 0}^{N_k}}\left(i^*_{Q_0}I^{Gr, S^2}_{\vec n^{(k)}}\right)q_{k}^{\abs{\vec n^{(k)}}+\sum_{I=1}^{N_k}m^{(k+1)}_{I}}q_{k+1}^{\abs{\vec m^{(k+1)}}}q_{k-1}^{\abs{\vec m^{(k-1)}}}\,\nonumber\\
    &\times \prod_{I=1}^{N_k}\prod_{F\in [N_{k+1}]\backslash [N_k]}
 \frac{\prod_{l\leq 0}(\lambda_I-\lambda_F+l)}{\prod_{l\leq m_I^{(k+1)}-m_F^{(k+1)}}(\lambda_I-\lambda_F+l)}
 \prod_{A=1}^{N_{k-1}}\prod_{I=1}^{N_k}
 \frac{\prod_{l\leq 0}(\lambda_A-\lambda_I+l)}{\prod_{l\leq m_A^{(k-1)}-m^{(k+1)}_I}(\lambda_A-\lambda_I+l)}\,,
\end{align}
where $I^{Gr, S^2}_{\vec n^{(k)}}$ is the degree $\vec n^{(k)}$-term of the $I$-function of $S_1^{\oplus N_{k-1}}\rightarrow Gr(N_k,N_{k+1})$,
and $Q_0=([N_{k}]\subseteq [N_{k+1}])$ is the $S^2$-fixed point. Similarly, 
\begin{align}\label{eqn:N2=N03}
    I^{(k)}_{am}(\vec q')&=  \sum_{\vec n^{(k)}\in \Z_{\geq 0}^{N_k'}}\left(i^*_{P_0}I^{Gr^\vee, S^2}_{\vec n^{(k)}}\right)(q_{k}')^{-\abs{\vec n^{(k)}}+\sum_{I=1}^{N_k'}m_{N_k+I}^{(k+1)}}(q'_{k+1})^{\abs{\vec m^{(k+1)}}}(q'_{k-1})^{\abs{\vec m^{(k-1)}}}\,\nonumber\\
    &\prod_{I=1}^{N_k}\prod_{F\in [N_{k+1}]\backslash [N_k]}
    \frac{\prod_{l\leq 0}(\lambda_I-\lambda_F+l)}{\prod_{l\leq m_I^{(k+1)}-m_F^{(k+1)}}(\lambda_I-\lambda_F+l)}
    \prod_{A=1}^{N_{k-1}}\prod_{I=1}^{N_k}
    \frac{\prod_{l\leq 0}(\lambda_A-\lambda_I+l)}{\prod_{l\leq m_A^{(k-1)}-m^{(k+1)}_I}(\lambda_A-\lambda_I+l)}\,,
\end{align}
where $I^{Gr^\vee,S^2}_{\vec n^{(k)}}$ is the equivariant quasimap small $I$-function of $(S_1^{\vee})^{\oplus N_{k-1}}\rightarrow Gr(N_k',N_{k+1})$ and  $P_0=\iota(Q_0)$.
In the above two expressions, the equivariant parameters of $S^2$ are $\lambda_F+m^{(k+1)}_F$ and $\eta_A+m^{(k-1)}_A$.
Therefore we can derive the first two relations \eqref{eqn:lemflag1}\eqref{eqn:lemflag2} for $N_2\geq N_0+1$ by Theorem \ref{thm:Haidong}. 
As to the case $N_{k-1}=N_{k+1}$, by Theorem \ref{thm:Haidong}, we get,
\begin{align}\label{eqn:N2=N02}
&I^{(k)}(\vec q) =\sum_{\vec n^{(k)}\in Z_{\geq 0}^{N_1'}}I^{Gr^\vee,S^2}_{\vec n^{(k)}}q_{k}^{\abs{\vec n^{(k)}}+\sum_{I=1}^{N_k}m^{(k+1)}_{I}}(\frac{q_{k+1}}{1+q_k})^{\abs{\vec m^{(k+1)}}}((1+q_k)q_{k-1})^{\abs{\vec m^{(k-1)}}}\nonumber\\
&\times \prod_{I=1}^{N_k}\prod_{F\in [N_{k+1}]\backslash [N_k]}
    \frac{\prod_{l\leq 0}(\lambda_I-\lambda_F+l)}{\prod_{l\leq m_I^{(k+1)}-m_F^{(k+1)}}(\lambda_I-\lambda_F+l)}
    \prod_{A=1}^{N_{k-1}}\prod_{I=1}^{N_k}
    \frac{\prod_{l\leq 0}(\lambda_A-\lambda_I+l)}{\prod_{l\leq m_A^{(k-1)}-m^{(k+1)}_I}(\lambda_A-\lambda_I+l)}\,.
\end{align}
Then we can obtain the relation \eqref{eqn:N2=N04} via the variable change \eqref{eqn:variablechangeN2=N0k}.
\end{proof}
For fixed vectors $\vec m^{(i)},\, i\neq k$, it's clear that $R^{(k)}=R_{am}^{(k)}$. Also the variables $q_i,\,i\neq k-1,\,k,\,k+1$ are not affected at all. 
Therefore $i_{Q_0}^* I^{F,S}(\vec q)=\sum_{(\vec m^{(i)})\in \op{Eff}^T_{bm}}R^{(k)}I^{(k)}\prod_{i=1}^{D-1}q_i^{\abs{\vec m^{(i)}}}$ and $i_{P_0}^*I^{\mc Z^1,S}(\vec q')=\sum_{(\vec m^{(i)})\in \op{Eff}^T_{am}}R^{(k)}_{am}I^{(k)}_{am}\prod_{i=1}^{D-1}(q_i')^{\abs{\vec m^{(i)}}} $ satisfy all relations in Lemma \ref{lem:mainlemmaflag}. 
The above procedure can be generalized to any pair of $S$-fixed points $Q\in \mathfrak{F}^{bm}, \, P=\iota(Q)$ easily, so we have proved Theorem \ref{thm:main} for $I^{F,S}(\vec q)$ and $I^{\mc Z^1,S}(\vec q')$.

To prove the relation between $I^{F,S}$ and $I^{\mc X^2,S}$ in Theorem \ref{thm:main}, 
we can replace $N_{k-1}$ in the above arguments by 0 naively and repeat the above procedure.

\subsection{Proofs for Theorem \ref{thm:main2}}
We use a similar method in the previous sub-section. 
For a pair of torus fixed points $Q\in \mathfrak{F}^{bm}$ and $P=\iota(Q)\in \mathfrak{F}^{am}$, we are going to study the relations of the restriction $i^*_QI^{tF, S^2}(\vec q)$, $i^*_PI^{\mc Z^3, S^2}(\vec q')$ and  $i^*_PI^{\mc Z^4, S^2}(\vec q')$. 

Note that $S_1^{\oplus N_0}\rightarrow F(N_0,\ldots, N_D)$ and $\mc Z^3$ are the local targets over $F(N_1,\ldots, N_D)$ and $\mc Z^1$. Comparing $I^{F,S}_{(\vec m^{(i)})}(\vec q)$ and
$I^{tF, S^2}_{(\vec m^{(i)})}(\vec q)$ in Lemma \ref{lem:Ifuncoftautflag} and Lemma \ref{lem:Itauflagam}, we can find that
they  differ by a twisted factor.
Also $I^{\mc Z^1,S}_{(\vec m^{(i)})}(\vec q')$ and $I^{\mc Z^3, S^2}_{(\vec m^{(i)})}(\vec q')$ differ by the same twisted factor. The proof in the previous subsection for Theorem \ref{thm:main} can be performed to $S_1^{\oplus N_0}\rightarrow F(N_0,\ldots, N_D)$ and $\mc Z^3$  directly without any obstacle. We omit this proof.  

We only prove the relations between $I^{tF, S^2}$ and $I^{\mc Z^4,S^2}$ in Theorem \ref{thm:main2} in detail. 
We prefer to work on a shorter quiver to make the proof easier, and it has no obstacle to generalizing it. 

Consider the quiver diagram below with $ N_1\leq N_2\leq N_3$ and $N_0\leq N_2$, whose quiver variety is $S_1^{\oplus N_0}\rightarrow F(N_1,N_2,N_3)$,
\begin{figure}[H]
    \centering
\includegraphics[width=2.5in]{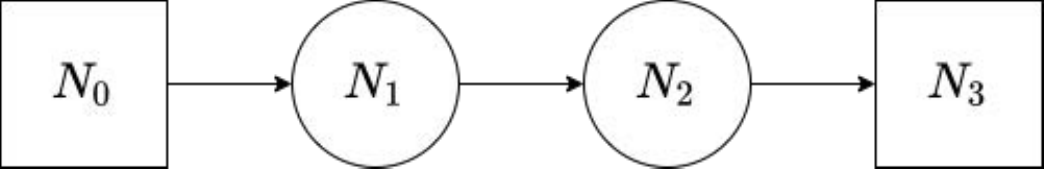}
\end{figure}
\noindent 
By Example \ref{ex:dualoftautflag}, performing quiver mutation at gauge node $k=1$, we get $\mc Z^4$ which is a subvariety in $\mc X^4$ defined by the equation $BA_1=0$. The quiver variety $\mc X^4$ is defined by the quiver diagram below, 
\begin{figure}[H]
    \centering
\includegraphics[width=2.5in]{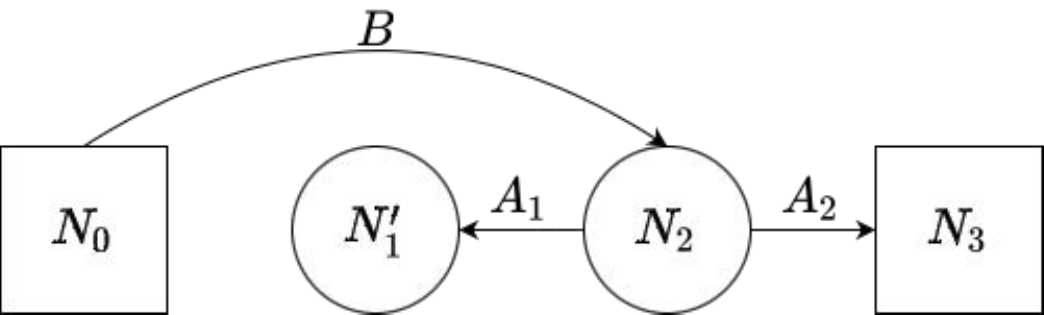}
\end{figure}
Without loss of generality, we consider a special $S^2$-fixed point in $\mathfrak F^{bm}$
\begin{equation}
    Q_{0}=([N_1]\subseteq [N_2]\subseteq [N_3])\,
\end{equation}
and the restriction of ${I}^{tF,S^2}(q_1, q_2)$ in Lemma \ref{lem:Ifuncoftautflag} to it. 
 Being restricted to $Q_0$, $x^{(i)}_{I}\rvert_{Q_0}=\lambda_I$ and 
 \begin{align}
    i^*_{Q_0}{I}^{tF,S^2}(q_1, q_2)&=\sum_{\substack{(\vec p,\vec m)\\ \in \op{Eff}^T_{bm} }}
   \prod_{I\neq J}^{N_2}\frac{\prod_{l\leq m_I-m_J}(\lambda_{I}-\lambda_{J}+l)}{\prod_{l\leq 0}(\lambda_{I}-\lambda_{J}+l)}\prod_{I=1}^{N_2}\prod_{J=1}^{N_3}\frac{1}{\prod_{l=1}^{m_I}(\lambda_{I}-\lambda_J+l)} q_2^{\abs{\vec m}}\nonumber\\
    &
    \prod_{I\neq J}^{N_1}\frac{\prod_{l\leq p_I-p_J}(\lambda_{I}-\lambda_{J}+l)}{\prod_{l\leq 0}(\lambda_{I}-\lambda_{J}+l)}
    \prod_{I=1}^{N_1}\prod_{F=1}^{N_2}
    \frac{\prod_{l\leq 0}(\lambda_I-\lambda_F+l)}{\prod_{l\leq p_I-m_F}(\lambda_I-\lambda_F+l)}\nonumber\\
    &\prod_{A=1}^{N_0}\prod_{I=1}^{N_1}\frac{\prod_{l\leq 0}(\eta_A-\lambda_I+l)}{\prod_{l\leq -
    p_I}(\eta_A-\lambda_I+l)}q_1^{\abs{\vec p}}\,.
\end{align} 
Fixing a degree $(\vec p, \vec m)\in \op{Eff}_{bm}^T$,
 we can split $i^*_{Q_0}I^{tF,S^2}_{(\vec p,\vec m)}$ into two parts
\begin{equation}
    i^*_{Q_0}I^{tF,S^2}_{(\vec p,\vec m)}=R^{(1)}
    I^{(1)}\,,
\end{equation}
where 
\begin{align}
I^{(1)}&=
       \prod_{I\neq J}^{N_1}\frac{\prod_{l\leq p_I-p_J}(\lambda_{I}-\lambda_{J}+l)}{\prod_{l\leq 0}(\lambda_{I}-\lambda_{J}+l)}
    \prod_{I=1}^{N_1}\prod_{F=1}^{N_2}
    \frac{\prod_{l\leq 0}(\lambda_I-\lambda_F+l)}{\prod_{l\leq p_I-m_F}(\lambda_I-\lambda_F+l)}\nonumber\\
    &\prod_{A=1}^{N_0}\prod_{I=1}^{N_1}\frac{\prod_{l\leq 0}(\eta_A-\lambda_I+l)}{\prod_{l\leq -
    p_I}(\eta_A-\lambda_I+l)}\,,
\end{align}
and 
\begin{equation}
R^{(1)}=
    \prod_{I\neq J}^{N_2}\frac{\prod_{l\leq m_I-m_J}(\lambda_{I}-\lambda_{J}+l)}{\prod_{l\leq 0}(\lambda_{I}-\lambda_{J}+l)}\prod_{I=1}^{N_2}\prod_{J=1}^{N_3}\frac{1}{\prod_{l=1}^{m_I}(\lambda_{I}-\lambda_J+l)}\,.
\end{equation}
In the expression of $I^{(1)}$, we have $n_I:=p_I-m_{I}\geq 0$, otherwise $I^{(1)}=0$. Replacing $p_I$ by $n_I+m_I$, we have 
\begin{subequations}
 \begin{align}
    {I}^{(1)}&=
   \prod_{I\neq J}^{N_1}\frac{\prod_{l\leq n_I-n_J}\lambda_{I}-\lambda_{J}+m_I-m_J+l}{\prod_{l\leq 0}\lambda_{I}-\lambda_{J}+m_I-m_J+l}
    \prod_{I=1}^{N_1}\frac{\prod_{A=1}^{N_0}\prod_{l=0}^{n_I-1}(\eta_A-\lambda_I-m_I-l))}{\prod_{F=1}^{N_2}\prod_{l=1}^{n_I}(\lambda_I-\lambda_F+m_I-m_F+l)}\label{eqn:Itaut1}\\
    &\prod_{I=1}^{N_1}\prod_{J=1}^{N_1'}\frac{\prod_{l\leq 0}(\lambda_I-\lambda_{N_1+J}+l)}{\prod_{l\leq m_I-m_{N_1+J}}(\lambda_I-\lambda_{N_1+J}+l)}\prod_{A=1}^{N_0}\prod_{I=1}^{N_1}\frac{\prod_{l\leq 0}(\eta_A-\lambda_I+l)}{\prod_{l\leq -m_I}(\eta_A-\lambda_I+l)}q_{1}^{\abs{\vec{n}}}\label{eqn:Itaut2}\,.
\end{align} 
\end{subequations}
We can find that the sub-equation \eqref{eqn:Itaut1} is the degree $\vec{n}$ term of the $I$-function of $S_1^{\oplus N_0}\rightarrow Gr(N_1,N_2)$ if we pretend that $\lambda_I+m_I$ are equivariant parameters of the torus $(\C^*)^{N_2}$ and  $\eta_A$ are still equivariant parameters of the torus $(\C^*)^{N_0}$. 

On the other side, consider the restriction of $I^{\mc Z^4,S^2}$ in Lemma \ref{lem:IofZ4} to the torus fixed point $P_0:=\iota(Q_0)=\{[N_2]\backslash [N_1]\subseteq [N_2]\subseteq [N_3]\}$. 
Then $x^{(1)}_I\rvert_{P_0}=\lambda_{N_1+I}$, $x_J^{(2)}\rvert_{P_0}=\lambda_J$, and
\begin{align}
    i_{P_0}^*{I}^{\mc Z^4,S^2}(\vec q')=&
    \sum_{\substack{(\vec p,\vec m)\\ \in \op{Eff}^T_{am} }}
   \prod_{I\neq J}^{N_2}\frac{\prod_{l\leq m_I-m_J}(\lambda_{I}-\lambda_{J}+l)}{\prod_{l\leq 0}(\lambda_{I}-\lambda_{J}+l)}\prod_{I=1}^{N_2}\prod_{J=1}^{N_3}\frac{1}{\prod_{l=1}^{m_I}(\lambda_{I}-\lambda_J+l)} q_2'^{\abs{\vec m}}
    \nonumber\\
    &\prod_{I\neq J}^{N_1'}\frac{\prod_{l\leq p_I-p_J}(\lambda_{N_1+I}-\lambda_{N_1+J}+l)}{\prod_{l\leq 0}(\lambda_{N_1+I}-\lambda_{N_1+J}+l)}\prod_{I=1}^{N_1'}\prod_{F=1}^{N_2}
    \frac{\prod_{l\leq 0}(-\lambda_{N_1+I}+\lambda_F+l)}{\prod_{l\leq -p_I+m_F}(-\lambda_{N_1+I}+\lambda_F+l)}\nonumber\\
    & \prod_{A=1}^{N_0}\prod_{I=1}^{N_1'}\frac{\prod_{l\leq -p_I}(\eta_A-
    \lambda_{N_1+I}+l)}{\prod_{l\leq 0}(\eta_A-\lambda_{N_1+I}+l)}
    \prod_{A=1}^{N_0}\prod_{F=1}^{N_2}\frac{\prod_{l\leq 0}(\eta_A-\lambda_F+l)}{\prod_{l\leq -m_F}(\eta_A-\lambda_F+l)} q_1'^{\abs{\vec p}}\,.
\end{align}
Fix a degree $(\vec p, \vec m)\in \op{Eff}_{am}^T$. 
The term  $i_{P_0}^*I^{\mc Z^4}_{(\vec p,\vec m)}$  can be split into two parts,
\begin{equation}
    i^*_{P_0}I^{\mc Z^4,S^2}_{(\vec p,\vec m)}=R_{am}^{(1)}I^{(1)}_{am}\,,
\end{equation}
where 
\begin{align}
    I^{(1)}_{am}=
    &\prod_{I\neq J}^{N_1'}\frac{\prod_{l\leq p_I-p_J}(\lambda_{N_1+I}-\lambda_{N_1+J}+l)}{\prod_{l\leq 0}(\lambda_{N_1+I}-\lambda_{N_1+J}+l)}\prod_{I=1}^{N_1'}\prod_{F=1}^{N_2}
    \frac{\prod_{l\leq 0}(-\lambda_{N_1+I}+\lambda_F+l)}{\prod_{l\leq -p_I+m_F}(-\lambda_{N_1+I}+\lambda_F+l)}\nonumber\\
    & \prod_{A=1}^{N_0}\prod_{I=1}^{N_1'}\frac{\prod_{l\leq -p_I}(\eta_A-
    \lambda_{N_1+I}+l)}{\prod_{l\leq 0}(\eta_A-\lambda_{N_1+I}+l)}
    \prod_{A=1}^{N_0}\prod_{F=1}^{N_2}\frac{\prod_{l\leq 0}(\eta_A-\lambda_F+l)}{\prod_{l\leq -m_F}(\eta_A-\lambda_F+l)}\,,
\end{align}
and $R^{(1)}_{am}$ is the remaining part. 
 Notice that in the expression of $I^{(1)}_{am}$,  $n_I=-p_I+m_{N_1+I}\geq 0$. Making replacement $p_I=m_{N_1+I}-n_I$, we can transform $I^{(1)}_{am}$ to
\begin{subequations}
\begin{align}
   I^{(1)}_{am}=&
   \prod_{I\neq J}^{N_1'}\frac{\prod_{l\leq -n_I+n_J}\lambda_{N_1+I}-\lambda_{N_1+J}+m_{N_1+I}-m_{N_1+J}+l}{\prod_{l\leq 0}\lambda_{N_1+I}-\lambda_{N_1+J}+m_{N_1+I}-m_{N_1+J}+l}\label{Itautflag:adual0} \\
    &\prod_{I=1}^{N_1'}
    \frac{\prod_{A=1}^{N_0}\prod_{l=1}^{ n_I}(\eta_A-\lambda_{N_1+I}-m_{N_1+I}+l)}
    {\prod_{F=1}^{N_2}\prod_{l=1}^{ n_I}(-\lambda_{N_1+I}-m_{N_1+I}+\lambda_F+m_F+l)} \label{Itautflag:adual1}  \\
    &
    \prod_{I=1}^{N_1}\prod_{J=1}^{N_1'}
    \frac{\prod_{l\leq 0}(\lambda_I-\lambda_{N_1+J}+l)}{\prod_{l\leq m_I-m_{{N_1+J}}}(\lambda_I-\lambda_{N_1+J}+l)}
    \prod_{A=1}^{N_0}\prod_{J=1}^{N_1}\frac{\prod_{l\leq 0}(\eta_A-\lambda_J+l)}{\prod_{l\leq -m_J}(\eta_A-\lambda_J+l)}
\label{Itautflag:adual2}\,.
\end{align}
\end{subequations}
 One can find that the sub-equations \eqref{Itautflag:adual0} and \eqref{Itautflag:adual1} together can be viewed as the degree ${\vec n}$ term of the ${I}$-function of $(S_1^\vee)^{\oplus N_0}\rightarrow Gr(N_1',N_2)$, if we view $\lambda_F+m_F$ as the equivariant parameters of the torus $(\C^*)^{N_2}$. 
 
 Let $\vec m$ fixed. Consider the sum of $I^{(1)}$ over all possible $\vec p\in \Z^{N_1}$, such that $(\vec p, \vec m)\in \op{Eff}_{bm}^T$, and denote it by 
 \begin{equation}
     I^{(1)}(\vec q)=\sum_{\vec p}I^{(1)}q_1^{\abs{\vec p}}q_2^{\abs{\vec m}}\,.
 \end{equation}
Similarly, consider the sum of $I^{(1)}_{am}$ over all possible $\vec{\tilde p}\in \Z^{N_1'}$ where $(\vec{\tilde p}, \vec m)\in \op{Eff}_{am}^T$, and denote it by 
\begin{equation}
    I^{(1)}_{am}(\vec q')=\sum_{\vec{\tilde p}}I^{(1)}_{am}(q_1')^{\abs{\vec {\tilde p}}}(q_2')^{\abs{\vec m}}\,.
\end{equation}
\begin{lem}
$ I^{(1)}(\vec q)$ and $I^{(1)}_{am}(\vec q')$ satisfy the following relations.
\begin{enumerate}
    \item When $N_2\geq N_0+2$, 
    \begin{equation}\label{eqn:lemtautflag1}
       I^{(1)}(\vec q)=I^{(1)}_{am}(\vec q')\,,
    \end{equation}
    via the variable change 
    \begin{equation}
        q_1'=q_1^{-1}, q_2'=q_2q_1\,.
    \end{equation}
    \item When $N_2=N_0+1$, 
       \begin{equation}\label{eqn:lemtautflag2}
       I^{(1)}(\vec q)=e^{(-1)^{N_1'}q_1}I^{(1)}_{am}(\vec q')\,,
    \end{equation}
    via the variable change
    \begin{equation}
         q_1'=q_1^{-1}, q_2'=q_2q_1\,. 
    \end{equation}
    \item When $N_2=N_0$, 
      \begin{equation}\label{eqn:N2=N0}
      I^{(1)}(\vec q)=(1+(-1)^{N_1'}q_1)^{\sum_{A=1}^{N_0}\eta_A-\sum_{F=1}^{N_2} \lambda_F+N_1'}I^{(1)}_{am}(\vec q')\,,
    \end{equation}
    via variable change 
    \begin{equation}\label{eqn:variablechangeN2=N0}
        q_1'=q_1^{-1},q_2=\frac{q_1q_2}{1+(-1)^{N_1'}q_1}\,.
    \end{equation}
In the above expression, the formula
\begin{equation}
    (1+(-1)^{N_1'}q_1)^{\sum_{A=1}^{N_0}\eta_A-\sum_{F=1}^{N_2} \lambda_F+N_1'}
\end{equation}
 is formally expanded as 
  \begin{equation}
   \sum_{m\geq 0}\frac{\prod_{l=0}^{m-1}(\sum_{A=1}^{N_0}\eta_A-\sum_{F=1}^{N_2}\lambda_{F}+N_1'-l)}{m!} ((-1)^{N_1'}q_1)^m\,.
 \end{equation}
\end{enumerate}
\end{lem}
\begin{proof}
$I^{(1)}(\vec q)$ can be transformed to,
\begin{align}
    I^{(1)}(\vec q)&=\sum_{\vec n\in \Z_{\geq 0}^{N_1}} i_{Q_0}^*I^{Gr, S^2}_{\vec n}q_1^{\abs{\vec n}+\sum_{I=1}^{N_1}m_{I}}q_2^{\abs{\vec m}}\nonumber\\
    &\times \prod_{I=1}^{N_1}\prod_{J=1}^{N_1'}
    \frac{\prod_{l\leq 0}(\lambda_I-\lambda_{N_1+J}+l)}{\prod_{l\leq m_I-m_{{N_1+J}}}(\lambda_I-\lambda_{N_1+J}+l)}
    \prod_{A=1}^{N_0}\prod_{J=1}^{N_1}\frac{\prod_{l\leq 0}(\eta_A-\lambda_J+l)}{\prod_{l\leq -m_J}(\eta_A-\lambda_J+l)}\,,
\end{align}
where $I^{Gr, S^2}_{\vec n}$ is the degree $\vec n$-term of the equivariant quasimap small $I$-function of $S_1^{\oplus N_0}\rightarrow Gr(N_1,N_2)$, and $Q_0=([N_1]\subseteq [N_2])$.
Similarly, 
\begin{align}
    I^{(1)}_{am}(\vec q')&=\sum_{\vec n\in \Z_{\geq 0}^{N_1'}} i^*_{P_0}I^{Gr^\vee, S^2}_{\vec n}(q_1')^{-\abs{\vec n}+\sum_{I=1}^{N_1'}m_{N_1+I}}(q_2')^{\abs{\vec m}}\nonumber\\
    &\times \prod_{I=1}^{N_1}\prod_{J=1}^{N_1'}
    \frac{\prod_{l\leq 0}(\lambda_I-\lambda_{N_1+J}+l)}{\prod_{l\leq m_I-m_{{N_1+J}}}(\lambda_I-\lambda_{N_1+J}+l)}
    \prod_{A=1}^{N_0}\prod_{J=1}^{N_1}\frac{\prod_{l\leq 0}(\eta_A-\lambda_J+l)}{\prod_{l\leq -m_J}(\eta_A-\lambda_J+l)}\,,
\end{align}
where $I^{Gr^\vee, S^2}_{\vec n}$ is the degree $\vec n$-term of the equivariant quasimap small $I$-function of $(S_1^\vee)^{N_0}\rightarrow Gr(N_1',N_2)$, and $P_0=\iota(Q_0)$.

When $N_2\geq N_0+1$, $I^{(1)}(\vec q)$ and $I^{(1)}_{am}(\vec q')$ satisfy the relations \eqref{eqn:lemtautflag1} and \eqref{eqn:lemtautflag2} by exactly the same arguments in Lemma \ref{lem:mainlemmaflag}.
When $N_2=N_0$, 
by Theorem \ref{thm:Haidong},
\begin{align}
    I^{(1)}(\vec q)&=(1+(-1)^{N_1'}q_1)^{\sum_{A=1}^{N_0}\eta_A-\sum_{F=1}^{N_2}\lambda_F+N_1'}\sum_{\vec n\in \Z_{\geq 0}^{N_1'}} i_{P_0}^*I^{G^\vee, S^2}_{\vec n}q_1^{\abs{\vec n}+\sum_{I=1}^{N_1}m_{I}}(\frac{q_2}{1+(-1)^{N_1'}q_1})^{\abs{\vec m}}
    \,\nonumber\\
    &\times \prod_{I=1}^{N_1}\prod_{J=1}^{N_1'}
    \frac{\prod_{l\leq 0}(\lambda_I-\lambda_{N_1+J}+l)}{\prod_{l\leq m_I-m_{{N_1+J}}}(\lambda_I-\lambda_{N_1+J}+l)}
    \prod_{A=1}^{N_0}\prod_{J=1}^{N_1}\frac{\prod_{l\leq 0}(\eta_A-\lambda_J+l)}{\prod_{l\leq -m_J}(\eta_A-\lambda_J+l)}\,\nonumber\\
    &=(1+(-1)^{N_1'}q_1)^{\sum_{A=1}^{N_0}\eta_A-\sum_{F=1}^{N_2}\lambda_F+N_1'} I^{(1)}_{am}(\vec q')\,,
\end{align}
with $q_1'=q_1^{-1},\, q_2'=\frac{q_1q_2}{1+(-1)^{N_1'}q_1}$.
Therefore we have proved the case for $N_0=N_2$.
\end{proof}
Notice that $R^{(1)}=R^{(1)}_{am}$ no matter what $\vec m$ is, and the above lemma holds for an arbitrary vector $\vec m\in \Z_{\geq 0}^{N_2}$. Hence  $i_{Q_0}^*I^{tF,S^2}(\vec q)$ and $i_{P_0}^*I^{\mc Z^4,S^2}(\vec q')$ satisfy the relations between $I^{(1)}(\vec q)$ and $I^{(q)}_{am}(\vec q')$.
Being restricting to any other pair of torus fixed points $Q$ and $P:=\iota(Q)$, $i^*_QI^{tF,S^2}(q_1,q_2)$ and $i^*_PI^{\mc Z^4,S^2}(q_1',q_2')$ still satisfy the relations in the above lemma. Hence, we have proved  Theorem \ref{thm:main2}.

\appendix
\addcontentsline{toc}{section}{Appendices}
\section{Proof of Theorem \ref{thm:Haidong}}
Define Pochhammer symbol $(a)_n=\prod_{i=0}^{n-1}(a+i)=\frac{\Gamma(a+n)}{\Gamma(a)}$. 
We first list some useful equations for Pochhammer symbol.
\begin{lem}
\begin{enumerate}
    \item For any integer $n\in \Z_{\geq 0}$,
    \begin{equation}\label{eqn:A1}
    (-a)_n=(-1)^{n}(a+1-n)_n\,.
    \end{equation}
    \item Assume $x_1$ and $x_2$ are distinct variables and $d_1,d_2\in Z_\geq 0$ are positive integers. Then we have the following equation,
    \begin{align}\label{eqn:A2}
        \frac{x_1-x_2+d_1-d_2}{(x_1-x_2)(x_1-x_2+1)_{d_1}(x_2-x_1+1)_{d_2}}=\frac{1}{(x_1-x_2-d_2)_{d_1}(x_2-x_1-d_1)_{d_2}}\,.
    \end{align}
    \item Suppose $a$ is a variable, and $m,\,l$ are positive integers, then
    \begin{equation}\label{eqn:A7}
   1+\frac{m-l}{a}=\frac{(a+1)_m(-a+1)_l}{(a-l)_m(-a-m)_l}\,.
    \end{equation}
\end{enumerate}
\end{lem}
\begin{proof}
     The first statement directly results from the definition, so we omit it.
     
     To prove the second statement, we firstly claim the following equation,
    \begin{align}\label{eqn:A4}
       \frac{x_1-x_2+d_1-d_2}{x_1-x_2}=(-1)^{d_1-d_2}\frac{\prod_{l\leq d_1-d_2}(x_1-x_2+l)}{\prod_{l\leq 0}(x_1-x_2+l)}
      \times\frac{\prod_{l\leq d_2-d_1}(x_2-x_1+l)}{\prod_{l\leq 0}(x_2-x_1+l)}\,.
    \end{align}
    The case for $d_1=d_2$ is trivial.
 Let $d_1>d_2$ without loss of generality, we have
    \begin{align}
        &\frac{\prod_{l\leq d_1-d_2}(x_1-x_2+l)}{\prod_{l\leq 0}(x_1-x_2+l)}
      \cdot\frac{\prod_{l\leq d_2-d_1}(x_2-x_1+l)}{\prod_{l\leq 0}(x_2-x_1+l)}=\frac{\prod_{1\leq l\leq d_1-d_2}(x_1-x_2+l)}{\prod_{d_2-d_1+1\leq l \leq 0}(x_2-x_1+l)}\nonumber\\
      &=(-1)^{d_1-d_2}\frac{\prod_{1\leq l\leq d_1-d_2}(x_1-x_2+l)}{\prod_{0\leq l \leq d_1-d_2-1}(x_1-x_2+l)}=(-1)^{d_1-d_2} \frac{x_1-x_2+d_1-d_2}{x_1-x_2}\,,
    \end{align}
    which proves the Equation \eqref{eqn:A4}.
Then the left hand side of \eqref{eqn:A2} is equal to 
    \begin{align}
       &(-1)^{d_1-d_2}\frac{\prod_{l\leq d_1-d_2}(x_1-x_2+l)}{\prod_{l\leq 0}(x_1-x_2+l)}
      \frac{\prod_{l\leq d_2-d_1}(x_2-x_1+l)}{\prod_{l\leq 0}(x_2-x_1+l)}\frac{\prod_{l\leq 0}(x_1-x_2+l)}{\prod_{l\leq d_1}(x_1-x_2+l)}
      \frac{\prod_{l\leq 0}(x_2-x_1+l)}{\prod_{l\leq d_2}(x_2-x_1+l)}\nonumber\\
      &=(-1)^{d_1-d_2}\frac{1}{\prod_{d_1-d_2+1\leq l
      \leq d_1}(x_1-x_2+l)}\frac{1}{\prod_{d_2-d_1+1\leq l\leq d_2}(x_2-x_1+l)}\nonumber\\
      &=(-1)^{d_1-d_2}\frac{1}{(x_1-x_2+d_1-d_2+1)_{d_2}}\frac{1}{(x_2-x_1+d_2-d_1+1)_{d_1}}\nonumber\\
      &=\frac{1}{(x_2-x_1-d_1)_{d_2}(x_1-x_2-d_2)_{d_2}}\,,
    \end{align}
where we have applied the Equation \eqref{eqn:A1} in the last step.

The third statement is a special situation of the second one by taking $x_1=a$, $x_2=0$ and $d_1=m,\,d_2=l$.
\end{proof}

The following will be devoted to proofs of the Theorem \ref{thm:Haidong}. We will use notations there. 
Consider an $S^2$-fixed point $Q=( \vec C_{[N_1]}\subset [N_2])$, and $P=\iota(Q)=(\{f_1'<\ldots<f_{N_1'}'\}\subset [N_2])$. 
Let $I^{Gr,S^2}_d$ and $I^{Gr^\vee, S^2}_d$ be the coefficient of $q^d$ in $i_Q^*(I^{Gr, S^2})$ and coefficient of 
$(q')^{-d}$ in $i_P^*(I^{Gr^\vee, S^2})$. 
We need to do some combinatorics to simplify them.
For $I^{Gr,S^2}_d$, we are able to rewrite its "abelianization translator" as follows
by the Equation \eqref{eqn:A4}, 
\begin{equation}
     \prod_{\substack{I,J=1\\I\neq J}}^{N_1}\frac{\prod_{\leq d_I-d_J}(\lambda_{f_I}-\lambda_{f_J}+l)}{\prod_{l\leq 0}(\lambda_{f_I}-\lambda_{f_J}+l)}
     =
    (-1)^{(N-1)\abs{\vec d}}
    \prod_{\substack{I,J=1\\I<J}}^{N_1}\frac{\lambda_{f_I}-\lambda_{f_J}+d_I-d_J}{\lambda_{f_I}-\lambda_{f_J}}\,.
\end{equation}
Hence,
\begin{align}\label{eqnA:IGr1}
    &I_d^{Gr, S^2}=\sum_{\abs{\vec d}=d}(-1)^{(N_1-1)d}
    \prod_{I<J}^{N_1}\frac{\lambda_{f_I}-\lambda_{f_J}+d_I-d_J}{\lambda_{f_I}-\lambda_{f_J}}
    \prod_{I=1}^{N_1}\frac{\prod_{A=1}^{N_0}\prod_{l=0}^{d_I-1}(-\lambda_{f_I}+\eta_A-l)}
    {\prod_{F=1}^{N_2}\prod_{l=1}^{d_I}(\lambda_{f_I}-\lambda_F+l)}\,\nonumber\\
    &=(-1)^{(N_1-1)d}\sum_{\abs{\vec d}=d}
    \prod_{I=1}^{N_1}
    \left(\prod_{I<J}^{N_1}\frac{\lambda_{f_I}-\lambda_{f_J}+d_I-d_J}{\lambda_{f_I}-\lambda_{f_J}}\frac{\prod_{A=1}^{N_0}\prod_{l=0}^{d_I-1}(-\lambda_{f_I}+\eta_A-l)}
    {\prod_{J=1}^{N_1}(\lambda_{f_I}-\lambda_{f_J}+1)_{d_I}\prod_{J=1}^{N_1'}(\lambda_{f_I}-\lambda_{f'_J}+1)_{d_I}}\right)\,.
\end{align}
Applying formula \eqref{eqn:A2}, we change $I_d^{Gr,S^2}$ to
\begin{align}
    I_d^{Gr,S^2}
    =\sum_{\abs{\vec d}=d}(-1)^{(N_1-1+N_0)d} \prod_{\substack{I,J=1\\I\neq J}}^{N_1}
    \frac{1}{(\lambda_{f_J}-\lambda_{f_I}-d_I)_{d_J}}\prod_{I=1}^{N_1}\frac{\prod_{A=1}^{N_0}(\lambda_{f_I}-\eta_A)_{d_I}}{d_I!\prod_{J=1}^{N_1'}(\lambda_{f_I}-\lambda_{f_J'}+1)_{d_I}}\,.
\end{align}
Similarly, we repeat the above procedure to 
\begin{equation}
    I_d^{Gr^\vee,S^2}=\sum_{\abs{\vec d}=d}
    \prod_{I\neq J}^{N_1'}\frac{\prod_{l\leq d_J-d_I}(\lambda_{f_I'}-\lambda_{f_J'}+l)}{\prod_{l\leq 0}(\lambda_{f_I'}-\lambda_{f_J'}+l)}
    \prod_{I=1}^{N_1'}\frac{\prod_{A=1}^{N_0}\prod_{l=1}^{d_I}(-\lambda_{f_I'}+\eta_A+l)}{ \prod_{F=1}^{N_2}\prod_{l=1}^{d_I}(-\lambda_{f_I'}+\lambda_F+l)}\,,
\end{equation}
and change $I^{Gr^\vee,S^2}_d$ to 
\begin{align}
    I_n^{Gr^\vee,S^2}=\sum_{\abs{\vec d}=d}
    (-1)^{(N_1'-1)d} \prod_{I\neq J=1}^{N_1'}
    \frac{1}{(\lambda_{f_J'}-\lambda_{f_I'}-d_J)_{d_I}}\prod_{I=1}^{N_1'}\frac{\prod_{A=1}^{N_0}(-\lambda_{f_I'}+\eta_A+1)_{d_I}}{d_I!\prod_{F=1}^{N_1}(\lambda_{f_F}-\lambda_{f_I'}+1)_{d_I}}\,.
    \end{align}
\begin{lem}
Suppose $N_2\geq N_0+2$, then
\begin{equation}\label{eqnA:lemma1}
  I^{Gr^\vee,S^2}_d=I^{Gr,S^2}_d\,.
\end{equation}
The above equation induces the first statement of Theorem \ref{thm:Haidong}.
\end{lem}
\begin{proof}
Define a meromorphic function,
\begin{align}\label{eqn:f}
    f(\vec \phi)=\prod_{\alpha=1}^{d}\prod_{\alpha<\beta}^{d}
    \frac{\phi_{\alpha}-\phi_{\beta}}{\phi_{\alpha}-\phi_{\beta}+1}
    \frac{\phi_\beta-\phi_\alpha}{\phi_\beta-\phi_\alpha+1}\prod_{\alpha=1}^{d}\frac{\prod_{A=1}^{N_0}(\eta_A-\phi_\alpha)}{\prod_{I=1}^{N_1}(\lambda_{f_I}-\phi_\alpha)\prod_{J=1}^{N_1'}(\phi_\alpha-\lambda_{f_{J}'}+1)}\,.
\end{align}
Suppose that 
\begin{equation} 
    0<Re\lambda_I<1,\,\, \lambda_I-\lambda_J\notin \Z, \text{ for } I\neq J\,.
\end{equation}
Note that $f(\vec \phi)$ is symmetric with respect to its variables. 
Let $C$ be the contour from $-i\infty$ to $i\infty$ and $C^-$ is in opposite direction, and then the following equation always holds,
\begin{equation}\label{eqn:prf1}
    \frac{1}{d!}\frac{1}{(2\pi i)^{d}}\int_{C}\cdots\int_{C}f(\vec \phi)d\phi_1\cdots d\phi_{d}=
    \frac{1}{d!}\frac{(-1)^{d}}{ (2\pi i)^{d}}\int_{C^-}\cdots\int_{C^-}f(\vec \phi)d\phi_1\cdots d\phi_{d}\,.
\end{equation}
Let $C_{R}^+$ be a semicircle in the left half-plane, and $C_{R}^-$ be a semicircle in the right half-plane and assume that both are centered at zero and have radius $R$.
Since $N_2\geq N_0+2$,  both $\int_{C_R-}f(\vec \phi)d\phi_i$ and $\int_{C_R^+}f(\vec \phi)d\phi_i$ approach to zero as $R$ goes to infinity. 
Hence the left-hand side of \eqref{eqn:prf1} is equal to the sum of residues of $f$ at all poles that are in the left half-plane, and the right-hand side of \eqref{eqn:prf1} is equal to the sum of residues at all poles that are in the right half-plane. 
Again since $N_2\geq N_0+2$,  we are free to exchange the integration order.

The poles on the right half plane can be classified as 
\begin{align}
    & \phi_\alpha=\lambda_{f_I}\,,\\
    &\phi_{\alpha}=\phi_{\beta}\,.
\end{align}
The above classification of poles gives a partition of the $d$. 
More explicitly, the poles of $f(\vec{\phi}_\alpha)$ in the right half plane are 
\begin{equation}\label{eqn:polesr}
   \{\lambda_{f_I}+n_I, \text{ for } I=1, \ldots, N_1, n_I=0,\ldots,d_I-1\}\,,
\end{equation}
with $d_1+\ldots+d_{N_1}=d$.
We can permute the pole of each variable $\phi_{\alpha}$ in the set \eqref{eqn:polesr}, and that is the reason why there is a factor $\frac{1}{d!}$ in front of the integral of \eqref{eqn:prf1}. 
Make a variable change 
\begin{equation}
    \phi_\alpha=\lambda_{f_I}+n_I+s_{n_I}^{I}\,,\, \text{for }\alpha=1,\ldots, d\,; \,
    I=1,\ldots,N_1\,;\, 
    n_I=0,\ldots,d_I-1\,.
\end{equation}
Then the right hand side of \eqref{eqn:prf1} is
\begin{subequations}
\begin{align}
   & {(-1)^d}\sum_{\abs{\vec d}=d}\int_{S_{\epsilon}^{\otimes d}}\prod_{I=1}^{N_1}\prod_{n_I=0}^{d_I-1}\frac{ds^{I}_{n_I}}{2\pi i}\\
   & \times \prod_{I\neq J}\prod_{n_I,n_J}\frac{\lambda_{f_I}-\lambda_{f_J}+n_I-n_J+s^{I}_{n_I}- s^{J}_{n_J}}{\lambda_{f_I}-\lambda_{f_J}+n_I-n_J+s^{I}_{n_I}- s^{J}_{n_J}+1}
     \prod_{I=1}^{N_1}\prod_{\substack{n_I,m_I=0,\\m_I\neq n_I}}^{d_I-1}\frac{n_I-m_I+s^{I}_{n_I}- s^{I}_{m_I}}{n_I-m_I+s^{I}_{n_I}- s^{I}_{m_I}+1}\label{subeqnA:int2}\\
   &\times \prod_{I=1}^{N_1}\prod_{n_I=0}^{d_I-1}
     \frac{\prod_{A=1}^{N_0}(\eta_A-\lambda_{f_I}-n_I-s^{I}_{n_I})}{\prod_{J=1}^{N_1}(\lambda_{f_I}-\lambda_{f_J}-n_I-s^I_{n_I})\prod_{K=1}^{N_1'}(\lambda_{f_I}-\lambda_{f'_K}+n_I+1+s^I_{n_I})}\,,\label{subeqnA:int3}
     \end{align}
\end{subequations}
where $S_{\epsilon}^{\otimes d}$ represent $d$ copies of circles of radius $\epsilon$ around the origin.
Each $s^I_{n_I}$ has a simple pole around the origin. 
When $m_I=n_I+1$ in \eqref{subeqnA:int2}, the factor $\prod_{I=1}^{N_1} \prod_{n_I=1}^{d_I} \frac{s^{I}_{n_I}- s^{I}_{n_I+1}-1}{s^{I}_{n_I}- s^{I}_{n_I+1}}$ admit poles, and when $I=J,\,n_I=0$ in \eqref{subeqnA:int3}, 
factors   $\prod_{I=1}^{N_1}\frac{1}{\lambda_{f_I}-\lambda_{f_J}-n_I-s^I_{n_I}} =\prod_{I=1}^{N_1}\frac{1}{-s^I_{0}}$ admit poles. Hence by splitting the holomorphic part, the above formula is
\begin{subequations}
\begin{align}
  & (-1)^d\sum_{\abs{\vec d}=d}\int_{S_{\epsilon}^{\otimes d}}\prod_{I=1}^{N_1}\frac{-1}{s^I_0} \prod_{n_I=0}^{d_I-1}\frac{s^{I}_{n_I}- s^{I}_{n_I+1}-1}{s^{I}_{n_I}- s^{I}_{n_I+1}} \frac{ds^{I}_{n_I}}{2\pi i} \label{subeqnA:int4}\\
   & \times \prod_{I\neq J}\prod_{n_I=0}^{d_I-1}\prod_{n_J=0}^{d_J-1}\frac{\lambda_{f_I}-\lambda_{f_J}+n_I-n_J+s^{I}_{n_I}- s^{J}_{n_J}}{\lambda_{f_I}-\lambda_{f_J}+n_I-n_J+s^{I}_{n_I}- s^{J}_{n_J}+1}
     \prod_{I=1}^{N_1}\prod_{\substack{n_I,m_I=0,\\m_I\neq n_I,\\m_I\neq n_I+1}}^{d_I-1}\frac{n_I-m_I+s^{I}_{n_I}- s^{I}_{m_I}}{n_I-m_I+s^{I}_{n_I}- s^{I}_{m_I}+1} \label{subeqnA:int5}\\
     &\times \prod_{I=1}^{N_1}\prod_{n_I=0}^{d_I-1}
     \frac{\prod_{A=1}^{N_0}(\eta_A-\lambda_{f_I}-n_I-s^{I}_{n_I})}
     {\prod_{J\neq I}^{N_1}(\lambda_{f_I}-\lambda_{f_J}-n_I-s^I_{n_I})\prod_{n_I=1}^{d_I-1}(-n_I-s^I_{n_I})\prod_{K=1}^{N_1'}(\lambda_{f_I}-\lambda_{f'_K}+n_I+1+s^I_{n_I})}\,.\label{subeqnA:int6}
\end{align}
\end{subequations}
One can check that the integral \eqref{subeqnA:int4} is $(-1)^{N_1}$, and the two rows \eqref{subeqnA:int5} and \eqref{subeqnA:int6} are holomorphic around $s^I_{n_I}=0$.
Therefore, after taking residues, the above integral is equal to 
\begin{subequations}
\begin{align}
  &(-1)^{d+N_1}\sum_{\abs{\vec d}=d}
   \prod_{I\neq J}\prod_{n_I=0}^{d_I-1}\prod_{n_J=0}^{d_J-1}\frac{\lambda_{f_I}-\lambda_{f_J}+n_I-n_J}{\lambda_{f_I}-\lambda_{f_J}+n_I-n_J+1}
     \prod_{I=1}^{N_1}\prod_{\substack{m_I,n_I=0,\\m_I\neq n_I,\\m_I\neq n_I+1}}^{d_I-1}\frac{n_I-m_I}{n_I-m_I+1} \label{subeqnA:int8}\\
     &\times \prod_{I=1}^{N_1}
     \frac{\prod_{n_I=0}^{d_I-1}\prod_{A=1}^{N_0}(\eta_A-\lambda_{f_I}-n_I)}
     {(-1)^{d_I-1}(d_I-1)!\prod_{J\neq I}^{N_1}(-1)^{d_I}(\lambda_{f_J}-\lambda_{f_I})_{d_I}\prod_{K=1}^{N_1'}(\lambda_{f_I}-\lambda_{f'_K}+1)_{d_I}}\,.\label{subeqnA:int9}
\end{align}
\end{subequations}
In the above formula, one can check that 
\begin{equation}
     \prod_{I\neq J}^{N_1}\prod_{n_I=0}^{d_I-1}\prod_{n_J=0}^{d_J-1}\frac{\lambda_{f_I}-\lambda_{f_J}+n_I-n_J}{\lambda_{f_I}-\lambda_{f_J}+n_I-n_J+1}
     =\prod_{I\neq J}\frac{(\lambda_{f_J}-\lambda_{f_I})_{d_J}}{(\lambda_{f_J}-\lambda_{f_I}-d_I)_{d_J}}\,,
\end{equation}
and 
\begin{equation}
    \prod_{I=1}^{N_1}\prod_{n_I=0}^{d_I-1}\prod_{\substack{m_I\neq n_I,\\m_I\neq n_I+1}}^{d_I}\frac{n_I-m_I}{n_I-m_I+1}=\prod_{I=1}^{N_1}\frac{1}{d_I}\,.
\end{equation}
Substituting the two formulae back to \eqref{subeqnA:int8} and \eqref{subeqnA:int9}, we can prove that it equals
\begin{equation}
    \sum_{\abs{\vec d}=d}(-1)^{(N_1-1+N_0)\abs{\vec d}} \prod_{I\neq J=1}^{N_1}
    \frac{1}{(\lambda_{f_J}-\lambda_{f_I}-d_I)_{d_J}}\prod_{I=1}^{N_1}\frac{\prod_{A=1}^{N_0}(\lambda_{f_I}-\eta_A)_{d_I}}{d_I!\prod_{J=1}^{N_1'}(\lambda_{f_I}-\lambda_{f_J'}+1)_{d_I}}\,,
\end{equation}
which is exactly $I^{Gr, S^2}_d$.

On the other hand, the poles of $f(\vec \phi)$ in the left half plane are 
\begin{align}\label{eqn:polesl}
   \{\phi_\alpha\}=\{\lambda_{f_I'}-n_I, \text{ for } I=1, \ldots, N_1', n_I=1,\ldots,d_I\}\,,
\end{align}
for all possible partition  $d=(d_1,\ldots, d_{N_1'})$ of $d$.
By mimicking the above procedure, one can prove that the left hand side of \eqref{eqn:prf1} equals $I^{Gr^\vee,S^2}_d$. Hence we have proved this lemma.
\end{proof}

\begin{cor}
When $N_2=N_0+1$,
\begin{equation}\label{eqn:N2=N01}
    I_d^{Gr,S^2}=\sum_{p=0}^d\frac{(-1)^{N_1'(d-p)}}{(d-p)!}I_p^{Gr^\vee,S^2}\,,
\end{equation}
 so we conclude
\begin{equation}
    I^{Gr,S^2}(q)=e^{(-1)^{N_1'}q}I^{Gr^\vee,S^2}(q^{-1})\,.
\end{equation}
\end{cor}
\begin{proof}
Denote $N_1'=N_2-N_1$. 
Consider instead the $I$-functions of $S_1^{\oplus N_0}\rightarrow Gr(N_1,N_2+1)$ and $(S_1^\vee)^{\oplus N_0}\rightarrow Gr(N_1'+1, N_2+1)$. 
By the above lemma, their $I$-functions are equal. Consider the pair of torus fixed points $Q$, $P=\iota (Q)$ such that the index $N_2+1$ is in $P$. That means $Q=(\{f_1,\ldots, f_{[N_1]}\}\subseteq [N_2+1])$ and $P=(\{f_1',\ldots, f_{[N_1']},N_2+1\})$. 
Denote $x:=\lambda_{N_2+1}$. 
The $I$-functions of $S_1^{\oplus N_0}\rightarrow Gr(N_1,N_2+1)$ and $(S_1^\vee)^{\oplus N_0}\rightarrow Gr(N_1'+1, N_2+1)$ satisfy the Equation \eqref{eqnA:lemma1}, and it can be written as follows, 
\begin{align}\label{eqnA:N2N01}
    &\sum_{\abs{\vec d}=d}(-1)^{(N_1-1+N_0)d} \prod_{I\neq J=1}^{N_1}
    \frac{1}{(\lambda_{f_J}-\lambda_{f_I}-d_I)_{d_J}}\prod_{I=1}^{N_1}\frac{\prod_{A=1}^{N_0}(\lambda_{f_I}-\eta_A)_{d_I}}{d_I!\prod_{J=1}^{N_1'}(\lambda_{f_I}-\lambda_{f_J'}+1)_{d_I}}\cdot
    \frac{1}{\prod_{I=1}^{N_1}(\lambda_{f_I}-x+1)_{d_I}}
    \nonumber\\
    =&\sum_{\abs{\vec n}+p=d}(-1)^{(N_1'-1)\abs{\vec n}} \prod_{I\neq J=1}^{N_1'}
    \frac{1}{(\lambda_{f_J'}-\lambda_{f_I'}-n_J)_{n_I}}\prod_{I=1}^{N_1'}\frac{\prod_{A=1}^{N_0}(-\lambda_{f_I'}+\eta_A+1)_{n_I}}{n_I!\prod_{J=1}^{N_1}(\lambda_{f_J}-\lambda_{f_I'}+1)_{n_I}}\nonumber\\
    &\times(-1)^{N_1'p+d-p} \prod_{I=1}^{N_1'}\frac{1}{(x-\lambda_{f_I'}-p)_{n_I}(\lambda_{f_I'}-x-n_I)_{p}}\frac{\prod_{A=1}^{N_0}(-x+\eta_A+1)_{p}}{p!\prod_{J=1}^{N_1}(\lambda_{f_J}-x+1)_{p
    }}\,.
\end{align}
Multiplying both sides by $(-x)^d$ and 
taking limit $x\rightarrow \infty$, the left hand side is 
\begin{equation}
    I^{Gr,S^2}_d\cdot \lim_{\lambda_{N_2+1}\rightarrow \infty} \frac{(-x)^{d}}{\prod_{I=1}^{N_1}(\lambda_{f_I}-x+1)_{d_I}}=I^{Gr,S^2}_d\,,
\end{equation}
where $I^{Gr,S^2}_d$ is the degree $d$-term of the equivariant small $I$-function of $S_1^{\oplus N_0}\rightarrow Gr(N_1,N_2)$ being restricted to $Q=(\{f_1,\ldots, f_{[N_1]}\}\subseteq [N_2])$.
Taking limit $x\rightarrow \infty$ to the right hand side of the Equation \eqref{eqnA:N2N01}, we get 
\begin{align}
    &\lim_{x\rightarrow \infty}\sum_{p+\abs{\vec n}=d} I^{Gr^\vee, S^2}_{\abs{\vec n}}\nonumber\\
    &\times \frac{(-1)^{N_1'p+d-p}(-x)^{d}}{\prod_{I=1}^{N_1'}(\lambda_{f_I'}-x-n_I)_{p}(x-\lambda_{f_I'}-p)_{n_{I}}}
    \frac{\prod_{A=1}^{N_0}(-\lambda_{N_2+1}+\eta_A+1)_{p}}{p!\prod_{J=1}^{N_1}(\lambda_{f_J}-x+1)_{p}}\nonumber\\
    &=
    \sum_{p=0}^{d}\frac{(-1)^{(N_1'p)}}{(p)!}I_{d-p}^{Gr^\vee,S^2}\,. 
\end{align}
Therefore we have proved the lemma. 
\end{proof}
The situation when $N_0=N_2$ is more involved. In the following,  $I^{Gr,S^2}_d$ will always represent the degree $d$-term of the $I$-function of $S_1^{\oplus N_0}\rightarrow Gr(N_1,N_2)$ and $I^{Gr^\vee,S^2}_d$ will represent the $I$-function of $(S_1^\vee)^{N_0}\rightarrow Gr(N_1',N_2)$ with $N_1'=N_2-N_1$.
\begin{lem}
When $N_0=N_2$,
\begin{align}
    I_d^{Gr,S^2}=\sum_{m=0}^d(-1)^{N_1'm}
   \begin{pmatrix}
   \sum_{A=1}^{N_0}\eta_A-\sum_{F=1}^{N_2}\lambda_F+N_1'\\
   m
   \end{pmatrix}
   I^{Gr^\vee,S^2}_{d-m}\,,
\end{align}
where 
\begin{equation}
    \begin{pmatrix}
   \sum_{A=1}^{N_0}\eta_A-\sum_{F=1}^{N_2}\lambda_F+N_1'\\
   m
   \end{pmatrix}=\frac{\prod_{l=0}^{m-1} \sum_{A=1}^{N_0}\eta_A-\sum_{F=1}^{N_2}\lambda_F+N_1'-l }{m!}\,.
\end{equation}
\end{lem}
\begin{proof}
We consider the $I$-functions of $S_1^{\oplus N_0}\rightarrow Gr(N_1,N_2+1)$ and $(S_1^\vee)^{\oplus N_0}\rightarrow Gr(N_1'+1,N_2+1)$ and they satisfy the relation  \eqref{eqn:N2=N01}. Again, we consider the restriction of the two $I$-functions to torus fixed points $Q$ and $P=\iota(Q)$ such that the index $N_2+1$ is in $P$. Denote $x:=\lambda_{N_2+1}$, and $y:=x^{-1}$.
When $y$ is small enough,  the  degree $d$-term of the $I$-function of $S_1^{\oplus N_0}\rightarrow Gr(N_1,N_2+1)$ can be expanded as power series of $y$ as follows,
\begin{align}\label{eqn:A8left}
    &\sum_{\abs{\vec d}=d}(-1)^{(N_1-1+N_0)\abs{\vec d}} \prod_{I\neq J=1}^{N_1}
    \frac{1}{(\lambda_{f_J}-\lambda_{f_I}-d_I)_{d_J}}\prod_{I=1}^{N_1}\frac{\prod_{A=1}^{N_0}(\lambda_{f_I}-\eta_A)_{d_I}}{d_I!\prod_{F\in \vec C^c}(\lambda_{f_I}-\lambda_F+1)_{d_I}}(-y)^d+O(y^{d+1})\nonumber\\
    &=I_d^{Gr,S^2}(-y)^d+O(y^{d+1})\,.
\end{align}
By splitting the factors involving $x$, the right hand side of the Equation \eqref{eqn:N2=N01} can be written as  
\begin{align}\label{eqn:A6}
   &\sum_{\abs{\vec n}+m=d}
   (-1)^{(N_1'-1)\abs{\vec n}}\prod_{I=1}^{N_1'}
   \frac{\prod_{A=1}^{N_0}(\lambda_{f'_I}-\eta_A+1)_{n_{I}}}{n_I!\prod_{J\neq I}(\lambda_{f_I'}-\lambda_{f_J'}-n_I)_{n_J}\prod_{F\in \vec C}(\lambda_F-\lambda_{f_I'}+1)_{n_I}}\nonumber\\
   &\times\frac{ (-1)^{d+N_1'm}}{m!}\sum_{q=0}^m(-1)^{q}C_m^q
   \prod_{I=1}^{N_1'}\frac{x-\lambda_{f_I'}+n_I-q}{(x-\lambda_{f_I'})_{n_I+1}}
   \frac{\prod_{A=1}^{N_0}(-x+\eta_A+1)_q}{\prod_{F=1}^{N_2}(\lambda_F-x+1)_q}\,,
\end{align}
where we have applied the formula \eqref{eqn:A7}.
Changing variable $x$ to $y^{-1}$, the second row of the formula \eqref{eqn:A6} appears as,
\begin{align}\label{eqn:A8right}
    \frac{1}{m!}\cdot 
   \frac{(-1)^{d+N_1'm}y^{d-m}Q_{m}(y, -\lambda_{f_I'}+n_I, \lambda_F, \eta_A)}{\prod_{I=1}^{N_1'}\prod_{k=0}^{n_I}\left(1-(\lambda_{f_I'}-k)y\right)\prod_{F=1}^{N_2}\prod_{k=1}^m\left(1-(\lambda_F+k)y\right)}
\end{align}
where 
\begin{align}
    Q_m(y, -\lambda_{f_I'}+n_I, \lambda_F, \eta_A)&=
    \sum_{q=0}^m(-1)^{q}C_m^q\prod_{I=1}^{N_1'}\left(1+(-\lambda_{f_I'}+n_I-q)y\right)\nonumber\\
    &\times
    \prod_{A=1}^{N_0}\prod_{k=1}^q(1-(\eta_A+k)y)\prod_{F=1}^{N_2}\prod_{k=q+1}^m\left(1-(\lambda_F+k)y\right)\,.
\end{align}
By \cite[Appendix B]{benini2015cluster}, 
\begin{align}\label{eqn:A8right2}
    Q_m(y, -\lambda_{f_I'}+n_I, \lambda_F, \eta_A)=\prod_{k=1}^{m}\left(-\sum_{F=1}^{N_2}\lambda_F+\sum_{A=1}^{N_0}\eta_A+N_1'-k\right)y^{m}+O(y^{m+1})\,.
\end{align}
Hence, the formula \eqref{eqn:A6} expanded as power series of $y$ is
\begin{equation}\label{eqn:A8rightf}
\sum_{\abs{\vec n}+m=d}I^{Gr^\vee,S^2}_{\abs{\vec n}}\frac{(-1)^{N_1'm}}{m!} \prod_{k=1}^{m}\left(-\sum_{F=1}^{N_2}\lambda_F+\sum_{A=1}^{N_0}\eta_A+N_1'-k\right)(-y)^d+O(y^{d+1})\,.
\end{equation}
Comparing coefficients of $(-y)^d$ of \eqref{eqn:A8left} and \eqref{eqn:A8rightf}, we get the relation below, 
\begin{align}\label{eqn:A9}
    I_d^{Gr,S^2}&=\sum_{\abs{\vec n}+m=d}(-1)^{N_1'm}
   \begin{pmatrix}
   \sum_{A=1}^{N_0}-\sum_{F=1}^{N_2}\lambda_F+N_1'\\
   m
   \end{pmatrix}I^{Gr^\vee,S^2}_{\abs{\vec n}}\,,
\end{align}
where we have used the formal expression 
\begin{equation}
    \begin{pmatrix}
   \sum_{A=1}^{N_0}-\sum_{F=1}^{N_2}\lambda_F+N_2-N_1\\
   m
   \end{pmatrix}
\end{equation}
to represent 
\begin{equation}
    \frac{\prod_{l=0}^{m-1}\left(\sum_{A=1}^{N_0}\eta_A-\sum_{F=1}^{N_2}\lambda_F+N_2-N_1-l\right)}{m!}\,.
\end{equation}
Therefore, we derive the third formula in Theorem \ref{thm:Haidong}.
\end{proof}

 \bibliographystyle{amsalpha}

\bibliography{Anmianreference}

\end{document}